%% file: all.tex
\documentclass[12pt,reqno]{amsart}

\usepackage{a4wide}
\usepackage[unicode,final]{hyperref}
\usepackage[T1]{fontenc}    
\usepackage{enumitem}
\usepackage{graphicx}

\usepackage{amsmath}
\usepackage{amssymb}
\usepackage{amsxtra}
\usepackage{amsthm}

\input{imports}

\setlist{}
\setenumerate{leftmargin=*,labelindent=\parindent}
\setitemize{leftmargin=*,labelindent=0.5\parindent}
\setdescription{leftmargin=1em}

\swapnumbers

\theoremstyle{plain}

\newtheorem{thm}{Theorem}[subsection]
\newtheorem{lem}[thm]{Lemma}
\newtheorem{cor}[thm]{Corollary}

\newtheorem{prop}[thm]{Proposition}

\theoremstyle{definition}
\newtheorem{defn}[thm]{Definition}
\newtheorem{ex}[thm]{Example}
\newtheorem{ntn}[thm]{Notation}

\theoremstyle{remark}
\newtheorem{obs}[thm]{Observation}
\newtheorem{rec}[thm]{Recall}
\newtheorem{rmk}[thm]{Remark}

\makeatletter
\let\c@equation\c@thm
\makeatother
\numberwithin{equation}{subsection}

\title{The comprehension construction}

\author[Riehl]{Emily Riehl}
\address{
  Department of Mathematics \\
Johns Hopkins University \\
Baltimore, MD 21218\\
  USA
}
\email{eriehl@math.jhu.edu}

\author[Verity]{Dominic Verity}
\address{
  Centre of Australian Category Theory \\
  Macquarie University \\
  NSW 2109 \\
  Australia
}
\email{dominic.verity@mq.edu.au}

\date{\today}


\dedicatory{The authors would like to dedicate this work to the memory of Brian Day (1945-2012) on the 5\textsuperscript{th} anniversary of his death. Brian was an excellent mathematician, a valued friend, and a key founding member of the category theory community in Australia.}

\begin{document}

  \ifpdf
  \DeclareGraphicsExtensions{.pdf, .jpg, .tif}
  \else
  \DeclareGraphicsExtensions{.eps, .jpg}
  \fi

  \begin{abstract}
    In this paper we construct an analogue of Lurie's ``unstraightening'' construction that we refer to as the \emph{comprehension construction}. Its input is a cocartesian fibration $p \colon E \tfib B$ between $\infty$-categories together with a third $\infty$-category $A$. The comprehension construction then defines a map from the quasi-category of functors from $A$ to $B$ to the large quasi-category of cocartesian fibrations over $A$ that acts on $f \colon A \to B$ by forming the pullback of $p$ along $f$. To illustrate the versatility of this construction, we define the covariant and contravariant Yoneda embeddings as special cases of the comprehension functor. We then prove that the hom-wise action of the comprehension functor coincides with an ``external action'' of the hom-spaces of $B$ on the fibres of $p$ and use this to prove that the Yoneda embedding is fully faithful, providing an explicit equivalence between a quasi-category and the homotopy coherent nerve of a Kan-complex enriched category.
  \end{abstract}

  \maketitle
  \tableofcontents

 \input{intro}
\input{background}

\input{cocartesian}
 \input{computads}

 \input{cocones}
\input{comprehend}
 \input{homs}

  \bibliographystyle{special}
  \bibliography{../common/index}


\end{document}

%% file: imports.tex
\usepackage[silent]{fixme}
\fxsetup{
  status=final
}

\input{macros}

\usepackage{mathtools}

\makeatletter
\newcommand{\qc}[1]{{\mathord{\text{\normalfont{\textsf{#1}}}}}}
\newcommand{\qop}[1]{{\mathord{\qc{#1}}}}
\def\ec@#1#2<.>{\mathcal{#1}\mkern-2mu\text{\normalfont{\textsf{\slshape #2}}}}
\newcommand{\ec}[1]{{\mathord{\ec@#1<.>}}}
\newcommand{\eop}[1]{{\mathord{\ec{#1}}}}
\makeatother

\newcommand{\qA}{\qc{A}}
\newcommand{\qB}{\qc{B}}
\newcommand{\qC}{\qc{C}}

\newcommand{\qE}{\qc{E}}
\newcommand{\qF}{\qc{F}}

\newcommand{\qK}{\qc{K}}
\newcommand{\qL}{\qc{L}}
\newcommand{\qS}{\qc{S}}
\newcommand{\qQ}{\qc{Q}}
\newcommand{\qX}{\qc{X}}

\newcommand{\qU}{\qc{U}}

\newcommand{\eK}{\ec{K}}
\newcommand{\eL}{\ec{L}}
\newcommand{\eM}{\ec{M}}
\newcommand{\eA}{\ec{A}}
\newcommand{\eB}{\ec{B}}
\newcommand{\eC}{\ec{C}}

\newcommand{\Hom}{\qop{Hom}}
\newcommand{\Fun}{\qop{Fun}}

\renewcommand{\coCart}{\qop{coCart}}
\newcommand{\Cart}{\qop{Cart}}

\newcommand{\Graph}{\category{Graph}}

\newcommand{\sCptd}{\sSet\text{-}\category{Cptd}}
\newcommand{\Cptd}{\category{Cptd}}

\newcommand{\SSet}{\eop{SSet}}
\renewcommand{\qCat}{\eop{QCat}}
\newcommand{\qqCat}{\qop{qCat}}

\renewcommand{\Kan}{\eop{Kan}}

\newcommand{\hcSimp}{\gC\Del^\infty}
\newcommand{\oSimp}{\scat{S}}

\newcommand{\ep}{_{\mathrm{epi}}}

\newcommand{\Cube}{\mathord{\sqcap\llapm\sqcup}}
\newcommand{\CHorn}{\mathord{\sqcap\llapm\sqcap}}

\newmop{core}
\newmop{Fam}
\newmop{last}
\newmop{rv}
\newmop{ob}
\newmop{free}

\makeatletter
\def\hookrightarrowfill@{\arrowfill@{\lhook\mkern-2mu\relbar}\relbar\rightarrow}
\newcommand{\xrightincarrow}[2][]{\ext@arrow 0359\hookrightarrowfill@{#1}{#2}}
\makeatother

\usepackage{xr}

\newcommand{\extRef}[3]{%
  {\protect\IfBeginWith{#3}{itm:}{}{#2.}}\ref*{#1:#3}}

\newcommand{\refI}{\extRef{found}{I}}
\newcommand{\refII}{\extRef{cohadj}{II}}
\newcommand{\refIV}{\extRef{yoneda}{IV}}
\newcommand{\refV}{\extRef{equipment}{V}}

%% file: macros.tex
\usepackage[mathscr]{eucal} 
\usepackage{mathrsfs}       
\usepackage{slashed}        

\usepackage{mathtools}
\usepackage{extpfeil}
\usepackage{proof}
\usepackage{ifnextok}
\usepackage{xstring}

\usepackage[all]{xy}
\SelectTips{cm}{}
\SilentMatrices
\ifx\pdfoutput\undefined
  \xyoption{dvips}
\fi
\usepackage{tikz-cd}

\newcommand{\ifundef}[1]{\expandafter\ifx\csname#1\endcsname\relax}

\DeclareMathAlphabet{\mathbbe}{U}{bbold}{m}{n}

\makeatletter

\def\re@DeclareMathSymbol#1#2#3#4{%
    \let#1=\undefined
    \DeclareMathSymbol{#1}{#2}{#3}{#4}}

\ifundef{Top}
  \DeclareSymbolFont{tcSyC}{U}{txsyc}{m}{n}
  \SetSymbolFont{tcSyC}{bold}{U}{txsyc}{bx}{n}
  \DeclareFontSubstitution{U}{txsyc}{m}{n}

  \re@DeclareMathSymbol{\Top}{\mathord}{tcSyC}{120}
  \re@DeclareMathSymbol{\Bot}{\mathord}{tcSyC}{121}
\fi

\ifundef{righthalfcup}
  \DeclareFontFamily{U}{MnSymbolC}{}
  \DeclareSymbolFont{mnSyC}{U}{MnSymbolC}{m}{n}
  \SetSymbolFont{mnSyC}{bold}{U}{MnSymbolC}{b}{n}
  \DeclareFontShape{U}{MnSymbolC}{m}{n}{
      <-6>  MnSymbolC5
     <6-7>  MnSymbolC6
     <7-8>  MnSymbolC7
     <8-9>  MnSymbolC8
     <9-10> MnSymbolC9
    <10-12> MnSymbolC10
    <12->   MnSymbolC12}{}
  \DeclareFontShape{U}{MnSymbolC}{b}{n}{
      <-6>  MnSymbolC-Bold5
     <6-7>  MnSymbolC-Bold6
     <7-8>  MnSymbolC-Bold7
     <8-9>  MnSymbolC-Bold8
     <9-10> MnSymbolC-Bold9
    <10-12> MnSymbolC-Bold10
    <12->   MnSymbolC-Bold12}{}

  \re@DeclareMathSymbol{\righthalfcup}{\mathord}{mnSyC}{184}
  \re@DeclareMathSymbol{\lefthalfcap}{\mathord}{mnSyC}{185}
\fi

\DeclareFontFamily{U}{MnSymbolA}{}
\DeclareSymbolFont{mnSyA}{U}{MnSymbolA}{m}{n}
\SetSymbolFont{mnSyA}{bold}{U}{MnSymbolA}{b}{n}
\DeclareFontShape{U}{MnSymbolA}{m}{n}{
    <-6>  MnSymbolA5
   <6-7>  MnSymbolA6
   <7-8>  MnSymbolA7
   <8-9>  MnSymbolA8
   <9-10> MnSymbolA9
  <10-12> MnSymbolA10
  <12->   MnSymbolA12}{}
\DeclareFontShape{U}{MnSymbolA}{b}{n}{
    <-6>  MnSymbolA-Bold5
   <6-7>  MnSymbolA-Bold6
   <7-8>  MnSymbolA-Bold7
   <8-9>  MnSymbolA-Bold8
   <9-10> MnSymbolA-Bold9
  <10-12> MnSymbolA-Bold10
  <12->   MnSymbolA-Bold12}{}

\re@DeclareMathSymbol{\twoheadedswarrow}{\mathord}{mnSyA}{30}

\makeatother

\newcommand{\mlaux}[3]{\setbox0=\hbox{$\mathsurround=0pt #2{#3}$}%
  \dimen0=\dp0\advance\dimen0 by \ht0\lower#1\dimen0\box0}

\newcommand{\makellapm}[2]{\hbox to 0pt{\hss$\mathsurround=0pt #1{#2}$}}
\newcommand{\llapm}{\relax\mathpalette\makellapm}
\newcommand{\makerlapm}[2]{\hbox to 0pt{$\mathsurround=0pt #1{#2}$\hss}}

\newcommand{\makelapm}[2]{\hbox to 0pt{\hss$\mathsurround=0pt #1{#2}$\hss}}

\newcommand{\makeushort}[3]{%
  \setbox0=\hbox{$\mathsurround=0pt #2{#3}$}%
  \hbox to 1\wd0{\hss\underbar{\hbox to #1\wd0{\hss\box0\hss}}\hss}}

\def\makebigger#1#2#3{\scalebox{#1}{$\mathsurround=0pt #2{#3}$}}
\def\bigger#1#2{{\relax\mathpalette{\makebigger{#1}}{#2}}}

\def\scaleuphalf{1.0954}
\def\scaleupone{1.2}

\newcommand{\dual}{^\circ}

\newcommand{\op}{^{\mathord{\text{\rm op}}}}
\newcommand{\co}{^{\mathord{\text{\rm co}}}}

\newcommand{\defeq}{\mathrel{:=}}

\makeatletter

\def\newmop{\@ifstar{\@newmop m}{\@newmop o}}
\def\@newmop#1{\@ifnextchar[{\@@newmop #1}{\@@@newmop #1}}
\def\@@newmop#1[#2]{\@declmathop #1#2}
\def\@@@newmop#1#2{\expandafter\@declmathop\expandafter #1\csname #2\endcsname{#2}}

\makeatother

\newdir{ |}{{}*!/-5pt/@{|}}

\newmop{im}
\newmop{coim}
\newmop{dom}
\newmop{cod}
\newmop{id}
\newmop{Map}

\newmop{obj}
\newmop{arr}
\newmop{sq}
\newmop{norm}

\newmop{el}

\newmop{ev}

\newmop{fun}

\newmop{Ext}
\newmop{icon}
\newmop{pbk}
\newmop{icom}

\newmop{cell}
\newmop{cof}
\newmop{fib}

\newmop{diag}

\newmop*{colim}
\newmop*{holim}

\newmop[\lan]{lan}
\newmop[\ran]{ran}

\newmop*{cone}
\newmop*{cocone}
\newmop*{base}

\newcommand{\comma}{\mathbin{\downarrow}}

\makeatletter
\newcommand{\rotatemath}[2]{\rotatebox[origin=c]{180}{$\m@th #1{#2}$}}
\makeatother

\newcommand{\yoneda}{\mathscr{Y}\!}

\newmop[\pr]{pr}
\newmop[\dm]{d}

\newmop[\cpr]{in}

\newcommand{\pocorner}{\hbox to 8pt{{\vrule height8pt depth0pt width0.5pt}%
    \vbox to 8pt{{\hrule height0.5pt width7.5pt depth0pt}\vfill}}}
\newcommand{\poexcursion}{\save[]-<15pt,-15pt>*{\pocorner}\restore}
\newcommand{\pbcorner}{\vbox to 0pt{\kern 4pt\hbox to 0pt{\kern 4pt%
      \vbox{{\hrule height0.5pt width7.5pt depth0pt}}%
      {\vrule height8pt depth0pt width0.5pt}\hss}\vss}}
\newcommand{\pbexcursion}{\save[]+<5pt,-5pt>*{\pbcorner}\restore}

\newcommand{\pwr}{\mathbin\pitchfork}

\newmop{cls}

\newcommand{\leib}[1]{\mathbin{\widehat{#1}}}

\newcommand{\category}[1]{\underline{\smash[b]{\text{\rm{#1}}}}}

\newcommand{\lcat}{\mathcal}
\newcommand{\scat}{\mathbf}

\newcommand{\catthree}{{\bigger{1.12}{\mathbbe{3}}}}
\newcommand{\cattwo}{{\bigger{1.12}{\mathbbe{2}}}}
\newcommand{\catone}{{\bigger{1.16}{\mathbbe{1}}}}
\newcommand{\iso}{{\mathbb{I}}}

\makeatletter

\def\Del@Sym{{\bigger\scaleuphalf{\mathbbe{\Delta}}}}

\def\del@fn{\futurelet\del@next}
\def\del@dn{\def\del@next}

\def\parsedel@{%
  \ifx +\del@next \del@dn+{\Del@Sym_{\mathord{+}}}%
  \else \del@dn {\del@fn\parsedel@@}%
  \fi\del@next}

\def\parsedel@@{%
  \ifx\space@\del@next \expandafter\del@dn\space{\del@fn\parsedel@@}%
  \else\ifx [\del@next \del@dn[{\del@fn\parsedel@@@}%
  \else\ifx _\del@next \del@dn{\Delta}%
  \else\ifx ^\del@next \del@dn{\Delta}%
  \else \del@dn{\Del@Sym}%
  \fi\fi\fi\fi\del@next}

\def\parsedel@@@{%
  \ifx\space@\del@next \expandafter\del@dn\space{\del@fn\parsedel@@@}%
  \else\ifx t\del@next \del@dn t{\Del@Sym_\infty\del@fn\parsedel@@@@}%
  \else\ifx b\del@next \del@dn b{\Del@Sym_{-\infty}\del@fn\parsedel@@@@}%
  \else \del@dn{\errmessage{unexpected modifier}}%
  \fi\fi\fi\del@next}

\def\parsedel@@@@{%
  \ifx\space@\del@next \expandafter\del@dn\space{\del@fn\parsedel@@@@}%
  \else\ifx ]\del@next \del@dn]{}%
  \else \del@dn{\errmessage{expecting close of option block}}%
  \fi\fi\del@next}

\def\Del{\del@fn\parsedel@}

\makeatother

\newcommand{\Horn}{\Lambda}

\newcommand{\Set}{\category{Set}}
\newcommand{\Cat}{\category{Cat}}

\newcommand{\sCat}{\sSet\text{-}\Cat}
\newcommand{\sSet}{\category{sSet}}
\newcommand{\qCat}{\category{qCat}}
\newcommand{\Segal}{\category{Segal}}
\newcommand{\CSS}{\category{CSS}}

\newcommand{\coCart}{\category{coCart}}

\newcommand{\dmod}[3]{\xymatrix@=1.25em{{#2} \ar[r]|\mid^{ {#1}} & {#3}}}

\newcommand{\pbshape}{{\mathord{\bigger\scaleupone\righthalfcup}}}

\newcommand{\face}{\delta}

\newcommand{\degen}{\sigma}

\newcommand{\fbv}[1]{\{{#1}\}}

\newcommand{\join}{\mathbin\star}

\newmop{dec}
\newmop{fatdec}
\newmop{slc}
\newmop{fatslc}

\newcommand{\slice}{/}
\newcommand{\slicel}[2]{\vphantom{#2}^{{#1}\slice{}}\mkern-2mu{#2}}

\newmop{ir} 
\newmop{incl}

\newcommand{\nrv}{N}
\newcommand{\ho}{h}

\newcommand{\nrvhc}{\nrv}
\newcommand{\hN}{\nrv}
\newcommand{\gC}{\mathfrak{C}}

\newcommand{\boundary}{\partial}

\def\reedyfilt#1_#2{#1_{\leq #2}}
\newmop{sk}
\newmop{cosk}
\newmop{res}

\newmop{map}
\newmop{Ho}

\newmop[\pth]{path}
\newmop{cyl}

\newmop[\const]{c}

\newcommand{\Kan}{\category{Kan}}

\newmop{coll}
\newmop{wgt}

\newdir{ >}{{}*!/-7pt/@{>}}
\newdir{u(}{{}*!/-4pt/@^{(}}
\newdir{d(}{{}*!/-4pt/@_{(}}
\newdir{|>}{%
  !/4.5pt/@{|}*:(1,-.2)@^{>}*:(1,+.2)@_{>}*+@{}}

\makeatletter
\def\makeslashed#1#2#3#4#5{#1{\mathpalette{\sla@{#2}{#3}{#4}}{#5}}}

\def\@mathlower#1#2#3{\setbox0=\hbox{$\m@th#2#3$}\lower#1\ht0\box0}
\def\mathlower#1#2{\mathpalette{\@mathlower{#1}}{#2}}
\makeatother

\newcommand{\inc}{\hookrightarrow}

\newcommand{\tfib}{\twoheadrightarrow}

\newcommand{\longtwoheadrightarrow}{\mathrel{\mathord{-}\mkern-3mu\mathord\twoheadrightarrow}}

\newcommand{\we}{\xrightarrow{\mkern10mu{\smash{\mathlower{0.6}{\sim}}}\mkern10mu}}

\newcommand{\trvfib}{\stackrel{\smash{\mkern-2mu\mathlower{1.5}{\sim}}}\longtwoheadrightarrow}
\newcommand{\trvcof}{\xhookrightarrow{\mkern8mu{\smash{\mathlower{1}{\sim}}}\mkern12mu}}

\newcommand{\To}{\Rightarrow}

\makeatletter

\def\tens@fn{\futurelet\tens@next}
\def\tens@dn{\def\tens@nextcont}
\newtoks\tens@toks
\def\addtotens@toks#1{\tens@toks=\expandafter{\the\tens@toks#1}}

\def\parsetens@@{%
    \ifx\space@\tens@next \expandafter\tens@dn\space{\tens@fn\parsetens@@}%
    \else\ifx ^\tens@next \tens@dn ^##1{\parsetens@procsep^\addtotens@toks{##1}%
      \tens@fn\parsetens@@}%
    \else\ifx _\tens@next \tens@dn _##1{\parsetens@procsep_\addtotens@toks{##1}%
      \tens@fn\parsetens@@}%
    \else\tens@dn{\ifx *\tens@last \else\addtotens@toks\egroup\fi\the\tens@toks}%
    \fi\fi\fi\tens@nextcont}

\def\parsetens@procsep#1{%
  \ifx *\tens@last \addtotens@toks{#1}\addtotens@toks\bgroup%
  \else\ifx \tens@last\tens@next \addtotens@toks,%
  \else \addtotens@toks\egroup\addtotens@toks\bgroup%
    \addtotens@toks\egroup\addtotens@toks{#1}\addtotens@toks\bgroup%
  \fi\fi\let\tens@last\tens@next}

\newcommand{\tn}[1]{\let\tens@last=*\tens@toks={#1}\tens@fn\parsetens@@}

\makeatother

\def\adjdisplay#1-|#2:#3->#4.{{%
    \xymatrix@R=0em@!C=2.5em{%
      *+[l]{#3} \ar@/_0.55pc/[rr]_-{#2} & {\bot} &
      *+[r]{#4}\ar@/_0.55pc/[ll]_-{#1}}}}

\def\adjdisplaytwo#1-|#2:#3->#4.{{%
\xymatrix@=1.2em{
      {#3}\ar@/_1.5ex/[rr]_-{#2}^-{}="one"
      & & {#4}
      \ar@/_1.5ex/[ll]_-{#1}^-{}="two"
      \ar@{}"one";"two"|{\bot}
    }}}

\def\tripleadjdisplay#1-|#2-|#3:#4->#5.{{%
\xymatrix@=2.4em{
{#4}\ar[r]|{#2} &
{#5} \ar@/_3ex/[l]_{#1}^{\bot} \ar@/^3ex/[l]_{\bot}^{#3}}
}}

\def\adjinline#1-|#2:#3->#4.{{#1}\dashv{#2}:#3\to #4}

\newcommand{\pent}[1]{
  \xybox{
    \POS (0,-15)*+{\a}="0",
         (-14,-5)*+{\b}="1",
         (-9,12)*+{\c}="2",
         (9,12)*+{\d}="3",
         (14,-5)*+{\e}="4"
    \POS"0" \ar "1"^{\labelstyle \ab}|{}="01"
    \POS"1" \ar "2"^{\labelstyle \bc}|{}="12"
    \POS"2" \ar "3"^{\labelstyle \cd}|{}="23"
    \POS"3" \ar "4"^{\labelstyle \de}|{}="34"
    \POS"0" \ar "4"_{\labelstyle \ae}|{}="04"
    \ifcase #1
    \POS"0" \ar "2"|{\labelstyle \ac}="02"
    \POS"0" \ar "3"|{\labelstyle \ad}="03"
    \POS"02";"1"**{}, ?(0.3) \ar@{=>} ?(0.7)^{\labelstyle \abc}
    \POS"03";"2"**{}, ?(0.25) \ar@{=>} ?(0.5)_{\labelstyle \acd}
    \POS"04";"3"**{}, ?(0.2) \ar@{=>} ?(0.4)_{\labelstyle \ade}
    \or
    \POS"1" \ar "3"|{\labelstyle \bd}="13"
    \POS"1" \ar "4"|{\labelstyle \be}="14"
    \POS"13";"2"**{}, ?(0.3) \ar@{=>} ?(0.7)_{\labelstyle \bcd}
    \POS"14";"3"**{}, ?(0.25) \ar@{=>} ?(0.5)_{\labelstyle \bde}
    \POS"04";"1"**{}, ?(0.25) \ar@{=>} ?(0.5)_{\labelstyle \abe}
    \or
    \POS"2" \ar "4"|{\labelstyle \ce}="24"
    \POS"0" \ar "2"|{\labelstyle \ac}="02"
    \POS"02";"1"**{}, ?(0.3) \ar@{=>} ?(0.7)^{\labelstyle \abc}
    \POS"04";"2"**{}, ?(0.2) \ar@{=>} ?(0.35)_{\labelstyle \ace}
    \POS"24";"3"**{}, ?(0.2) \ar@{=>} ?(0.6)^{\labelstyle \cde}
    \or
    \POS"1" \ar "3"|{\labelstyle \bd}="13"
    \POS"0" \ar "3"|{\labelstyle \ad}="03"
    \POS"04";"3"**{}, ?(0.2) \ar@{=>} ?(0.4)_{\labelstyle \ade}
    \POS"13";"2"**{}, ?(0.3) \ar@{=>} ?(0.7)_{\labelstyle \bcd}
    \POS"03";"1"**{}, ?(0.25) \ar@{=>} ?(0.5)^{\labelstyle \abd}
    \or
    \POS"2" \ar "4"|{\labelstyle \ce}="24"
    \POS"1" \ar "4"|{\labelstyle \be}="14"
    \POS"24";"3"**{}, ?(0.2) \ar@{=>} ?(0.6)^{\labelstyle \cde}
    \POS"04";"1"**{}, ?(0.25) \ar@{=>} ?(0.5)_{\labelstyle \abe}
    \POS"14";"2"**{}, ?(0.25) \ar@{=>} ?(0.5)^{\labelstyle \bce}
    \else\fi
  }
}

\newcommand{\pentofpent}[1]{
  \def\baselen{#1}
  \begin{xy}
    0;<\baselen,0mm>:
    *{\xybox{
        \POS(0,-4)*[o]{\pent 0}="zero"
        \POS(16,40)*[o]{\pent 3}="three"
        \POS(72,40)*[o]{\pent 1}="one"
        \POS(88,-4)*[o]{\pent 4}="four"
        \POS(44,-36)*[o]{\pent 2}="two"
        \ar@<1ex>"zero";"three"^-{\objectstyle\abcd}
        \ar@<1ex>"three";"one"^-{\objectstyle\abde}
        \ar@<1ex>"one";"four"^-{\objectstyle\bcde}
        \ar@<-1ex>"zero";"two"_-{\objectstyle\acde}
        \ar@<-1ex>"two";"four"_-{\objectstyle\abce}
        \ar@{=>}(44,-5);(44,+15)^{\objectstyle\abcde}
     }}
  \end{xy}
}


%% file: intro.tex

\section{Introduction}

This paper is a continuation of previous work \cite{RiehlVerity:2012tt, RiehlVerity:2012hc, RiehlVerity:2013cp,RiehlVerity:2015fy,RiehlVerity:2015ke} to develop the formal theory of $\infty$-\emph{categories}, which model weak higher categories. In contrast with the pioneering work of Joyal \cite{Joyal:2008tq} and Lurie \cite{Lurie:2009fk,Lurie:2012uq}, our approach is ``synthetic'' in the sense that our proofs do not depend on what precisely these $\infty$-categories \emph{are}, but rather rely upon an axiomatisation of the universe in which they \emph{live}. To that end, we define an $\infty$-\emph{cosmos} to be a (large) simplicial category  $\eK$ satisfying certain axioms. The objects of an $\infty$-cosmos are called $\infty$-\emph{categories}. A theorem, e.g., that characterises a cartesian fibration of $\infty$-categories in terms of the presence of an adjunction between certain other $\infty$-categories \cite[4.1.10]{RiehlVerity:2015fy}, is a result about the objects of any $\infty$-cosmos, and thus applies of course to every $\infty$-cosmos. There are $\infty$-cosmoi whose objects are quasi-categories, complete Segal spaces, or Segal categories, each of these being models of $(\infty,1)$-\emph{categories}; $\theta_n$-spaces or iterated complete Segal spaces, or $n$-trivial saturated complicial sets, each modelling $(\infty,n)$-\emph{categories}; and also fibred versions of each of these. Thus each of these objects are $\infty$-categories in our sense and our theorems apply to all of them.\footnote{This may seem like sorcery but in some sense it is really just the Yoneda lemma. To a close approximation, an $\infty$-cosmos is a ``category of fibrant objects enriched over quasi-categories,'' \emph{quasi-categories} being a model of $(\infty,1)$-categories as simplicial sets satisfying the weak Kan condition. When the theory of quasi-categories is expressed in a sufficiently categorical way, it generalises to include analogous results for the corresponding representably defined notions in a general $\infty$-cosmos.}  

One theme of our work to develop the foundations of $\infty$-category theory is that much of it can be done 2-categorically, closely paralleling ordinary or enriched category theory. Any $\infty$-cosmos has an accompanying \emph{homotopy 2-category}, a quotient strict 2-category whose objects are $\infty$-categories, whose morphisms are the $\infty$-functors between them, and whose 2-cells we appropriately refer to as \emph{natural transformations}.  The homotopy 2-category is a categorification of the familiar homotopy category of a model category.   Ordinary 1-categories define an $\infty$-cosmos whose homotopy 2-category $\eop{Cat}_2$ is the usual 2-category of categories, functors, and natural transformations. Much of basic category theory---e.g.,~adjunctions, limits, and cartesian fibrations---can be developed internally to the strict 2-category $\eop{Cat}_2$. In the homotopy 2-category of an $\infty$-cosmos, the standard 2-categorical notion of equivalence  precisely captures the usual  homotopy-theoretic notion of weak equivalence, suggesting that the homotopy 2-category is a reasonable context to develop a homotopically meaningful theory of $\infty$-categories. Indeed, our work shows that much of $\infty$-category theory can also be developed internally to  the homotopy 2-category of each specific model.  

Surprisingly, and non-obviously, Joyal's definition of \emph{limits} (inverse limits) or \emph{colimits} (direct limits) in a quasi-category and Lurie's notion of \emph{adjunctions} between quasi-categories---definitions that refer explicitly to the higher homotopical structure contained in a quasi-category---are precisely captured by internal 2-categorical definitions that we introduce in \cite{RiehlVerity:2012tt}, which instead make use of judiciously chosen universal properties inside the homotopy 2-category. Viewed from this perspective, they become easier to manipulate and proofs, e.g., that right adjoints preserve limits, mirror the standard categorical arguments.

This paper addresses what might be described as the  major objection to this narrative.  Perhaps the main technical challenge in extending classical categorical results to the $\infty$-categorical context is in merely \emph{defining} the  Yoneda embedding. In \S\ref{sec:comprehension}, we construct this functor as an instance of the versatile \emph{comprehension construction}. As we explain below, the comprehension construction  is closely related to both the unstraightening and straightening constructions of \cite{Lurie:2009fk}.

In an elementary topos, ``comprehension'' refers to the process that takes a proposition on an object $X$ and returns the maximal subobject on which that proposition is satisfied. Explicitly, a proposition on $X$ is encoded  by a morphism $X \to \Omega$ to the subobject classifier, and the subobject is constructed as the pullback of the subobject ``true'' $\top \colon 1 \rightarrowtail \Omega$ along this morphism. This construction carries additional higher-dimensional  structure because the subobject classifier has an ordering ``false implies true''\footnote{More precisely, $\Omega$ is an internal Heyting algebra.} which induces a corresponding ordering ``implication'' on the set $\hom(X,\Omega)$ of propositions on $X$. The comprehension construction is then a map
\[
    \xymatrix@R=0em@C=6em{
      {\hom(X,\Omega)}\ar[r]^{c_{\top,X}} & {\mathrm{Sub}(X)}
    }
    \]
 from the ordered set of propositions on $X$ to the set of subobjects of $X$, ordered by inclusion.

In higher categorical contexts, the role of the subobject classifier is replaced by a particular cocartesian fibration \cite{Weber:2007ys}. Here we will not require the use of any particular classifying fibration and instead describe an $\infty$-categorical comprehension construction that works for any cocartesian or cartesian fibration $p \colon E \tfib B$. Roughly, a functor $p \colon E \tfib B$ between $\infty$-categories is a \emph{cocartesian fibration} if its fibres depend covariantly functorially on its base; the precise definition is reviewed in \S\ref{sec:cocartesian}. The associated \emph{comprehension functor} provides one expression of this covariantly functorial dependence: each object $a \colon 1 \to B$ is mapped to the fibre $E_a$ of $p$ over $a$, each 1-simplex $f$ from $a$ to $b$ is mapped to a functor $E_a \to E_b$, and this construction extends to the higher simplices in the underlying quasi-category $\qB \defeq \Fun_{\eK}(1,B)$ of $B$.

{
\renewcommand{\thethm}{\ref{thm:general-comprehension}}
\begin{thm} For any cocartesian fibration $p \colon E \tfib B$ in an $\infty$-cosmos $\eK$ and any $\infty$-category $A$, there is a functor
\[
    \xymatrix@R=0em@C=6em{
      {\Fun_{\eK}(A,B)}\ar[r]^{c_{p,A}} & {\coCart(\qK)_{/A}}
    }
    \]
defined on 0-arrows by mapping a functor $a \colon A \to B$ to the pullback:
\[  \xymatrix{ E_a \ar@{->>}[d]_{p_a} \ar[r]^-{\ell_a} \pbexcursion & E \ar@{->>}[d]^p \\ A \ar[r]_a & B}\]
  Its action on 1-arrows $f\colon a\to b$ is defined by lifting $f$ to a $p$-cocartesian 1-arrow as displayed in the diagram
  \begin{equation*}
    \begin{xy}
      0;<1.4cm,0cm>:<0cm,0.75cm>::
      *{\xybox{
          \POS(1,0)*+{A}="one"
          \POS(0,1)*+{A}="two"
          \POS(3,0.5)*+{B}="three"
          \ar@{=} "one";"two"
          \ar@/_5pt/ "one";"three"_{b}^(0.1){}="otm"
          \ar@/^10pt/ "two";"three"^{a}_(0.5){}="ttm"|(0.325){\hole}
          \ar@{=>} "ttm"-<0pt,7pt> ; "otm"+<0pt,10pt> ^(0.3){f}
          \POS(1,2.5)*+{E_{b}}="one'"
          \POS(1,2.5)*{\pbcorner}
          \POS(0,3.5)*+{E_{a}}="two'"
          \POS(0,3.6)*{\pbcorner}
          \POS(3,3)*+{E}="three'"
          \ar@/_5pt/ "one'";"three'"_{\ell_{{b}}}^(0.1){}="otm'"
          \ar@/^10pt/ "two'";"three'"^{\ell_{{a}}}_(0.55){}="ttm'"
          \ar@{->>} "one'";"one"_(.3){p_b}
          \ar@{->>} "two'";"two"_{p_a}
          \ar@{->>} "three'";"three"^{p}
          \ar@{..>} "two'";"one'"_*!/^2pt/{\scriptstyle E_f}
          \ar@{=>} "ttm'"-<0pt,4pt> ; "otm'"+<0pt,4pt> ^(0.3){\ell_{{f}}}
        }}
    \end{xy}
  \end{equation*}
  and then factoring its codomain to obtain the requisite cartesian functor $E_f \colon E_a \to E_b$ between the fibres over $a$ and $b$. 
\end{thm}
\addtocounter{thm}{-1}
}
The \emph{comprehension functor} $c_{p,A}$ defines a functor from the quasi-category $\Fun_\eK(A,B)$ of functors from $A$ to $B$ to the large quasi-category of cocartesian fibrations over $A$, or equivalently by adjunction a simplicial functor between simplicially-enriched categories described in more detail below.  The action on objects of Lurie's ``straightening'' construction, which carries a cocartesian fibration between quasi-categories to a simplicial functor, is modeled by the comprehension construction given here, which converts a cocartesian fibration $p$ into the comprehension functor $c_{p,1}$; see Remark \ref{rmk:straightening}.

The comprehension \emph{functor} itself can be understood as an analogue of Lurie's ``unstraightening'' construction, though presented as a functor between (large) quasi-categories rather than as a right Quillen functor between suitable model categories. Modulo this contextual difference, Lurie's ``unstraightening'' construction is then the special case of the comprehension functor where $p$ is taken to be the classifying cocartesian fibration for quasi-categories. It is important to note that at the level of generality of Theorem \ref{thm:general-comprehension}, the comprehension functor is \emph{not} typically an equivalence: one would only expect the comprehension functor $c_{p,A}$ to define an equivalence when $p$ is taken to be a particular classifying cocartesian fibration. In the $\infty$-cosmos of quasi-categories, an appropriate classifying cocartesian fibration $u\colon\qQ_*\tfib \qQ$ is constructed in Remark \ref{rmk:unstraightening}, though we defer the proof that for any quasi-category $\qB$ the associated comprehension functor 
\[
    \xymatrix@R=0em@C=6em{
      {\qQ^\qB = \Fun_{\qqCat}(\qB,\qQ)}\ar[r]^-{c_{u,\qB}} &
      {\coCart(\qqCat)_{/\qB}}}\] defines an equivalence  to a future paper. In that paper, we show that any small cocartesian fibration of quasi-categories $p \colon \qE \tfib \qB$ fits into a canonical homotopy pullback diagram   
      \[
  \xymatrix{ \qE \ar@{->>}[d]_p \ar[r]  \pbexcursion & \qQ_* \ar@{->>}[d]^u \\ \qB \ar[r]_{c_{p,1}} & \qQ}
  \]
and we apply the $\infty$-categorical Beck monadicity theorem of \cite{RiehlVerity:2012hc} to prove that $c_{u,\qB}$ defines an equivalence, which then carries the ``straightened'' functor $c_{p,1}$ to the cocartesian fibration $p$. This result is the closest quasi-category level analog of Lurie's right Quillen equivalence \cite[3.2.0.1]{Lurie:2009fk}.
  
There is a particular reason for our interest in the general form of the comprehension construction given here, which permits an arbitrary cocartesian fibration as input and allows us the freedom to work in any ambient $\infty$-cosmos. Namely, a specific instance of this construction yields a quasi-categorical version of the \emph{Yoneda embedding}, which is notoriously difficult to construct in the $\infty$-categorical context. We define the covariant Yoneda embedding associated to an $\infty$-category $A$ as a restriction of  the comprehension functor associated to the cocartesian fibration $(p_1,p_0) \colon A^\cattwo \tfib A \times A$ in the sliced $\infty$-cosmos $\eK_{/A}$ whose domain is the  \emph{arrow $\infty$-category} of $A$. The contravariant Yoneda embedding is defined dually; see Definitions \ref{defn:yoneda-embedding} and \ref{defn:contra-yoneda-embedding}.

Our constructive proof of Theorem \ref{thm:general-comprehension} exploits the ``freeness'' of simplicial categories defined by \emph{homotopy coherent realisation}, our term for the left adjoint of the familiar \emph{homotopy coherent nerve} functor from simplicial categories to simplicial sets.\footnote{The codomain of the comprehension functor is a large quasi-category defined as a homotopy coherent nerve, so in fact we define the comprehension functor by constructing its transpose, a simplicial functor $c_{p,A} \colon \gC\Fun_{\eK}(A,B) \to \coCart(\eK)_{/A}$ defining a homotopy coherent diagram of shape $\Fun_{\eK}(A,B)$ valued in the Kan complex enriched category of cocartesian fibrations over $A$ and cartesian functors between them.} Specifically, we show in \S\ref{sec:coherent-nerve} that every simplicial category indexing a homotopy coherent diagram is a \emph{simplicial computad}, a ``cofibrant'' simplicial category in a suitable sense. This allows us to precisely enumerate the data required to define functors whose domains are simplicial computads, the comprehension functor being one important instance. Because the homotopy coherent nerve and homotopy coherent realisation functors are ubiquitous in higher category theory, we indulge ourselves in a comprehensive exposition of the combinatorics of these constructions in anticipation that the technical results proven here will be broadly utilised.

In \S\ref{sec:cosmos}, we introduce the reader to our basic universes for formal category theory --- an $\infty$-cosmos and its homotopy 2-category --- and review the construction of the comma $\infty$-category associated to a cospan of functors. In \S\ref{sec:cocartesian}, we develop the theory of cocartesian fibrations between $\infty$-categories, emphasising the quasi-categorically enriched weak universal properties in contrast with the 2-categorical ones studied in the original account \cite{RiehlVerity:2015fy}. In \S\ref{sec:coherent-nerve}, we include a self-contained development of the theory of simplicial computads and apply it to the analysis of the homotopy coherent realisation and homotopy coherent nerve functors. 

In \S\ref{sec:cocones}, we combine the work of the previous two sections to prove the technical results that describe the essential mechanics of the comprehension construction. Namely, we show that a \emph{cocartesian cocone}, a simplicial natural transformation between a pair of \emph{lax cocones} whose shape is  (the homotopy coherent realisation of) a simplicial set $X$ satisfying certain properties, can be extended along any inclusion $X \inc Y$ to define a cocartesian cocone of shape $Y$. We prove also that the space of such extensions in a contractible Kan complex. Finally, the advertised comprehension functor is defined in \S\ref{sec:comprehension} as the domain component of a cocartesian cocone over a canonically-defined lax cocone. The contractibility condition then implies that the comprehension functor is homotopically unique.

In the concluding \S\ref{sec:computing}, we analyse the action of the comprehension functor on the hom-space $\Hom_{\qB}(a,b)$ between two objects $a,b \colon 1 \to B$ in the underlying quasi-category of $B$. We prove that the induced map $\Hom_{\qB}(a,b) \to \Hom_{\qK}(E_a,E_b)$ is equivalent to a morphism that we refer to as the \emph{external action} of the hom-space $\Hom_{\qB}(a,b)$ on the fibres of $p \colon E \tfib B$. As a corollary we deduce:

{
\renewcommand{\thethm}{\ref{thm:yoneda-ff}}
\begin{thm}$\quad$
\begin{enumerate}[label=(\roman*)]
  \item The Yoneda embedding  is a fully  faithful functor of quasi-categories. 
  \item Every quasi-category is
  equivalent to the homotopy coherent nerve of some Kan complex enriched
  category.
  \end{enumerate} 
\end{thm}
\addtocounter{thm}{-1}
}

\subsection{Notational conventions}

To concisely cite previous work in this program, we refer to the results of \cite{RiehlVerity:2012tt, RiehlVerity:2012hc, RiehlVerity:2013cp,RiehlVerity:2015fy,RiehlVerity:2015ke}  as I.x.x.x., II.x.x.x, III.x.x.x, IV.x.x.x, or V.x.x.x respectively. However note that most of these citations are to definitions, which we also reproduce here. The results in this paper are only minimally reliant on theorems from our previous work. 

We adopt a distinguished typeface to differentiate quasi-categories $\qA, \qB, \qC$ from generic $\infty$-categories $A, B, C$. We designate the quasi-categories constructed from other varieties of higher categories using the same roman letter: e.g., $\qC$ for the homotopy coherent nerve of a Kan complex enriched category $\eC$ (see \ref{prop:qcat-from-kan-enriched}); $\qK$ for the homotopy coherent nerve of the Kan complex enriched core of an $\infty$-cosmos $\eK$ (see \ref{ntn:qcat-from-cosmos}); and $\qB \defeq \Fun_{\eK}(1,B)$ for the underlying quasi-category of an $\infty$-category $B$ (see \ref{defn:underlying-qcat}). 

Herein, as is typical, the adjective ``small'' is used to distinguish those sets
that are members of a Grothendieck universe defined relative to a fixed
inaccessible cardinal. The categories deployed in our meta-theory will generally
be large and locally small, and if we say that such things possess all limits or
colimits then we shall assume it as given that we are asking only for all small
such. The quasi-categories defined as homotopy coherent nerves are typically
large. All other quasi-categories or simplicial sets, particularly those used to
index homotopy coherent diagrams, are assumed to be small.
 
\subsection{Acknowledgements}

The authors are grateful for support from the National Science Foundation (DMS-1551129) and from the Australian Research Council (DP160101519). This work was commenced when the second-named author was visiting the first at Harvard and then at Johns Hopkins and completed while the first-named author was visiting the second at Macquarie. We thank all three institutions for their assistance in procuring the necessary visas as well as for their hospitality.


%% file: background.tex

\section{\texorpdfstring{$\infty$}{infinity}-cosmoi and their homotopy 2-categories}\label{sec:cosmos}

We begin in \S\ref{ssec:cosmoi-background} by reviewing the context for this
work, describing the ``universe'' in which our $\infty$-categories live as
objects: an $\infty$-\emph{cosmos}. In \S\ref{ssec:htpy-2-cat}, we define the
quotient \emph{homotopy 2-category} of an $\infty$-cosmos, within which one
defines adjunctions between $\infty$-categories. In \S\ref{ssec:commas}, we
review the important constructions of \emph{arrow} and \emph{comma}
$\infty$-categories. These are used to describe an internal ``equational''
characterisation of the (co)cartesian fibrations that are studied in
\S\ref{sec:cocartesian}.

\begin{rec}[simplicial categories]\label{rec:simp-cat}
  We work extensively with \emph{simplicial categories}, that is to say
  categories enriched in the cartesian category of simplicial sets (denoted
  using calligraphic letters $\eA, \eB, \eK,...$). We shall often call the
  enriched homs of a simplicial category its \emph{function complexes}, and use
  $\Fun_{\eC}(A,B)$ for the simplicial set of arrows from an object $A$ to an
  object $B$ in $\eC$. An $n$-simplex in $\Fun_{\eC}(A,B)$ is sometimes said to
  be an \emph{$n$-arrow} from $A$ to $B$.
  
  A simplicial category $\eC$ may also be presented as a simplicial object $\eC
  \colon \Del\op \to \Cat$ which acts identically on objects. In this
  representation $\eC$ is comprised of
  \begin{itemize}
  \item categories $\eC_n$ for $n \geq 0$ with a common set of objects
    $\text{ob}\eC$, whose arrows are the $n$-arrows of $\eC$, and
  \item identity-on-objects functors $-\cdot \alpha\colon \eC_m \to\eC_n$
    indexed by and contravariantly functorial in $\alpha \colon [n] \to [m] \in
    \Del$.
  \end{itemize}
  We shall identify $\sCat$ with the full subcategory of $\Cat^{\Del\op}$ of
  those simplicial objects satisfying these properties. It is clear that $\sCat$
  is closed in $\Cat^{\Del\op}$ under all (small) limits and colimits.
\end{rec}

\begin{rec}[simplicial dual]\label{rec:simp-dual}
  Let $(-)\dual\colon\Del\to\Del$ denote the functor which acts by carrying each
  ordinal $[n]$ to its categorical dual. Precomposing by this duality we obtain
  a functor $(-)\op\colon\sSet\to\sSet$ called the \emph{simplicial dual}.
\end{rec}

\begin{rec}[duals of simplicial categories]\label{rec:simp-cat-duals}
 A simplicial category $\eC$ admits two distinct duals:
 \begin{enumerate}[label=(\roman*)]
 \item The \emph{opposite category} $\eC\op$ is simply the enriched variant of
   the familiar dual category, with the same set of objects but with 
   \[ \Fun_{\eC\op}(A,B) \defeq\Fun_{\eC}(B,A).\]
 \item The simplicial category $\eC\co$ is constructed by applying the product
   preserving simplicial dual functor to the hom-spaces of $\eC$. That is $\eC$
   and $\eC\co$ share the same sets of objects and their hom-spaces are related
   by the equation \[\Fun_{\eC\co}(A,B)=\Fun_{\eC}(A,B)\op.\]
\end{enumerate}
\end{rec}

\subsection{\texorpdfstring{$\infty$}{infinity}-cosmoi}\label{ssec:cosmoi-background}

An $\infty$-cosmos is a category $\eK$ whose objects $A, B$ we call
$\infty$-\emph{categories} and whose function complexes $\Fun_{\eK}(A,B)$ are
quasi-categories of \emph{functors} between them. The handful of axioms imposed
on the ambient quasi-categorically enriched category $\eK$ permit the
development of a general theory of $\infty$-categories ``synthetically,'' i.e.,
only in reference to this axiomatic framework. We work in an $\infty$-cosmos
$\eK$ with all objects cofibrant, in contrast to the more general notion first
introduced in \cite{RiehlVerity:2015fy}.

\begin{defn}[$\infty$-cosmos]\label{defn:cosmos}
  An $\infty$-\emph{cosmos} is a simplicial category $\eK$ whose
  function spaces $\Fun_{\eK}(A,B)$ are quasi-categories and which is equipped
  with a specified subcategory of \emph{isofibrations}, denoted by ``$\tfib$'',
  satisfying the following axioms:
 \begin{enumerate}[label=(\alph*)]
    \item\label{defn:cosmos:a} (completeness) As a simplicial category,  $\eK$ possesses a terminal object $1$, small products, cotensors $A^U$ of  objects $A$ by all small simplicial sets $U$, inverse limits of countable sequences of isofibrations, and pullbacks of isofibrations along any functor.
    \item\label{defn:cosmos:b} (isofibrations) The class of isofibrations contains the isomorphisms and all of the functors $!\colon A \tfib 1$ with codomain $1$; is stable under pullback along all functors; is closed under inverse limit of countable sequences; and if $p\colon E\tfib B$ is an isofibration in $\eK$ and $i\colon U\inc V$ is an inclusion of  simplicial sets then the Leibniz cotensor $i\leib\pwr p\colon E^V\tfib E^U\times_{B^U} B^V$ is an isofibration. Moreover, for any object $X$ and isofibration $p \colon E \tfib B$, $\Fun_{\eK}(X,p) \colon \Fun_{\eK}(X,E) \tfib \Fun_{\eK}(X,B)$ is an isofibration of quasi-categories.
\end{enumerate}
The underlying category of an $\infty$-cosmos $\eK$ has a canonical subcategory of (representably-defined) \emph{equivalences}, denoted by ``$\we$'', satisfying the 2-of-6 property. A functor $f \colon A \to B$ is an \emph{equivalence} just when the induced functor $\Fun_{\eK}(X,f) \colon \Fun_{\eK}(X,A) \to \Fun_{\eK}(X,B)$ is an equivalence of quasi-categories for all objects $X \in \eK$.  The  \emph{trivial fibrations}, denoted by ``$\trvfib$'', are those functors that are both equivalences and isofibrations.
 \begin{enumerate}[label=(\alph*), resume]
    \item\label{defn:cosmos:c} (cofibrancy) All objects are \emph{cofibrant}, in the sense that they enjoy the left lifting property with respect to all trivial fibrations in $\eK$. 
\[ \xymatrix{ & E \ar@{->>}[d]^{\rotatebox{90}{$\displaystyle\sim$}} \\ A \ar[r] \ar@{-->}[ur]^{\exists} & B}\] 
  \end{enumerate}
It follows from \ref{defn:cosmos}\ref{defn:cosmos:a}-\ref{defn:cosmos:c} that:
 \begin{enumerate}[label=(\alph*), resume]
    \item\label{defn:cosmos:d} (trivial fibrations) The trivial fibrations define a subcategory containing  the isomorphisms; are stable under pullback along all functors; the formation of inverse limits of countable sequences; and the Leibniz cotensor $i\leib\pwr p\colon E^V\trvfib E^U\times_{B^U}  B^V$ of an isofibration $p\colon E\tfib B$ in $\eK$ and a monomorphism $i\colon U\inc V$ between presented simplicial sets   is a trivial fibration when $p$ is a trivial fibration in $\eK$ or $i$ is trivial cofibration in the Joyal model structure on $\sSet$ (see \refV{lem:triv.fib.stab}).
\item\label{defn:cosmos:e} (factorisation) Any functor $f \colon A \to B$ may be factored as $f = p j$ 
\[ \xymatrix{ & N_f \ar@{->>}[dr]^p \ar@{->>}@/_3ex/[dl]_{q}^*-{\rotatebox{45}{$\labelstyle\sim$}} \\ A \ar[rr]_f \ar[ur]^*-{\rotatebox{45}{$\labelstyle\sim$}}_j & & B}\] where $p \colon N_f \tfib B$ is an isofibration and $j \colon A \we N_f$ is right inverse to a trivial fibration $q \colon N_f \trvfib A$ (see \refIV{lem:Brown.fact}).
\end{enumerate}
\end{defn}

\begin{ex}[$\infty$-cosmos of quasi-categories]\label{ex:qcat-cosmos}
The prototypical example is the $\infty$-cosmos of quasi-categories, with function complexes  inherited from the usual cartesian closed category of simplicial sets. An \emph{isofibration} is an inner fibration that has the right lifting property with respect to the inclusion $\catone\inc\iso$ of either endpoint of the (nerve of the) free-standing isomorphism. The equivalences are the simplicial homotopy equivalences defined with respect to the interval $\iso$. That is, a map $f \colon A \to B$ of quasi-categories is an \emph{equivalence} just when there exists a map $g \colon B \to A$ together with maps $A \to A^\iso$ and $B \to B^\iso$ that restrict along the vertices of $\iso$ to the maps $\id_A$, $gf$, $fg$, and $\id_B$ respectively: 
\[ f \colon A \we B\quad \mathrm{iff}\quad \exists g \colon B \we A\quad \mathrm{and} \quad \vcenter{\xymatrix{ & A \\ A \ar[r] \ar[ur]^{\id_A} \ar[dr]_{gf} & A^\iso \ar@{->>}[u]_{p_0} \ar@{->>}[d]^{p_1} \\ & A}} \quad\mathrm{and}\quad \vcenter{\xymatrix{ & B \\ B \ar[r]\ar[ur]^{fg} \ar[dr]_{\id_B} & B^{\iso} \ar@{->>}[u]_{p_0} \ar@{->>}[d]^{p_1} \\ & B}}\]
\end{ex}

The $\infty$-cosmos of quasi-categories can also be described using the language of model categories. It is the full subcategory of fibrant objects, with the isofibrations and equivalences respectively taken to be the fibrations and weak equivalences between fibrant objects, in a model category that is enriched over the Joyal model structure on simplicial sets and in which all fibrant objects are cofibrant; in this case, that model category is the Joyal model structure on simplicial sets itself. It is easy to verify that any category of fibrant objects arising in this way defines an $\infty$-cosmos (see Lemma \refIV{lem:model-categories-cosmoi}). This is the source of the majority of our examples, which are described in more detail in \S\refIV{sec:cosmoi}.

\begin{ex}[$\infty$-cosmoi of $(\infty,1)$-categories]\label{ex:infinity-1} There are $\infty$-cosmoi $\CSS$ and $\Segal$ whose $\infty$-categories are the complete Segal spaces or Segal categories respectively, models of $(\infty,1)$-categories introduced by Rezk \cite{Rezk:2001sf} and by Hirschowitz-Simpson.
\end{ex}

\begin{ex}[$\infty$-cosmoi of $(\infty,n)$-categories] There is also an $\infty$-cosmos whose objects are $\theta_n$-\emph{spaces}, a model of $(\infty,n)$-categories introduced by Rezk \cite{Rezk:2010fk}. For any sufficiently nice model category $\eM$, there is an $\infty$-cosmos of \emph{Rezk objects} in $\eM$, the iterated complete Segal spaces of Barwick being one special case.
\end{ex}

For any $\infty$-category $A$ in any $\infty$-cosmos $\eK$, the strict slice $\eK_{/A}$ is again an $\infty$-cosmos. It follows that all of our theorems in this axiomatic framework immediately have fibred analogues.

\begin{defn}\label{defn:sliced-cosmoi} 
If $\eK$ is any $\infty$-cosmos and $A \in \eK$ is any object, then there is an $\infty$-cosmos $\eK_{/A}$, the \emph{sliced $\infty$-cosmos of $\eK$ over $A$}, whose:
\begin{itemize}
\item objects are isofibrations $p \colon E \tfib A$ with codomain $A$;
\item mapping quasi-category from $p \colon E \tfib A$ to $q \colon F \tfib A$ is defined by  taking the pullback
\[
    \xymatrix@=1.5em{
      {\Fun_A(p,q)}\pbexcursion\ar[r]\ar@{->>}[d] &
      {\Fun_{\eK}(E,F)}\ar@{->>}[d]^{\Fun_{\eK}(E,q)} \\
      {1}\ar[r]_-{p} & {\Fun_{\eK}(E,A)}
    }
\]
  in simplicial sets;
\item isofibrations, equivalences, and trivial fibrations are created by the forgetful functor $\eK_{/A} \to \eK$;
\end{itemize}
and in which the simplicial limits are defined in the usual way for sliced simplicial categories (see \refV{defn:sliced-cosmoi}).
\end{defn}

\subsection{The homotopy 2-category of an \texorpdfstring{$\infty$}{infinity}-cosmos}\label{ssec:htpy-2-cat}

In fact most of the development of the basic theory of $\infty$-categories takes
place not in an $\infty$-cosmos, but in a quotient of the $\infty$-cosmos that
we call its \emph{homotopy 2-category}. Each $\infty$-cosmos has an underlying
1-category whose objects are the $\infty$-categories of that $\infty$-cosmos and
whose morphisms, which we call $\infty$-\emph{functors} or more often simply
\emph{functors}, are the vertices of the function complexes.

\begin{defn}[the homotopy 2-category of $\infty$-cosmos] The \emph{homotopy
    2-category} of an $\infty$-cosmos $\eK$ is a strict 2-category $\eK_2$ so
  that
\begin{itemize}
\item the objects of $\eK_2$ are the objects of $\eK$, i.e., the
  $\infty$-categories;
\item the 1-cells $f \colon A \to B$ of $\eK_2$ are the vertices $f \in
  \Fun_{\eK}(A,B)$ in the function complexes of $\eK$, i.e., the
  $\infty$-functors;
\item a 2-cell $\xymatrix@C=3em{ A \ar@/^1.5ex/[]!R(0.5);[r]!L(0.65)^f
    \ar@/_1.5ex/[]!R(0.5);[r]!L(0.65)_g \ar@{}[r]|{\Downarrow\alpha}& B}$ in
  $\eK_2$ is represented by a 1-simplex $\alpha \colon f \to g \in
  \Fun_{\eK}(A,B)$, and a parallel pair of 1-simplices in $\Fun_{\eK}(A,B)$
  represent the same 2-cell if and only if they bound a 2-simplex whose
  remaining outer face is degenerate.
\end{itemize}
Put concisely, the homotopy 2-category is the 2-category $\eK_2 \defeq \ho_*\eK$
defined by applying the homotopy category functor $\ho \colon \qCat \to
\eop{Cat}$ to the function complexes of the $\infty$-cosmos; so the hom-categories
in $\eK_2$ are defined by the formula: \[\Hom_{\eK_2}(A,B)\defeq
  \ho(\Fun_{\eK}(A,B))\]
\end{defn}

\begin{ex}[dual $\infty$-cosmoi]\label{ex:dual-cosmoi} The dual $\eK\co$ of any $\infty$-cosmos
  $\eK$ is again an $\infty$-cosmos. The homotopy 2-category of $\eK\co$ is the
  ``co'' dual of the homotopy 2-category of $\eK$, reversing the 2-cells but not
  the 1-cells.
\end{ex}

Proposition \refIV{prop:equiv.are.weak.equiv} proves that the equivalences between $\infty$-categories admit another important characterisation: they are precisely the equivalences in the homotopy 2-category of the $\infty$-cosmos. The upshot is that equivalence-invariant 2-categorical constructions are appropriately ``homotopical'', characterising $\infty$-categories up to equivalence, and that we may use the term ``equivalence'' unambiguously in both the quasi-categorically enriched and 2-categorical contexts.

Similarly, the reason we have chosen the term ``isofibrations'' for the designated class of $\infty$-functors $A \tfib B$ is because these maps define isofibrations in the homotopy 2-category.  An \emph{isofibration} in a 2-category is a 1-cell $p \colon A \tfib B$ so that any invertible 2-cell whose domain factors through $p$ can be lifted along $p$ to define an invertible 2-cell with codomain $A$; see \refIV{rec:trivial-fibration} and \refIV{lem:isofib.are.representably.so}.

\begin{defn} An \emph{adjunction} between $\infty$-categories $A, B \in \eK$ is
  simply an adjunction in the homotopy 2-category $\eK_2$. Such things comprise
  a pair of functors $f \colon B \to A$ and $u \colon A \to B$, together with a
  pair of 2-cells $\eta \colon \id_B \To uf$ and $\epsilon \colon fu \To \id_A$
  satisfying the triangle identities.
\end{defn}

\begin{defn}[left adjoint right inverses]\label{defn:lari}
  We say that a map $u\colon A\to B$ of $\eK$ has a \emph{left adjoint right
    inverse} (or \emph{lari}) if it admits a left adjoint $f\colon B\to A$ whose
  unit is an isomorphism. In the situation where $u\colon A\tfib B$ is an
  isofibration we may lift the unit isomorphism of $f\dashv u$ to give a another
  left adjoint $f'\colon B\to A$ to $u$ for which the unit is an identity. (The proof is left as an exercise, or see
  Lemma~\refI{lem:isofibration-RARI}.) 
\end{defn}  
  
  The class of functors with left adjoint
  right inverses is closed under composition and contains all equivalences.  Various duals of this notion exist, named by the obvious acronyms \emph{lali}
  (left adjoint left inverse), \emph{rari} (right adjoint right inverse), and
  $\emph{rali}$ (right adjoint left inverse).

\begin{defn} A \emph{functor of $\infty$-cosmoi} is a simplicial functor that preserves the classes of isofibrations and each of the limits specified in \ref{defn:cosmos}\ref{defn:cosmos:a}.
\end{defn}

\begin{ex}
Functors of $\infty$-cosmoi include:
\begin{enumerate}[label=(\roman*)]
\item the representable functor $\Fun_{\eK}(X,-) \colon\eK \to \qCat$ for any object $X \in \eK$;
\item as a special case, the \emph{underlying quasi-category functor} $\Fun_{\eK}(1,-) \colon \eK \to \qCat$; 
\item the simplicial cotensor $(-)^U \colon \eK \to \eK$ with any simplicial set $U$; and 
\item the pullback functor $f^* \colon \eK_{/B} \to \eK_{/A}$ for any functor $f \colon A \to B \in \eK$.
\end{enumerate}
among others.
\end{ex}

A functor of $\infty$-cosmoi induces a 2-functor between their homotopy 2-categories. Any 2-functor preserves the internal 2-categorical notions of adjunction or equivalence. A common theme is that $\infty$-categorical definitions that admit ``internal'' characterisations in the $\infty$-cosmos are also preserved by functors of $\infty$-cosmoi, as we shall see in Lemma \ref{lem:functor-preservation}.  Those internal characterisations make use of the \emph{comma construction}, a subject to which we now turn.

\subsection{Arrow and comma constructions}\label{ssec:commas}

The simplicially-enriched limits in the $\infty$-cosmos provided by the axioms  \ref{defn:cosmos}\ref{defn:cosmos:a} and \ref{defn:cosmos:b} can be used to build new $\infty$-categories from a diagram of $\infty$-categories and $\infty$-functors. Here we are interested in two important cases: \emph{arrow} $\infty$-categories, which are a  special case of a more general \emph{comma construction}.

\begin{defn}[arrow $\infty$-categories] For any $\infty$-category $A$, the simplicial cotensor 
\[ \xymatrix@C=30pt{ A^\cattwo \defeq A^{\Del^1} \ar@{->>}[r]^-{(p_1,p_0)} & {A^{\boundary\Delta^1}} \cong A \times A}\] defines the \emph{arrow $\infty$-category} $A^\cattwo$, equipped with an isofibration $(p_1,p_0)\colon A^\cattwo \tfib A \times A$, where $p_1 \colon A^\cattwo \tfib A$ denotes the codomain projection and $p_0 \colon A^\cattwo \tfib A$ denotes the domain projection.
\end{defn}

Using the notation $\cattwo \defeq \Delta^1$, the defining universal property of the simplicial cotensor asserts that the canonical map defines an isomorphism of quasi-categories
\[ \Fun_{\eK}(X, A^\cattwo) \stackrel{\cong}{\longrightarrow} \Fun_{\eK}(X,A)^{\cattwo}.\]  In particular, taking $X=A^\cattwo$, the identity functor $\id_{A^\cattwo}$ transposes to define a vertex in $\Fun_{\eK}(A^\cattwo,A)^\cattwo$, a 1-arrow in $\Fun_{\eK}(A^\cattwo,A)$ which we display as
\begin{equation}\label{eq:generic-arrow} \xymatrix@C=4em{ A^\cattwo \ar@{->>}@/^2ex/[]!R(0.6);[r]!L(0.6)^{p_0} \ar@{->>}@/_2ex/[]!R(0.6);[r]!L(0.6)_{p_1} \ar@{}[r]|-{\Downarrow\phi} & A}\end{equation} since it represents a  2-cell with this boundary in the homotopy 2-category $\eK_2$.

Using the simplicially enriched pullbacks of isofibrations that exist by virtue of  axioms  \ref{defn:cosmos}\ref{defn:cosmos:a} and \ref{defn:cosmos:b}, arrow $\infty$-categories can be used to define a general \emph{comma $\infty$-category} associated to a cospan of functors.

\begin{defn}[comma $\infty$-categories]\label{defn:comma} Any pair of functors  $f\colon B\to A$ and $g\colon C\to A$ in an $\infty$-cosmos $\eK$ has an associated \emph{comma $\infty$-category}, constructed by the following pullback, formed in $\eK$:
  \begin{equation}\label{eq:comma-as-simp-pullback}
    \xymatrix@=2.5em{
      {f\comma g}\pbexcursion \ar[r]\ar@{->>}[d]_{(p_1,p_0)} &
      {A^\cattwo} \ar@{->>}[d]^{(p_1,p_0)} \\
      {C\times B} \ar[r]_-{g\times f} & {A\times A}
    }
  \end{equation}
    Transposing the data in this diagram, we obtain an associated square
  \begin{equation}\label{eq:comma-1-arrow}
    \xymatrix@R=2em{
      {f\comma g}\ar@{->>}[r]^{p_0}\ar@{->>}[d]_{p_1} & {B}\ar[d]^f
      \ar@{}[dl] \ar@{=>} ?(0.4);?(0.6)_{\phi}  \\
      {C}\ar[r]_g & {A}
    }
  \end{equation}
  in which $\phi$ is a $1$-arrow. Note that, by construction, the map $(p_1,p_0) \colon f \comma g \tfib C \times B$ is an isofibration.
  
  If we let $\pbshape$ denote the three object category $\{a\rightarrow
  c\leftarrow b\}$, then the comma construction extends to a simplicial functor
  \begin{equation*}
    \xymatrix@R=0em@C5em{{\eK^{\pbshape}}\ar[r]^{\comma} & {\eK}}
  \end{equation*}
  Here $\eK^{\pbshape}$ denotes the simplicial functor category, so its objects
  are diagrams of the form
  \begin{equation*}
    \xymatrix@R=0em@C=5em{
      {B}\ar[r]^f & {A} & {C}\ar[l]_g
    }
  \end{equation*}
  and its $0$-arrows are natural transformations
\[
    \xymatrix@R=1.5em@C=5em{
      {B}\ar[r]^{f}\ar[d]_{q} & {A}\ar[d]_{p} & {C}\ar[l]_{g}\ar[d]^{r} \\
      {B'}\ar[r]_{f'} & {A'} & {C'}\ar[l]^{g'}
    }
  \]
  which are composed pointwise.   We shall use the notation $\comma(q, p,r)\colon
  f\comma g\to f'\comma g'$ to denote the image of this $0$-arrow of under this
  comma construction functor $\comma\colon\eK^\pbshape\to\eK$.
\end{defn}

As a simplicially-enriched limit in $\eK$, the $\infty$-category $f \comma g$ has a universal property expressed via a natural isomorphism of quasi-categories
\begin{equation}\label{eq:comma-simp-UP} \Fun_{\eK}(X, f\comma g)\stackrel{\cong}{\longrightarrow} \Fun_{\eK}(X,f) \comma \Fun_{\eK}(X,g)\end{equation} for any $X \in \eK$, where the right-hand side is computed by the analogous pullback to \eqref{eq:comma-as-simp-pullback}, formed in $\qCat$. 

\begin{lem}[homotopical properties of the comma construction]\label{lem:comma}
  The comma functor $\comma\colon\eK^\pbshape\to\eK$ 
  \begin{enumerate}[label=(\roman*)]
  \item carries each pointwise isofibration (respectively pointwise
    equivalence) in $\eK^\pbshape$ to an isofibration (respectively equivalence)
    in $\eK$; and 
  \item carries each pullback of a pointwise isofibration in $\eK^\pbshape$ to
    a pullback of an isofibration in $\eK$.
  \end{enumerate}
  It follows, therefore, that the comma construction maps each pointwise
  homotopy pullback in $\eK^\pbshape$ to a homotopy pullback in $\eK$.
\end{lem}
\begin{proof}
  As discussed in Lemma~\refI{lem:comma-obj-maps}, the first result follows by a
  standard argument originally due to Reedy~\cite{reedy1973htm}. Because limits commute, the second
  result follows.
\end{proof}

%% file: cocartesian.tex
\section{Cocartesian fibrations of \texorpdfstring{$\infty$}{infinity}-categories}\label{sec:cocartesian}

\emph{Cartesian fibrations} are isofibrations $p \colon E \tfib B$ in an
$\infty$-cosmos $\eK$ whose fibres depend functorially on the base, in a sense
described by a lifting property for certain 2-cells. To complement prior
treatments, we dualise the exposition, defining \emph{cocartesian fibrations}
instead. Cartesian fibrations are then cocartesian fibrations in the dual
$\infty$-cosmos $\eK\co$ of Example \ref{ex:dual-cosmoi}.

 In \S\refIV{sec:cartesian}, cocartesian fibrations are defined internally to
the homotopy 2-category of an $\infty$-cosmos as functors $p \colon E \tfib B$
for which any 2-cell with codomain $B$ admits a $p$-\emph{cocartesian} lift with
specified domain, satisfying certain properties. The precise definition is a
weakening of the standard definition of a fibration in any 2-category
\cite{Street:1974:FibYoneda}, as is appropriate for the present homotopical
context. However, we won't make use of this internal 2-categorical definition
here. Instead in \S\ref{ssec:cartesian} we introduce the internal definition at
the level of the $\infty$-cosmos. Then in \S\ref{ssec:cocartesian-arrows}, we
circle back after the fact to unravel the definition to reveal the relevant
class of $p$-\emph{cocartesian} 1-arrows, which precisely represent the
$p$-\emph{cocartesian 2-cells} in the homotopy 2-category that were the focus of
\S\refIV{sec:cartesian}.

\subsection{Cocartesian fibrations}\label{ssec:cartesian}
 
Cocartesian fibrations can be characterised internally to the $\infty$-cosmos
via adjoint functors involving comma $\infty$-categories, for which we now
establish notation.
 
\begin{ntn}\label{ntn:comma-adjoints}
  For any isofibration $p \colon E \tfib B$, there is a canonically defined
  functor $k=\comma(\id_E,p, p)\colon E^\cattwo\tfib p\comma B$, which is an
  isofibration by Lemma~\ref{lem:comma}. Write $i \colon E \to p\comma B$ for the
  restriction of $k$ along the diagonal $E \to E^\cattwo$. These functors whisker
  with the canonical 1-arrow of \eqref{eq:comma-1-arrow} to satisfy the following
  pasting identities:
  \begin{equation*}
    \vcenter{\xymatrix@C=0.8em@R=1.2em{
        & {E}\ar[d]^-{i} & \\
        & {p \comma B}\ar@{->>}[dl]_{p_1}\ar@{->>}[dr]^{p_0} & \\
        {B} && {E}\ar[ll]^p
        \ar@{} "2,2";"3,2" |(0.6){\Leftarrow\phi}
      }} = 
    \vcenter{\xymatrix@C=0.8em@R=1.2em{
        & {E}\ar@{->>}[dl]_p\ar@{=}[dr] & \\
        {B} && {E}\ar[ll]^p
        \ar@{} "1,2";"2,2" |(0.6){=}
      }}   
    \mkern50mu
    \vcenter{\xymatrix@C=0.8em@R=1.2em{
        & {E^\cattwo}\ar[d]^-{k} & \\
        & {p \comma B}\ar@{->>}[dl]_{p_1}\ar@{->>}[dr]^{p_0} & \\
        {B} && {E}\ar[ll]^p
        \ar@{} "2,2";"3,2" |(0.6){\Leftarrow\phi}
      }} = 
    \vcenter{\xymatrix@C=0.8em@R=1.2em{
        & {E^\cattwo}\ar@{->>}_{pp_1}[dl]\ar@{->>}[dr]^{p_0} & \\
        {B} && {E}\ar[ll]^p
        \ar@{} "1,2";"2,2" |(0.6){\Leftarrow p\phi}
      }}
  \end{equation*}
\end{ntn}

\begin{defn}[{\refIV{thm:cart.fib.chars}}]\label{defn:cocart-fibration}
  An isofibration $p \colon E \tfib B$ is a \emph{cocartesian fibration} if
  either of the following equivalent conditions hold:
  \begin{enumerate}[label=(\roman*)]
  \item\label{itm:cocart.fib.chars.ii} The functor $i\colon E\to p\comma B$
    admits a left adjoint in the slice $\infty$-cosmos $\eK_{/B}$:
    \begin{equation}\label{eq:cocartesian.fib.adj}
      \xymatrix@R=2em@C=3em{
        {E}\ar@{->>}[dr]_{p} \ar@/_0.6pc/[]!R(0.5);[rr]_{i}^{}="a" & &
        {p \comma B}\ar@{->>}[dl]^{p_1} \ar@{-->}@/_0.6pc/[ll]!R(0.5)_{\ell}^{}="b" 
        \ar@{}"a";"b"|{\bot} \\
        & B &
      }
    \end{equation}
  \item\label{itm:cocart.fib.chars.iii} The functor $k\colon E^\cattwo\to p\comma B$ admits a left adjoint right inverse in $\eK$:
    \begin{equation}\label{eq:cartesian.isosect.adj}
    \xymatrix@C=6em{
      {E^\cattwo}\ar@/_0.8pc/[]!R(0.6);[r]!L(0.45)_{k}^{}="u" &
      {p \comma B}\ar@{-->}@/_0.8pc/[]!L(0.45);[l]!R(0.6)_{\bar{\ell}}^{}="t"
      \ar@{}"u";"t"|(0.6){\bot}
    }
    \end{equation}
  \end{enumerate}
\end{defn}

\begin{defn}[{\refIV{thm:cart.fun.chars}}]\label{defn:cartesian-functor}
  Given two cocartesian fibrations $p\colon E\tfib B$ and
  $q\colon F\tfib A$ in $\eK$, then a pair of functors $(g,f)$ in the
  following commutative square
  \begin{equation}\label{eq:cart.fun}
    \xymatrix{{F}\ar[r]^{g}\ar@{->>}[d]_-{q} & {E} \ar@{->>}[d]^-{p} \\ {A}\ar[r]_{f} & {B}}
  \end{equation}
  comprise a {\em cartesian functor\/} if and only if the mate of either (and
  thus both) of the commutative squares
  \begin{equation*}
    \xymatrix@R=2em@C=3em{{F}\ar[r]^{g}\ar[d]_{i} & {E} \ar[d]^{i} \\
      {q\comma A}\ar[r]_{\comma (g,f,f)} & {p\comma B}}
    \mkern40mu
    \xymatrix@R=2em@C=3em{{F^{\cattwo}}\ar[r]^{g^{\cattwo}}\ar@{->>}[d]_{k} &
      {E^{\cattwo}} \ar@{->>}[d]^{k} \\ {q\comma A}\ar[r]_{\comma (g,f,f)} & {p\comma B}}
  \end{equation*}
  under the adjunctions of Definition \ref{defn:cocart-fibration} is an isomorphism.
\end{defn}

Immediately from these definitions:

\begin{lem}\label{lem:functor-preservation}
Functors of $\infty$-cosmoi preserve cocartesian fibrations and cartesian functors.
\end{lem}
\begin{proof}
  Cartesian fibrations and functors are defined using adjunctions, isofibrations
  and comma objects, all of which are preserved by any functor of
  $\infty$-cosmoi.
  \end{proof}
  
  \subsection{Cocartesian 1-arrows}\label{ssec:cocartesian-arrows}

  Our aim now is to introduce the notion of $p$-\emph{cocartesian} 1-arrow. We
  do this first in the case of cocartesian fibrations of quasi-categories and
  then generalise this notion to any $\infty$-cosmos.
  
\begin{defn}[$p$-cocartesian arrows]\label{defn:qcat-cocart-edge}
  In the case where $p\colon \qE\tfib\qB$ is a cocartesian fibration of
  quasi-categories, we say that a arrow ($1$-simplex) $\chi\colon e\to e'$ of
  $\qE$ is \emph{$p$-cocartesian\/} if and only if when it is regarded as an
  object of $\qE^\cattwo$ it is isomorphic to some object in the image of the
  left adjoint right inverse functor $\bar{\ell}\colon p\comma \qB\to
  \qE^\cattwo$ of Definition~\ref{defn:cocart-fibration}\ref{itm:cocart.fib.chars.iii}.
\end{defn}

\begin{lem}[{\refIV{cor:lurie-cartesian}}]\label{lem:qcat-cocart}
  An arrow $\chi\colon e\to e'$ of $\qE$ is $p$-cocartesian if and only if any
  lifting problem
  \begin{equation*}
    \xymatrix@C=3em{
      \Delta^1\ar@/^2ex/[rr]!L+/u 4pt/^\chi \ar[r]_-{\fbv{0,1}} &
      \Horn^{n,0} \ar[d] \ar[r] &      \qE \ar@{->>}[d]^{p} \\ &
      \Delta^n \ar[r] \ar@{-->}[ur] & \qB}
  \end{equation*}
  has a solution. Hence
  $p\colon\qE\tfib\qB$ is a cocartesian fibration of quasi-categories precisely when any arrow
  $\alpha\colon pe \to b$ in $\qB$ admits a lift to an arrow $\chi\colon e\to
  e'$ in $\qE$ which enjoys this lifting property.
\end{lem}

We now extend the notion of cocartesian edge from Definition \ref{defn:qcat-cocart-edge} to define a \emph{cocartesian cylinder}. 

\begin{defn}[cocartesian cylinders]\label{defn:cocartesian-cylinder}
  Suppose that $p\colon \qE\tfib\qB$ is a cocartesian fibration of
  quasi-categories and that $X$ is any simplicial set. We say that a cylinder
  $e\colon X\times\Del^1\to \qE$ is \emph{pointwise $p$-cocartesian\/} if and
  only if for each $0$-simplex $x\in X$ it maps the $1$-simplex $(x\cdot
  \degen^0,\id_{[1]})$ to a $p$-cocartesian arrow in $\qE$.
  \end{defn}
  
  \begin{lem}\label{lem:cocart-cyl-are-cocart-arrows}
    Let $p \colon \qE\tfib\qB$ be a cocartesian fibration of quasi-categories. A
    cylinder $e\colon X\times\Del^1\to \qE$ is pointwise $p$-cocartesian if and
    only if its dual $e \colon \Del^1 \to \qE^X$ defines a $p^X$-cocartesian
    arrow for the cocartesian fibration $p^X \colon \qE^X \tfib \qB^X$.
  \end{lem}
  \begin{proof}
    First note that Lemma \ref{lem:functor-preservation} implies that $p^X
    \colon \qE^X \tfib \qB^X$ is a cocartesian fibration. To prove the stated
    equivalence, note that a vertex is in the essential image of $\bar{\ell}^X
    \colon p^X \comma \qB^X \to \qE^{X \times \cattwo}$ if and only if it
    transposes to a diagram $X \to \qE^\cattwo$ whose vertices land in the
    essential image of $\bar\ell \colon p \comma \qB \to \qE$.
\end{proof}

  It follows that for all $f\colon
  X\to Y$ the pair of functors $(\qB^f,\qE^f)$ in the commutative square
  \begin{equation*}
    \xymatrix@=2em{
      {\qE^Y}\ar@{->>}[d]_-{p^Y}\ar[r]^{\qE^f} & {\qE^X}\ar@{->>}[d]^-{p^X} \\
      {\qB^Y}\ar[r]_{\qB^f} & {\qB^X}
    }
  \end{equation*}
  comprise a cartesian functor of cocartesian fibrations.

  \begin{lem}\label{lem:cocart-cylinder-extensions}
Let $X \inc Y$ be a simplicial subset of a simplicial set $Y$.
\begin{enumerate}[label=(\roman*)]
  \item Any lifting problem
    \begin{equation*}
      \xymatrix@=2em{
        {X\times\Del^1\cup Y\times\Del^{\{0\}}}\ar[r]^-{e}\ar@{^(->}[d] &
        {\qE}\ar@{->>}[d]^{p} \\
        {Y\times\Del^1}\ar[r]_{b}\ar@{.>}[ru]^{\bar{e}} & {\qB}
      }
    \end{equation*}
    with the property that the cylinder
    ${X\times\Del^1}\subseteq{X\times\Del^1\cup
      Y\times\Del^{\{0\}}}\stackrel{e}\longrightarrow {E}$ is pointwise
    $p$-cocartesian admits a solution $\bar{e}$ which is also pointwise
    $p$-cocartesian.
  \item Any lifting problem ($n>1$)
    \begin{equation*}
      \xymatrix@=2em{
        {X\times\Del^n\cup Y\times\Horn^{n,0}}\ar[r]^-{e}\ar@{^(->}[d] &
        {\qE}\ar@{->>}[d]^{p} \\
        {Y\times\Del^n}\ar[r]_{b}\ar@{.>}[ru]^{\bar{e}} & {\qB}
      }
    \end{equation*}
    in which the cylinder ${Y\times\Del^{\{0,1\}}}\subseteq {X\times\Del^n\cup
      Y\times\Horn^{n,0}}\stackrel{e}\longrightarrow {E}$ is pointwise
    $p$-cocartesian admits a solution $\bar{e}$.
\end{enumerate}
\end{lem}
\begin{proof} The proof is by a standard lifting argument. For (i) this is given in the proof of Corollary \refIV{cor:lurie-cartesian}.
\end{proof}

In an $\infty$-cosmos, adjunctions are representably defined  and of course the same is also true of simplicially enriched limit notions. As these are the two ingredients in the definitions of cocartesian fibrations and cartesian functors appearing in Definition \ref{defn:cocart-fibration} and \ref{defn:cartesian-functor}, it follows that these notions are also representably defined in the following sense.

\begin{obs}[cocartesian fibrations are representably defined]\label{obs:representable-cocart}
  Corollary~\refIV{cor:cart-fib-rep} tells us that an isofibration $p\colon
  E\tfib B$ in our $\infty$-cosmos $\eK$ is a cocartesian fibration if and only
  if
  \begin{itemize}
  \item for all objects $X\in\eK$ the associated isofibration $\Fun_{\eK}(X,p)\colon
    \Fun_{\eK}(X,E)\tfib \Fun_{\eK}(X,B)$ is a cocartesian fibration of
    quasi-categories, and
  \item for all $0$-arrows $f\colon Y\to X$ in $\eK$ the induced functor
  \[ \xymatrix@C=3em{ \Fun_{\eK}(X,E) \ar[r]^-{\Fun_{\eK}(f,E)} \ar@{->>}[d]_{\Fun_{\eK}(X,p)} & \Fun_{\eK}(Y,E) \ar@{->>}[d]^{\Fun_{\eK}(Y,p)} \\ \Fun_{\eK}(X,B) \ar[r]_-{\Fun_{\eK}(f,B)} & \Fun_{\eK}(Y,B)}\] is cartesian.
  \end{itemize}
  Consequently, we may characterise the cocartesian fibrations of $\eK$ in terms
  of lifting properties amongst arrows from each object $X\in\eK$ into the
  objects $B$ and $E$.

  In a similar manner, a commutative square in $\eK$ whose vertical maps are
  cocartesian fibrations
  \begin{equation*}
    \xymatrix@=2em{
      {E}\ar[r]^g\ar@{->>}[d]_{p} & {F}\ar@{->>}[d]^{q} \\
      {B}\ar[r]_f & {A}
    }
  \end{equation*}
  is a cartesian functor if and only if for all objects $X\in\eK$, the image of this data under $\Fun_{\eK}(X,-)$ defines a cartesian functor between cocartesian fibrations of quasi-categories. Consulting Definition \ref{defn:cartesian-functor}, this is the case if and only if 
  $\Fun_{\eK}(X,g)\colon\Fun_{\eK}(X,E)\to\Fun_{\eK}(X,F)$ carries
  $\Fun_{\eK}(X,p)$-cocartesian arrows to $\Fun_{\eK}(X,q)$-cocartesian ones.
\end{obs}

\begin{defn}[cocartesian 1-arrows]\label{defn:cocart-1-arrow}
  Suppose that $p\colon E\tfib B$ is an isofibration in $\eK$ and that
  $\chi\colon e\to e'$ is a $1$-arrow from some object $X$ to $E$. We say 
  that $\chi$ is a \emph{$p$-cocartesian $1$-arrow\/} just when $\chi$ is a cocartesian arrow for the isofibration
  $\Fun_{\eK}(X,p)\colon\Fun_{\eK}(X,E)\tfib \Fun_{\eK}(X,B)$ of
  quasi-categories. 
\end{defn}

Comparing Definition \ref{defn:cocart-1-arrow} with Lemma
\ref{lem:cocart-cyl-are-cocart-arrows}, we see that in the $\infty$-cosmos of
quasi-categories, a 1-arrow $\chi \colon e \to e'$ in $\Fun_{\qCat}(\qX,\qE)
\cong \qE^{\qX}$ is $p$-cocartesian if and only if it defines a pointwise
$p$-cocartesian cylinder, in the sense of Definition
\ref{defn:cocartesian-cylinder}.

\begin{lem}[{\refIV{lem:rep-cart-cells}}]\label{lem:2-categorical-connection}
  Fixing a cocartesian fibration $p \colon E \tfib B$ an $\infty$-cosmos $\eK$,
  a 1-arrow $\chi \colon e \to e'$ in $\Fun_{\eK}(X,E)$ is a $p$-cocartesian
  1-arrow if and only if it represents a $p$-cocartesian 2-cell in the sense of
  Definition \refIV{defn:weak.cart.2-cell}.
\end{lem}

Under this notational simplification the characterisation of Observation
\ref{obs:representable-cocart} can be restated much more compactly, recovering
the original definition of cocartesian fibration appearing in the dual as
Definition \refIV{defn:cart-fib}.
  
  \begin{cor}\label{cor:representable-cocart} An isofibration
  $p\colon E\tfib B$ in $\eK$ is a cocartesian fibration if and only if
\begin{itemize}
\item  $1$-arrows with codomain $B$ admit $p$-cocartesian lift  with a specified domain and 
\item $p$-cocartesian $1$-arrows are stable under precomposition by $0$-arrows.
\end{itemize}
\end{cor}

Via Lemma \ref{lem:2-categorical-connection}, we have the following stability properties for $p$-cocartesian 1-arrows.

\begin{lem}[{\refIV{lem:cart-arrows-compose}, \refIV{lem:cart-arrows-cancel}, \refIV{obs:weak.cart.conserv}}]\label{lem:cocart-closures}
Let $p \colon E \tfib B$ be a cocartesian fibration in $\eK$ and consider a 2-arrow 
\[ \vcenter{\xymatrix@R=1em{ & e_1 \ar[dr]^{\alpha\vert_{1,2}} \ar@{}[d]|(.6){\alpha} & \\ e_0 \ar[rr]_{\alpha\vert_{0,2}} \ar[ur]^{\alpha\vert_{0,1}} & & e_2}}\qquad \in \qquad \Fun_{\eK}(X,E)
\] from $X$ to $E$.
\begin{enumerate}[label=(\roman*)]
\item\label{itm:cocart-compose} If the edges $\alpha\vert_{0,1}$ and $\alpha\vert_{1,2}$ are $p$-cocartesian 1-arrows, then the edge $\alpha\vert_{0,2}$ is a $p$-cocartesian 1-arrow; that is, $p$-cocartesian 1-arrows compose.
\item\label{itm:cocart-cancel}  If the edges $\alpha\vert_{0,1}$ and $\alpha\vert_{0,2}$ are $p$-cocartesian 1-arrows, then the edge $\alpha\vert_{1,2}$ is a $p$-cocartesian 1-arrow; that is, $p$-cocartesian 1-arrows cancel on the right.
\item\label{itm:cocart-conserv} If the edges  $\alpha\vert_{0,1}$ and $\alpha\vert_{0,2}$ are $p$-cocartesian 1-arrows and $p\alpha\vert_{1,2}$ is an isomorphism in $\Fun_{\eK}(X,B)$, then $\alpha\vert_{1,2}$ is an isomorphism in $\Fun_{\eK}(X,E)$.
\end{enumerate}
\end{lem}

\begin{lem}\label{lem:cocartesian-interchange}
  Suppose that we are given a commutative diagram
  \begin{equation*}
    \xymatrix@R=2.5em@C=5em{
      {F}\ar@/^1.5ex/[r]^{g}_{}="1"\ar@/_1.5ex/[r]_{k}^{}="2"\ar@{->>}[d]_-{q} &
      {E}\ar@{->>}[d]^-{p} \\
      {C}\ar@/^1.5ex/[r]^{f}_{}="3"\ar@/_1.5ex/[r]_{h}^{}="4" & {B}
      \ar@{=>} "1";"2" ^{\chi} \ar@{=>} "3";"4" ^{\alpha}
    }
  \end{equation*}
  of $0$-arrows and $1$-arrows in $\eK$ in which $\chi$ is a
  $p$-cocartesian $1$-arrow. If the pair $(g,f)$ is a cocartesian functor,
  then the pair $(k,h)$ is also a cocartesian functor.
\end{lem}

\begin{proof}
We must show that whiskering with $k$ carries a $q$-cocartesian 1-arrow $\gamma \colon u \To v \colon X \to F$ to a $p$-cocartesian 1-arrow. The horizontal composition of a pair of 1-arrows, in this case $\gamma$ and $\chi$, provides a ``middle-four interchange'' square $\Del^1 \times \Del^1 \to \Fun_{\eK}(X,E)$
\[
\xymatrix{ gu \ar[r]^-{g\gamma} \ar[d]_{\chi u} \ar[dr]^{\chi\gamma} & gv \ar[d]^{\chi v} \\ ku \ar[r]_{k\gamma} & kv}\] If $\gamma$ is $q$-cocartesian, then $g\gamma$ is $p$-cocartesian, since $(g,f)$ is a cartesian functor. Since $\chi$ is $p$-cocartesian, then $\chi u$ and $\chi v$ are also $p$-cocartesian by Corollary \ref{cor:representable-cocart}. Now Lemma \ref{lem:cocart-closures} tells us that $\chi\gamma$ and hence also $k\gamma$ are also $p$-cocartesian, 
 and so $(k,h)$ is a cartesian functor as required.
\end{proof}

\begin{lem}\label{lem:functor-preservation-reprise}
Functors of $\infty$-cosmoi preserve cocartesian 1-arrows.
\end{lem}
\begin{proof}
Consider a cocartesian fibration $p \colon E \tfib B$ in $\eK$, a $p$-cocartesian 1-arrow $\chi \colon e \to e'$ in $\Fun_{\eK}(X,E)$, and a functor of $\infty$-cosmoi $G \colon \eK \to \eL$. By Definition \ref{defn:qcat-cocart-edge} this means that $\chi$ is in the essential image of the top-horizontal functor:
  \[ \xymatrix@C=50pt{ & \Del^0 \ar[d]^\chi \\   \Fun_{\eK}(X,p \comma B) \ar[r]^-{\Fun_{\eK}(X,\bar{\ell})}  \ar[d]_G &\Fun_{\eK}(X,E^\cattwo)\cong \Fun_{\eK}(X,E)^\cattwo \ar[d]^G\\ 
 \Fun_{\eL}(GX,Gp \comma GB)  \ar[r]^-*+{\labelstyle\Fun_{\eL}(GX,G\bar{\ell})}& \Fun_{\eL}(GX,GE^\cattwo)\cong \Fun_{\eL}(GX,GE)^\cattwo}\] The 1-arrow $G\chi$ is picked out by the right-hand vertical composite map. Since the $\infty$-cosmos functor $G$ commutes with the limits involved in the comma and arrow constructions, we see that $G\chi$ must be in the essential image of $\Fun_{\eL}(GX,G\bar{\ell})$ and so defines a $Gp$-cocartesian 1-arrow as claimed.  
  \end{proof}

Finally we note that pullback squares provide an important source of cartesian functors that create cocartesian 1-arrows

\begin{prop}[{\refIV{prop:cart-fib-pullback}}]\label{prop:cart-fib-pullback}
Consider a  pullback
\[ \xymatrix{ F \pbexcursion \ar@{->>}[d]_q \ar[r]^g & E \ar@{->>}[d]^p \\ A \ar[r]_f & B}\] in $\eK$. If $p \colon E \tfib B$ is a cocartesian fibration, then $q \colon F \tfib A$ is a cocartesian fibration, and $g\colon F\to E$
  preserves and reflects cocartesian $1$-arrows, in the sense that a $1$-arrow
  $\chi$ in $\Fun_{\eK}(X,F)$ is $q$-cocartesian if and only if the
  whiskered $1$-arrow $g\chi$ in $\Fun_{\eK}(X,E)$ is $p$-cocartesian. In particular, the pullback square defines a cartesian functor.
\end{prop}


%% file: computads.tex

\section{Simplicial computads and homotopy coherent realisation}\label{sec:coherent-nerve}

Many interesting examples of large quasi-categories arise as homotopy coherent
nerves of Kan complex enriched categories. In this section we develop tools that
will be used in \S\ref{sec:cocones} and \S\ref{sec:comprehension} to construct
the comprehension functor of Theorem \ref{thm:general-comprehension}.

We can probe the homotopy coherent nerve of a Kan complex enriched category by
making use of the \emph{homotopy coherent realisation functor}
\begin{equation}\label{eq:hocoh-realization-adj}
  \adjdisplay \gC -| \nrvhc : \sSet\text{-}\Cat -> \sSet.
\end{equation}
which is left adjoint to the homotopy coherent nerve functor. Importantly, this
left adjoint lands in the subcategory of \emph{simplicial computads}, a class of
``freely generated'' simplicial categories that are precisely the cofibrant
objects in the model structure due to Bergner \cite{Bergner:2007fk}. In
practice, this means that a simplicial functor whose domain is of the form
$\gC{X}$ for a simplicial set $X$ can be defined with relatively little data,
which we enumerate in Lemma \ref{lem:computad-extension}.

In \S\ref{sec:computad}, we review the theory of \emph{simplicial computads}
from \S\refII{sec:computads} and define two important classes of functors
between them. The prototypical example of a simplicial computad, introduced in
\S\ref{sec:htpy-coh-simplex}, is the \emph{homotopy coherent $\omega$-simplex},
from which we extract a cosimplicial object of simplicial computads referred to
as the \emph{homotopy coherent $n$-simplices}. This data determines the adjoint
pair of functors \eqref{eq:hocoh-realization-adj}, as explicitly described in
\S\ref{sec:realization-nerve}.

The simplicial category $\gC{X}$ defined as the homotopy coherent realisation of
a simplicial set $X$ is a simplicial computad. In
\S\ref{sec:realization-computads}, we analyse particular examples of this
construction, building towards a general description of the simplicial computad
structure of $\gC{X}$ in Proposition \ref{prop:gothic-C}, which recovers results
of Dugger-Spivak \cite{Dugger:2011ro} by a different route.

Finally, in \S\ref{sec:relative-computad} we develop the relative case,
enumerating the data required to extend a homotopy coherent diagram along the
\emph{simplicial subcomputad inclusion} $\gC{X}\inc\gC{Y}$ arising from an
inclusion of simplicial sets.

\subsection{Simplicial computads}\label{sec:computad}

Before pressing on to simplicial computads, we should first clarify the sense in which we shall use a couple of commonly abused categorical terms:

\begin{defn}[replete subcategories]\label{defn:replete}
  We say that a subcategory $\eA\subseteq\eB$ is \emph{replete\/} if it is
  closed under isomorphisms in $\eB$, in the sense that if $A$ in $\eA$ then any
  isomorphism $\alpha\colon A\cong B$ of $\eB$ is contained in the sub-category
  $\eA$. This is equivalent to postulating that the inclusion functor
  $\eA\inc\eB$ is an isofibration.
\end{defn}
  
\begin{defn}[closed subcategories]\label{defn:closed}
  We shall also say that the subcategory $\eA$ is \emph{closed\/} in $\eB$ under
  some universally defined structure (such as limits or colimits) if it possesses a
  structure of that kind which is preserved by the inclusion $\eA\inc\eB$. This
  is equivalent to saying that $\eB$ possesses \emph{some} structure of that
  kind which is contained in $\eA$ and which possesses the given universal
  property in that subcategory.
\end{defn}

\begin{obs}\label{obs:replete-closed}
  When a replete subcategory $\eA$ is closed in its super-category $\eB$ under
  some universally defined structure then it satisfies the stronger property
  that \emph{every} structure of that kind in $\eB$ is contained in $\eA$ and
  possesses the given universal property in that subcategory. To demonstrate
  this fact we argue as follows: given such a structure in $\eB$ we can pick
  some such structure in $\eA$ which also satisfies the given universal property
  in $\eB$. These two structures are uniquely isomorphic in $\eB$ and one end of
  that isomorphism is in $\eA$, but $\eA$ is replete in $\eB$ so we may infer
  that this isomorphism itself is in $\eA$. Consequently the originally selected
  structure is also in $\eA$ wherein it is isomorphic to a structure that has
  the given universal property in there, so it too must enjoy that property in
  $\eA$.
\end{obs}

\begin{defn}[atomic arrows and computads]
  A $n$-arrow $f$ in a category $\eC$ is \emph{atomic} if it is not an identity
  and it admits no non-trivial factorisations, i.e., if whenever $f=g \circ h$
  then one or other of $g$ and $h$ is an identity. We say that a category is a
  \emph{computad} if and only if each of its arrows may be expressed uniquely as
  a (possibly empty) composite of atomic arrows. Here we adopt the standard
  convention of defining the empty composite of arrows leading from an object to
  itself to be the identity on that object. Equally we could say that every
  non-identity arrow may be uniquely expressed as a (non-empty) composite of
  atomic arrows that identities admit no non-trivial factorisations. A category
  is a computad if and only if it is isomorphic to the free category generated
  by a reflexive directed graph: namely its sub-graph of atomic arrows.
\end{defn}

\begin{defn}[computad morphisms]\label{defn:computad}
  A functor $F\colon\eA\to\eB$ between computads $\eA$ and $\eB$ is a
  \emph{computad morphism\/} if it carries each atomic arrow $f$ of $\eA$ to an
  arrow $Ff$ which is either atomic or an identity in $\eB$.   A computad morphism $\eA\inc\eB$ that is injective on objects and fully  faithful is said to display $\eA$ as a \emph{subcomputad\/} of $\eB$.
\end{defn}

\begin{rmk}\label{rmk:cptd-stuff}
  Let $\Cptd$ denote the subcategory of $\Cat$ whose objects are computads and
  whose arrows are computad morphisms and let $\Graph$ denote the  presheaf
  category of reflexive graphs and graph morphisms. Then the free category
  functor $\free\colon \Graph\to\Cat$ factors through the inclusion
  $\Cptd\inc\Cat$ to give an equivalence $\Graph\simeq\Cptd$.

  Any isomorphism of categories preserves and reflects atomic arrows, so it
  follows that any category isomorphic to a computad in $\Cat$ is itself a
  computad and that any isomorphism in $\Cat$ between computads is a computad
  morphism; consequently $\Cptd$ is a replete subcategory of $\Cat$. The simplicial subcomputad inclusions are precisely the monomorphisms in the category $\Cptd$. 

We also note that the
  inclusion $\Cptd\inc\Cat$ is left adjoint, since it corresponds to the left
  adjoint free category functor under the equivalence $\Graph\simeq\Cptd$; in
  particular $\Cptd$ is closed in $\Cat$ under colimits.
\end{rmk}

\begin{defn}[simplicial computad]\label{defn:simplicial-computad}
  A simplicial category $\eA$ is a \emph{simplicial computad} if and only if:
  \begin{itemize}
  \item each category $\eA_n$ of $n$-arrows is freely generated by the graph of
    atomic $n$-arrows
  \item if $f$ is an atomic $n$-arrow in $\eA_n$ and $\alpha\colon [m]\to[n]$ is
    a degeneracy operator in $\Del$ then the degenerated $m$-arrow
    $f\cdot\alpha$ is atomic in $\eA_m$.
  \end{itemize}
  By the Eilenberg-Zilber lemma, $\eA$ is a simplicial computad if and only if
  all of its non-identity arrows $f$ can be expressed uniquely as a composite
  \begin{equation}\label{eq:computad-arrow-decomp}
    f = (f_1 \cdot \alpha_1) \circ (f_2 \cdot \alpha_2) \circ \cdots \circ
    (f_\ell \cdot \alpha_\ell)
  \end{equation}
  in which each $f_i$ is non-degenerate and atomic and each $\alpha_i\in\Del$ is
  a degeneracy operator.
\end{defn}

The simplicial computads are precisely the cofibrant simplicial categories in
the Bergner model structure \cite[\S 16.2]{Riehl:2014kx}.

\begin{defn}[simplicial computad morphism] Let $\eA$ and $\eB$ be simplicial
  computads.
  \begin{enumerate}[label=(\roman*)]
  \item A simplicial functor $F\colon\eA\to\eB$ is a \emph{simplicial computad
      morphism\/} if it maps every atomic arrow $f$ in $\eA$ to an arrow $Ff$
    which is either atomic or an identity in $\eB$.
  \item A simplicial computad morphism $\eA\inc\eB$ that is injective on objects
    and faithful displays $\eA$ as a \emph{simplicial subcomputad} of $\eB$.
  \end{enumerate}
  When give a simplicial subcomputad $\eA\inc\eB$ we shall usually identify
  $\eA$ with the simplicial subcategory of $\eB$ that is its image under that
  inclusion. We shall write $\sCptd$ for the non-full subcategory of $\sCat$ of
  simplicial computads and their morphisms.
\end{defn}

\begin{rmk}\label{rmk:degen-restr}
  Let $\Del\ep$ denote the subcategory of degeneracy operators in $\Del$, then
  restriction along the inclusion $\Del\ep\subset\Del$ gives rise to a functor
  $\Cat^{\Del\op}\to\Cat^{\Del\ep\op}$ which is faithful, injective on objects,
  and an isofibration. From hereon we shall identify $\Cat^{\Del\op}$ with the
  replete subcategory comprising the image of that functor in
  $\Cat^{\Del\ep\op}$ and note, in particular, that $\Cat^{\Del\op}$ is closed
  in $\Cat^{\Del\ep\op}$ under limits and colimits. Combining this with
  Recollection~\ref{rec:simp-cat} we also obtain a presentation of $\sCat$ as a
  subcategory of $\Cat^{\Del\ep\op}$ which is again closed in there under limits
  and colimits.
\end{rmk}

\begin{lem}\label{lem:computad-colimits}
  The category of simplicial computads $\sCptd$ is canonically isomorphic to,
  and thus may be identified with, the intersection of the sub-categories
  $\sCat$ and $\Cptd^{\Del\ep\op}$ of $\Cat^{\Del\ep\op}$. Both of these
  subcategories are closed under  colimits in $\Cat^{\Del\ep\op}$ and
  $\Cptd^{\Del\ep\op}$ is replete in there, so it follows that $\sCptd$ is closed
  under  colimits in $\sCat$ and in $\Cptd^{\Del\ep\op}$.
\end{lem}

\begin{proof}
  Suppose that $\eC$ is a simplicial category and that $\eC_\bullet\colon\Del\op
  \to \Cat$ denotes its presentation as a corresponding object of
  $\Cat^{\Del\op}$ as described in Recollection~\ref{rec:simp-cat}. It is then
  straightforward to see that the simplicial computad conditions of
  Definition~\ref{defn:simplicial-computad} are equivalent to the postulate that
  there exists a, necessarily unique, dotted functor making the following square
  commute:
  \begin{equation*}
    \xymatrix@=2em{
      {\Del\ep\op}\ar@{^(->}[d]_{\subset}\ar@{.>}[r]^-{\exists!} &
      {\Cptd}\ar@{^(->}[d] \\
      {\Del\op}\ar[r]_{\eC_\bullet} & {\Cat}
    }
  \end{equation*}
  Correspondingly, a simplicial functor $F\colon \eB\to \eC$ is a simplicial
  computad morphism if and only if the corresponding natural transformation
  $F_\bullet\colon \eB_\bullet\to \eC_\bullet$ when restricted to $\Del\ep
  \subset \Del$ also factors, necessarily uniquely, through the inclusion
  $\Cptd\inc\Cat$. These facts clearly verify that $\sCptd$ may indeed be
  presented as the intersection of the subcategories $\sCat$ and
  $\Cptd^{\Del\ep\op}$ in $\Cat^{\Del\ep\op}$ as stated.

  To show that $\sCptd$ has all colimits, suppose that $D$ is a diagram in
  $\sCptd = \sCat\cap\Cptd^{\Del\ep\op}\subset\Cat^{\Del\ep\op}$ and form its
  colimit in the super-category $\sCat$. Remarks~\ref{rmk:degen-restr} and
  Recollection~\ref{rec:simp-cat} serve to demonstrate that the sub-category
  $\sCat$ is closed under colimits in $\Cat^{\Del\ep\op}$, so it follows that
  the given colimit cocone in $\sCat$ also determines a colimit in the
  super-category $\Cat^{\Del\ep\op}$. Remark~\ref{rmk:cptd-stuff} implies that
  the sub-category $\Cptd^{\Del\ep\op}$ is both replete and closed under
  colimits in $\Cat^{\Del\ep\op}$ so, by Observation~\ref{obs:replete-closed},
  it follows that our colimit cocone in $\Cat^{\Del\ep\op}$ is contained in
  $\Cptd^{\Del\op\ep}$ and that it also displays a colimit in there.
  Consequently we have shown that our cocone is in the sub-category $\sCptd =
  \sCat\cap\Cptd^{\Del\ep\op}$ and it is easy to demonstrate that it possesses
  the required universal property in there using the fact that it has that
  property in each of the sub-categories $\sCat$ and $\Cptd^{\Del\ep\op}$.
\end{proof}

\begin{obs}\label{obs:simp-subcomp-char}
  It is clear now that a simplicial computad morphism $F\colon\eA\to\eB$ is
  injective on objects and faithful if and only if its presentation as a natural
  transformation $F_\bullet\colon\eA_\bullet\to\eB_\bullet$ in
  $\Cptd^{\Del\op\ep}$ has components that are subcomputad inclusions. By
  Remark~\ref{rmk:cptd-stuff} this latter condition is equivalent to postulating
  that the components of $F_\bullet$ are monomorphisms in $\Cptd$ or equally
  that $F_\bullet$ is a monomorphism in $\Cptd^{\Del\ep\op}$. This characterises
  the simplicial subcomputad inclusions among the simplicial computad morphisms,
  and leads us directly to the following important result:
\end{obs}
    
\begin{lem}\label{lem:sub-computad-stability} Simplicial subcomputads are
  stable under pushout, coproduct, and colimit of countable sequences in
  $\sCptd$.
\end{lem}

\begin{proof}
  The subcategory $\sCptd$ is closed in $\Cptd^{\Del\op\ep}$ under colimits, so
  in particular it is closed under the colimit types listed in the statement.
  Furthermore, the last observation tells us that simplicial subcomputad
  inclusions are precisely those morphisms of $\sCptd$ that are monomorphisms
  when regarded as maps of $\Cptd^{\Del\op\ep}$; it follows that the result in
  the statement would follow from the corresponding result for monomorphisms in
  $\Cptd^{\Del\op\ep}$. That latter category is equivalent to the presheaf
  category $\Graph^{\Del\op\ep}$, in which colimits and monomorphisms are
  determined pointwise in $\Set$ so the desired result follows from the fact
  that it clearly holds for monomorphisms in $\Set$.
\end{proof}

\subsection{The homotopy coherent \texorpdfstring{$\omega$}{omega}-simplex}\label{sec:htpy-coh-simplex}

The ur-example of a simplicial computad is the homotopy coherent
$\omega$-simplex, that being the simplicial category defined as the hom-wise
nerve of the poset enriched category that we now introduce:

\begin{defn}[the $\omega$-simplex]
  Let $\oSimp$ denote the locally ordered (partially ordered set enriched)
  category with:
  \begin{itemize}
  \item \textbf{objects} the natural numbers $0,1,2,\ldots$,
  \item \textbf{hom-sets} given by
    \begin{equation*}
      \Hom_{\oSimp}(k,l) \defeq
      \begin{cases}
        \emptyset & \text{if $k > l$,} \\
         \{ T \subseteq [k,l] \mid k,l\in T \} & \text{if $k \leq l$}
      \end{cases}
    \end{equation*}
    where $[k,l]$ denotes the set of integers $\{k,k+1,\ldots,l\}$, ordered by inclusion, 
  \item \textbf{composition} of $T_1\in \Hom_{\oSimp}(k,l)$ and
    $T_2\in\Hom_{\oSimp}(l,m)$ is written in the {\em natural order\/} and is
    given by union $T_1\circ T_2\defeq T_1\cup T_2$, and the
    \item \textbf{identity} arrow on the object $k$ of $\oSimp$ is the
  singleton $\{k\}$. 
  \end{itemize}
  \end{defn}
  
\begin{obs}[atomic arrows in $\oSimp$]
Note that $T\in\Hom_{\oSimp}(k,m)$ can be expressed as
  a composite $T_1\circ T_2$ with $T_1\in\Hom_{\oSimp}(k,l)$ and
  $T_2\in\Hom_{\oSimp}(l,m)$ if and only if $l\in T$, in which case $T_1 =
  T\cap[k,l]$ and $T_2= T\cap[l,m]$. So an arrow $T\in\Hom_{\oSimp}(k,m)$ may
  be decomposed as a composite of non-identity arrows if and only if the set
  $T\setminus\{k,m\}$ is non-empty. It follows that for each pair $l<m$ there is precisely one
  atomic arrow $\{l,m\}$ in the hom-set $\Hom_{\oSimp}(l,m)$. It is
  also clear that any non-identity arrow
  $\{l=k_0<k_1<k_2<\cdots<k_n=m\}\in\Hom_{\oSimp}(l,m)$ may be decomposed
  \textit{uniquely\/} as a composite $\{k_0,k_1\} \circ \{k_1,k_2\} \circ \cdots
  \circ \{k_{n-1},k_n\}$ of atomic arrows.

  Consequently we see that, as a mere category, $\oSimp$ is the freely
  generated category on the reflexive graph whose vertices are the integers and
  which has a unique edge non-identity edge $\{k,l\}$ from $k$ to $l$ whenever
  $k<l$. In other words, it is just the free category on the ordinal $\omega$
  regarded as a reflexive graph. As a locally ordered category it is freely
  generated by the same reflexive graph of atomic arrows and the
  primitive inclusions \[\{k,m\} \subset \{k,l\}\circ\{l,m\}.\]
\end{obs}

\begin{defn}[the homotopy coherent $\omega$-simplex]
The \textit{homotopy coherent $\omega$-simplex} is the simplicial category $\hcSimp$ obtained by  taking the nerve of each hom-set of $\oSimp$.   Explicitly:
  \begin{itemize}
  \item Its \textbf{objects} are the natural numbers $0,1,2,\ldots$.
\item An $r$-\textbf{arrow}, that is to say an $r$-simplex in the function complex
  $\Fun_{\hcSimp}(k,l)$, is an order preserving map
  $T^\bullet\colon[r]\to\Hom_{\oSimp}(k,l)$, in other words an ordered chain
  of subsets \[T^0\subseteq T^1\subseteq\cdots\subseteq T^r\] in
  $\Hom_{\oSimp}(k,l)$ of length $(r+1)$.
  \item \textbf{Composition} of such $r$-arrows is given pointwise, that is $(S^\bullet\circ
  T^\bullet)^i\defeq S^i\circ T^i = S^i\cup T^i$.
  \end{itemize}
  \end{defn}

\begin{ntn}[compact notation for arrows in $\hcSimp$]\label{ntn:compact-arrows}
We will sometimes depict an $r$-arrow $T^\bullet$ in the function complex $\Fun_{\hcSimp}(k,l)$, with $k\leq l$, using 
a more compact notation of the
  form \[\langle I_0\mid I_1\mid \ldots \mid I_r \rangle\] where
  \begin{itemize}
  \item $I_0\defeq T^0$ and
  \item  $I_i \defeq T^i\setminus T^{i-1}$ for $i=1,\ldots,r$.
  \end{itemize}
   So, for example, the
  expression $\langle 0,3,5 \mid 4 \mid {} \mid 1,2 \rangle$ denotes the
  degenerate $3$-arrow
  \begin{equation*}
    \{0,3,5\} \subseteq \{0,3,4,5\} \subseteq \{0,3,4,5\}\subseteq 
    \{0,1,2,3,4,5\}
  \end{equation*}
  in $\Fun_{\hcSimp}(0,5)$.

  An expression of the form $\langle {I_0\mid I_1\mid\ldots\mid I_r} \rangle$
  represents a unique $r$-arrow in $\Fun_{\hcSimp}(k,l)$ if and only if the
  $I_i$ are pairwise disjoint subsets of $[k,l]$ with $k,l\in I_0$. It is atomic
  if and only if $I_0=\{k,l\}$ and it is non-degenerate if and only if
  $I_i\neq\emptyset$ for $i=1,\ldots,r$.

  In this notation, the composite of an $r$-arrow $\langle I_0\mid I_1\mid \ldots \mid I_r \rangle$ 
 in $\Fun_{\hcSimp}(k,l)$ with an $r$-arrow $\langle J_0\mid J_1\mid \ldots\mid J_r
  \rangle$ in $\Fun_{\hcSimp}(l,m)$ is 
  denoted by $\langle I_0\cup J_0\mid I_1\cup J_1\mid \ldots\mid I_r\cup J_r
  \rangle$.   The identity $0$-arrow on an object $n$ is denoted $\langle n
  \rangle$.   This notation also makes it convenient to write whiskered composites, as these are simply given by the
  expressions $\langle {I_0} \rangle\circ\langle J_0\mid \ldots\mid J_r \rangle =
  \langle I_0\cup J_0\mid J_1\mid \ldots\mid J_r \rangle$ and $\langle I_0\mid
  \ldots\mid I_r \rangle\circ \langle {J_0} \rangle = \langle I_0\cup J_0\mid
  I_1\mid \ldots\mid I_r \rangle$.
\end{ntn}

Importantly: 

\begin{lem}\label{lem:simplex-computad} The homotopy coherent simplex $\hcSimp$ is a simplicial computad.
\end{lem}
\begin{proof}
We claim that a non-identity $r$-arrow $T^\bullet$ in $\Fun_{\hcSimp}(k,m)$
  may be decomposed into a composite of two non-identity $r$-arrows if and only
  if $T^0\setminus \{k,m\}$ is non-empty. This follows because if we pick an
  $l\in T^0\setminus \{k,m\}$ then it is also an element of each
  $T^i\setminus\{k,m\}$, so we may decompose each $T^i$ as $T_1^i\circ T_2^i$
  with $T_1^i\defeq T^i\cap[k,l]\in\Fun_{\oSimp}(k,l)$ and $T_2^i\defeq
  T^i\cap[l,m]\in\Fun_{\oSimp}(l,m)$, and in this way we get $r$-simplices
  $T_1^\bullet$ and $T_2^\bullet$ with $T^\bullet = T_1^\bullet\circ
  T_2^\bullet$. Consequently, we see that $T^\bullet$ is atomic in
  $\hcSimp$ if and only if $T^0$ is atomic in $\oSimp$. Furthermore, if
  $T^0 = \{k = l_0 < l_1 < \cdots < l_n = m\}$ then $T^\bullet$ decomposes into a
  unique composite of atomic $r$-arrows \[k=l_0\stackrel{T^\bullet_1}
  \longrightarrow l_1\stackrel{T^\bullet_2} \longrightarrow l_2\longrightarrow \cdots
 \longrightarrow l_{n-1}\stackrel{T^\bullet_n} \longrightarrow l_n\] with $T^i_j =
  T^i\cap[l_{j-1},l_j]$. It follows, therefore, that the category $\hcSimp_r$
  of the $r$-arrows of $\hcSimp$ is freely generated by the graph with objects
  the integers and edges the atomic $r$-arrows.

  A simplicial operator $\alpha\colon[s]\to[r]$ acts on an $r$-arrow $T^\bullet$
  of $\hcSimp$ by re-indexing, that is $(T^\bullet\cdot\alpha)^i\defeq
  T^{\alpha(i)}$. In particular, if $\alpha(0)=0$ and 
$T^\bullet$ is an atomic $r$-arrow, then $T^\bullet\cdot{\alpha}$ is also an atomic $r$-arrow.
  This condition holds whenever $\alpha$ is a degeneracy
  operator, so we have verified the conditions of Definition \ref{defn:simplicial-computad}.
 \end{proof}

The upshot of Lemma \ref{lem:simplex-computad} is that a simplicial functor with domain $\hcSimp$ is defined simply by
  specifying how it acts on non-degenerate atomic arrows at each dimension
  and verifying that these choices are compatible with the face operations on
  those arrows.

\begin{obs}[the function complexes of $\hcSimp$ are cubes]\label{obs:hcsimp-cubes}
  Observe that we have $\Fun_{\hcSimp}(k,k)=\Del^0$ and $\Fun_{\hcSimp}(k,l) =
  \emptyset$ whenever $k > l$. What is more, each function complex
  $\Fun_{\hcSimp}(k,l)$ with $k<l$ is isomorphic to the standard simplicial cube
  $(\Del^1)^{\times (l-k-1)}$, by an isomorphism
  \[ \Fun_{\hcSimp}(k,l) \cong (\Del^1)^{\times (l-k-1)}\]
   which maps a vertex $T \subset [k,l]$ of
  $\Fun_{\hcSimp}(k,l)$ to the vertex $(x_{k+1},\ldots,x_{l-1})$ of
  $(\Del^1)^{\times (l-k-1)}$ given by $x_i = 1$ if $i\in T$ and $x_i= 0$ if
  $i\notin T$. Under these isomorphisms, the composition operation  corresponds to the simplicial map 
  \[
  \xymatrix@R=1em{ \Fun_{\hcSimp}(k,l)\times\Fun_{\hcSimp}(l,m)\ar[r]^-\circ \ar@{}[d]|{\rotatebox{90}{$\cong$}}& \Fun_{\hcSimp}(k,m) \ar@{}[d]|{\rotatebox{90}{$\cong$}} \\
(\Del^1)^{\times
    (l-k-1)}\times(\Del^1)^{\times (m-k-1)}\ar[r] & (\Del^1)^{\times (m-l-1)}}\]
     which
  maps the pair of vertices $(x_{k+1},\ldots,x_{l-1})$ and $(x_{l+1},\ldots,x_{m-1})$
  to the vertex \[(x_{k+1},\ldots,x_{l-1},1,x_{l+1},\ldots,x_{m-1}).\]
\end{obs}

\subsection{Homotopy coherent realisation and the homotopy coherent nerve}\label{sec:realization-nerve}

The full subcategory of the homotopy coherent simplex $\hcSimp$ on the objects
$0,\ldots, n$ defines the \emph{homotopy coherent $n$-simplex}. These simplicial
computads assemble into a cosimplicial object in $\sCptd$ and that, in turn,
gives rise to the homotopy coherent nerve and its left adjoint, called homotopy
coherent realisation, by an application of a standard construction due to Kan.

\begin{defn}[the homotopy coherent $n$-simplex]\label{defn:hocoh-n-simplex}
  Let $\oSimp^{n}$ denote the full subcategory of $\oSimp$ on the objects
  $0,1,\ldots,n$ and let $\gC\Del^n$, the \textit{homotopy coherent
    $n$-simplex}, denote the corresponding full simplicial subcategory of
  $\hcSimp$. 
    
  A simplicial operator $\beta\colon[m]\to[n]$ gives rise to a local order
  preserving functor $\oSimp^\beta\colon \oSimp^{m}\to \oSimp^{n}$ which maps
  the object $k\in\oSimp^{m}$ to the object $\beta(k)\in\oSimp^{n}$ and carries
  an arrow $T$ of $\Hom_{\oSimp}(k,l)$ to $\beta(T)=\{\beta(i)\mid i\in T\}$
  which is quite clearly an arrow of $\Hom_{\oSimp^{n}}(\beta(k),\beta(l))$.
  Taking nerves of hom-posets it follows that $\oSimp^\beta\colon\oSimp^m\to
  \oSimp^n$ gives rise to a simplicial functor
  $\gC{\Del^{\beta}}\colon\gC{\Del^{m}}\to\gC{\Del^{n}}$; this acts pointwise
  to carry an $r$-arrow $T^\bullet\in\Fun_{\gC{\Del^{m}}}(k,l)$ to the $r$-arrow
  $\beta(T^\bullet)$ in $\Fun_{\gC{\Del^{n}}}(\beta(k),\beta(l))$ given by
  $\beta(T^\bullet)^i \defeq \beta(T^i)=\{\beta(j)\mid j\in T^i\}$ for
  $i=0,1,\ldots,r$. These constructions are functorial in $\beta$, giving us a
  functor
  \begin{equation*}
    \gC{\Del^\bullet}\colon \Del \longrightarrow \sCat
  \end{equation*}
  into the category of small simplicial categories and simplicial
  functors.
  \end{defn}
  
  \begin{defn}[homotopy coherent realization and the homotopy coherent nerve]
    Applying Kan's construction \cite[1.5.1]{Riehl:2014kx} to the functor
    $\gC{\Del^\bullet}\colon \Del \longrightarrow \sCat$ yields an adjunction
  \begin{equation*}
    \adjdisplay \gC -| \hN : \sCat -> \sSet .
  \end{equation*}
  the right adjoint of which is called the \textit{homotopy coherent nerve} and
  the left adjoint of which, defined by pointwise left Kan extension along the
  Yoneda embedding:
  \begin{equation}\label{eq:coherent-realization-kan-ext}
    \xymatrix@=1.5em{
      {\Del}\ar[rr]^{\Del^{\bullet}}\ar[dr]_{\gC{\Del^\bullet}} &
      \ar@{}[d]|(0.4){\cong} & {\sSet}\ar[dl]^{\gC} \\
      & {\sCat} &
    }
  \end{equation}
 we refer to as \emph{homotopy coherent realisation}. An $n$-simplex of a
  homotopy coherent nerve $\hN\eC$ is simply a simplicial functor
  $c\colon\gC{\Del^{n}}\to \eC$ and the action of an operator
  $\beta\colon[m]\to[n]$ on that simplex is given by precomposition with 
$\gC{\Del^{\beta}}\colon\gC{\Del^{m}}\to\gC{\Del^{n}}$.
\end{defn}

\begin{rmk}[homotopy coherent simplices explicitly]\label{rmk:explicit-hc-simp}
The simplices $c\colon\gC{\Del^n}\to\eC$ of the
  homotopy coherent nerve $\hN\eC$ are called \emph{homotopy coherent $n$-simplices in
    $\eC$}. These are given by specifying the following information:
  \begin{itemize}
  \item a sequence of objects $c_0, c_1, \ldots, c_n$ of $\eC$,
  \item simplicial maps $c_{i,j}\colon(\Del^1)^{j-i-1}\to\Fun_{\eC}(c_i,c_j)$ for
    $0\leq i < j\leq n$, satisfying the
  \item functoriality condition that for all $0\leq i < j < k\leq n$ the
    following square
    \begin{equation*}
      \xymatrix@R=2em@C=5em{
        {(\Del^1)^{j-i-1}\times(\Del^1)^{k-j-1}}\ar@{^(->}[d]
        \ar[r]^-{c_{i,j}\times c_{j,k}} &
        {\Fun_{\eC}(c_i,c_j)\times\Fun_{\eC}(c_j,c_k)}
        \ar[d]^{\circ} \\
        {(\Del^1)^{k-i-1}}\ar[r]_-{c_{i,k}} & {\Fun_{\eC}(c_i,c_k)}
      }
    \end{equation*}
    commutes, wherein the left-hand vertical is the simplicial map described in Observation \ref{obs:hcsimp-cubes}.
  \end{itemize}
\end{rmk}

\begin{obs}
  The opposite category and the simplicial dual constructions of
  Recollections~\ref{rec:simp-cat-duals} and~\ref{rec:simp-dual} are related by
  the homotopy coherent nerve construction, in the sense that there exists a
  natural simplicial isomorphism $\hN(\eC\op)\cong\hN(\eC)\op$. This follows
  directly from the observation that there exists a canonical isomorphism
  $(\gC\Del^n)\op\cong\gC\Del^n$ which maps each object $i$ to $n-i$.\footnote{We
    leave it as a diversion for the reader to work out the relationship between
    the homotopy coherent nerve and the codual construction.}
\end{obs}

We now turn our attention to the homotopy coherent realisation functor $\gC\colon\sSet\to\sCat$.

\begin{lem}\label{lem:realization-in-computads} The homotopy coherent realisation functor
  \begin{equation*}
    \xymatrix@R=0em@C=6em{{\sSet}\ar[r]^-{\gC} & \sCptd}
  \end{equation*}
 lands in the subcategory of simplicial computads and morphisms of such.
\end{lem}
\begin{proof}
  Observe that the functor $\oSimp^\beta$ associated with a simplicial operator
  $\beta\colon[m]\to[n]$ carries the atomic arrow $\{k,l\}$ to
  the arrow $\{\beta(k),\beta(l)\}$. This is the unique atomic arrow in
  $\Hom_{\oSimp^n}(\beta(k),\beta(l))$ when $\beta(k)<\beta(l)$ and is the identity on
  $\beta(k)=\beta(l)$ otherwise. It follows, from the fact that an $r$-arrow
  $T^\bullet$ is atomic in $\hcSimp$ if and only if $T^0$ is atomic in $\oSimp$,
  that the simplicial functor $\gC{\Del^{\beta}}$ carries an atomic $r$-arrow of
  $\Fun_{\hcSimp}(k,l)$ to an atomic $r$-arrow in
  $\Fun_{\hcSimp}(\beta(k),\beta(l))$ whenever $\beta(k)<\beta(l)$ and to the
  identity $r$-arrow on $\beta(k)=\beta(l)$ otherwise. In other words,
  $\gC{\Del^{\beta}}\colon\gC{\Del^{m}}\to\gC{\Del^{n}}$ is a functor of simplicial
  computads and so we may factor $\gC{\Del^\bullet}$ through $\sCptd\inc\sCat$ to
  give a functor:
  \begin{equation*}
    \xymatrix@R=0em@C=6em{{\Del}\ar[r]^-{\gC{\Del^\bullet}} & \sCptd}
  \end{equation*}
  Now the homotopy coherent realisation functor is defined by pointwise
  left Kan extension along the Yoneda embedding. Since $\gC{\Del^\bullet}$ lands in $\sCptd$ and Lemma \ref{lem:computad-colimits} tells us that colimits are created by the inclusion   $\sCptd\inc\sCat$, 
  we may factor the
  pointwise left Kan extension \eqref{eq:coherent-realization-kan-ext} through the colimit creating inclusion functor
to show that the homotopy coherent realisation functor also
  lands in $\sCptd$ as claimed.
\end{proof}

We may squeeze a little more information out of these observations using some
Reedy category theory, as explicated in~\cite{RiehlVerity:2013kx} for example.
  
\begin{lem}\label{lem:hocoh-simplex-reedy-cofibrant}
  The $n^{\text{th}}$ latching map $\boundary(\gC\Del^n)\to \gC{\Del^n}$of the
  functor
  \[
    \gC{\Del^\bullet}\colon \Del\longrightarrow\sCptd
  \]
  is isomorphic to the functor of simplicial computads
  $\gC\boundary\Del^n\to \gC{\Del^n}$ obtained by applying the homotopy
  coherent realisation functor to the inclusion $\boundary\Del^n\inc\Del^n$.
  What is more, $\gC{\Del^\bullet}$ is Reedy cofibrant, in the sense that all
  of its latching maps are simplicial subcomputad inclusions.
\end{lem}
\begin{proof}
  Latching objects are constructed as certain colimits, so the latching maps of
  the Yoneda embedding $\Del^\bullet\colon\Del\to\sSet$ are preserved by the
  cocontinuous homotopy realisation functor $\gC\colon\sSet\to\sCptd$. The
  2-cell in the left Kan extension~\eqref{eq:coherent-realization-kan-ext} is an
  isomorphism, since the Yoneda embedding is fully-faithful, so it follows in
  particular that $\gC$ carries latching maps of $\Del^\bullet$ to those of
  $\gC{\Del^\bullet}\colon\Del\to\sCptd$, which gives the first result in the
  statement.
  
  The colimits of $\sCptd$, and thus any latching objects constructed using
  them, are also preserved and reflected by the functor
  $J\colon\sCptd\to\Cptd^{\Del\op\ep}$. What is more, by
  Observation~\ref{obs:simp-subcomp-char} a functor in $\sCptd$ is a simplicial
  subcomputad inclusion if and only if it is mapped to a level-wise subcomputad
  inclusion by $J\colon\sCptd\inc\Cptd^{\Del\op\ep}$. It follows that
  $\gC{\Del^\bullet}$ is Reedy cofibrant if and only if the composite functor
  \begin{equation*}
    \xymatrix@R=0em@C=3em{
      {\Del}\ar[rr]^-{\gC{\Del^\bullet}} & & {\sCptd}\ar[r] &
      {\Cptd^{\Del\op\ep}}
    }
  \end{equation*}
  is Reedy cofibrant in the sense that its latching maps are all injective. Note
  here that $\Cptd^{\Del\op\ep}$ is equivalent to a presheaf category, so the
  Reedy cofibrancy of the composite functor follows by a routine application of
  the following lemma.
\end{proof}

\begin{lem}\label{lem:unaugmentable}
  Suppose that $X$ is a cosimplicial object in a presheaf category
  $\Set^{\scat{C}\op}$ that is unaugmentable, in the sense that the
  equaliser of the pair $X^{\face^0}, X^{\face^1}\colon X^0\rightrightarrows
  X^1$ is empty. Then the latching maps of $X$ are all injective.
\end{lem}

\begin{proof}
  Since latching objects are defined in terms of certain colimits computed
  pointwise in $\Set^{\scat{C}\op}$, we may reduce this result to the
  corresponding one for a cosimplicial set $X\colon\Del\to\Set$. A simplex in a cosimplicial set is ``non-degenerate'' if it is not in the image of a monomorphism from $\Del$. In an unaugmentable cosimplicial set, every degenerate simplex is uniquely expressible as the image of a non-degenerate simplex under a monomorphism. This uniqueness implies that the latching map is a monomorphism; see~\cite[14.3.8]{Riehl:2014kx} for further discussion.
\end{proof}

\begin{lem}\label{lem:subcomputad-inclusion} For any inclusion $i\colon X \inc Y$ of simplicial sets, the induced simplicial functor $\gC{i}\colon\gC{X}\inc\gC{Y}$ is a simplicial subcomputad.
\end{lem}
\begin{proof}
  Any inclusion of simplicial sets $i \colon X \inc Y$ can be built as a colimit
  of a countable sequence of skeleta, each stage of which may be constructed as
  a pushout of a coproduct of simplex boundaries $\boundary\Del^n\inc\Del^n$.
  All colimits are preserved by the left adjoint $\gC$, so it follows that
  $\gC{i}\colon\gC{X}\to\gC{Y}$ may also be expressed as a colimit of a
  countable sequence each step of which is constructed as a pushout of a
  coproduct of latching maps $\gC\boundary\Del^n\inc\gC{\Del^n}$, one for each
  non-degenerate $n$-simplex. Lemma \ref{lem:hocoh-simplex-reedy-cofibrant}
  demonstrates that these latching maps are all simplicial subcomputads and
  Lemma \ref{lem:sub-computad-stability} proves that the class of simplicial
  subcomputads is closed in $\sCptd$ under coproducts, pushouts, and colimits of
  countable sequences. It follows therefore that $\gC{i}\colon\gC{X}\inc\gC{Y}$
  is a simplicial subcomputad.
  \end{proof}
  
  \subsection{Simplicial computads defined by homotopy coherent realisation}\label{sec:realization-computads}
  
  In this section we build towards an explicit presentation of the homotopy coherent realisation $\gC{X}$ as a simplicial computad. To warm up, we recall the simplicial computad structure borne by the homotopy coherent simplices, which we describe geometrically for sake of contrast.
  
  \begin{ex}[homotopy coherent simplices as simplicial computads]\label{ex:homotopy-coherent-simplex} Recall the simplicial category $\mathfrak{C}\Del^n$ whose objects are integers $0,1,\ldots, n$ and whose function complexes are the cubes \[ \Fun_{\mathfrak{C}\Del^n}(i,j) = \begin{cases} (\Del^1)^{j-i-1} & i < j \\ \Del^0 & i = j \\ \emptyset & i > j \end{cases}\] 
  As revealed in the proof of Lemma \ref{lem:realization-in-computads}, $\gC\Del^n$  is a simplicial computad. Via the isomorphism of \ref{obs:hcsimp-cubes}, the atomic arrows in each function complex are precisely those simplices that contain the initial vertex in the poset whose nerve defines the simplicial cube.
\end{ex}

Before we analyse important subcomputads of the homotopy coherent simplices, we introduce notation that suggests that correct geometric intuition.

\begin{ntn}[cubes, boundaries, and cubical horns]\label{ntn:cubes}$\quad$
\begin{itemize}
\item  We shall adopt the notation $\Cube^k$ for the simplicial cube
  $(\Del^1)^{\times k}$.
  \item We write $\boundary\Cube^k\subset\Cube^k$ for the boundary of the cube, whose geometric definition is  clear. Formally, $\boundary\Cube^k$ is the domain of the iterated Leibniz product
  $(\boundary\Del^1\subset\Del^1)^{\leib\times k}$. That is, an $r$-simplex of $\Cube^k$, represented as a $k$-tuple of maps   $(\rho_1,\ldots,\rho_k)$ with each $\rho_i \colon [r] \to [1]$, is a member of $\boundary\Cube^k$ if and
  only if there is some $i$ for which $\rho_i\colon[r]\to[1]$ is constant at either $0$ or $1$ (in which case $\rho_i$ defines an $r$-simplex in $\boundary\Del^1\subset\Del^1$).

\item  We also define the \emph{cubical horn} $\CHorn^{k,j}_{e}\subset\Cube^k$,
  for $1 \leq j \leq k$ and $e \in \{0,1\}$, to be the domain of the following Leibniz product:
  \begin{equation*}
    (\boundary\Del^1\subset\Del^1)^{\leib\times(j-1)}\leib\times
    (\Del^{\fbv{e}}\subset\Del^1)\leib\times
    (\boundary\Del^1\subset\Del^1)^{\leib\times(k-j)}
  \end{equation*}
  So an $r$-simplex $(\rho_1,\ldots,\rho_k)$ of $\Cube^k$ is in $\CHorn^{k,j}_e$ if
  and only if $\rho_i$ is a constant operator for some $i\neq j$
  or $\rho_j$ is the constant operator which maps everything to $e$.
  \end{itemize}
\end{ntn}

\begin{ex}[homotopy coherent nerves of subsimplices]\label{ex:subcomputad-of-simplex}
  Lemma \ref{lem:subcomputad-inclusion} tells us that if $X$ is a simplicial
  subset of $\Del^n$, then its homotopy coherent realisation $\gC{X}$ is a
  simplicial subcomputad of $\gC\Del^n$. What is more, an atomic arrow of
  $\gC{\Del^{n}}$ is in $\gC{X}$ if and only if it is in the image of
  $\gC{\Del^{\alpha}}\colon \gC{\Del^{m}}\inc\gC{\Del^{n}}$ for some
  non-degenerate face $\alpha\colon[m]\inc [n]$ in $X$. Additionally, an
  $r$-arrow $T^\bullet$ of $\gC{\Del^{n}}$ is in the image of
  $\gC{\Del^{\alpha}}$ if and only if the elements of $T^r$ are all vertices of
  $\alpha$.

  It follows that $\gC{X}$ may be presented as the
  simplicial subcomputad of $\gC{\Del^{n}}$ generated by those atomic
  $r$-arrows $T^\bullet$ for which $T^r$ is the set of vertices of some
  non-degenerate $r$-dimensional face in $X\subset\Del^n$. Applying this
  result to some important special cases, we get descriptions of:
 
  \begin{enumerate}[label=(\roman*)]
  \item\label{itm:boundary} \textbf{boundaries:} The only non-degenerate simplex
    in $\Del^n$ that is not present in its boundary $\partial\Del^n$ is the top
    dimensional $n$-simplex. Consequently the subcomputad $\gC[\boundary\Del^n]$
    contains all of the objects of $\gC\Del^n$ and an atomic
    $r$-arrow $T^\bullet$ of $\gC\Del^n$ is not in $\gC[\boundary\Del^n]$ if and
    only if it has $T^r = [0,n]$; in particular each of these missing atomic
    arrows lie in the function complex $\Fun_{\gC\Del^n}(0,n)$ and
    $\Fun_{\gC[\boundary\Del^n]}(i,j)=\Fun_{\gC\Del^n}(i,j)$ for all but that
    particular function complex. Under the isomorphism of
    Observation~\ref{obs:hcsimp-cubes} the inclusion
    \[
      \xymatrix@R=1em{
        \Fun_{\gC[\boundary\Del^n]}(0,n)\ar@{^(->}[r]
        \ar@{}[d]|{\rotatebox{90}{$\cong$}}&
        \Fun_{\gC\Del^n}(0,n)\ar@{}[d]|{\rotatebox{90}{$\cong$}} \\
        \boundary\Cube^{n-1}\ar@{^(->}[r] & \Cube^{n-1}}
    \]
    is isomorphic to the cubical boundary inclusion.
  \item\label{itm:inner-horn} \textbf{inner horns:} For $0 < k < n$, the subset
    $\Horn^{n,k} \subset \Del^n$ contains all of the vertices and all but two of
    the simplices, the top dimensional $n$-simplex and its $k^{\text{th}}$ face.
    Consequently the subcomputad $\gC\Horn^{n,k}$ contains all of the objects of
    $\gC\Del^n$ and the atomic arrows of $\gC{\Del^{n}}$ that are missing from
    $\gC{\Horn^{n,k}}$ are all members of the function complex
    $\Fun_{\gC\Del^n}(0,n)$; so $\Fun_{\gC{\Horn^{n,k}}}(i,j)=
    \Fun_{\gC\Del^n}(i,j)$ for all but that particular function complex. What is
    more, under the isomorphism of Observation~\ref{obs:hcsimp-cubes} the
    inclusion
    \[
      \xymatrix@R=1em{
        \Fun_{\gC{\Horn^{n,k}}}(0,n)\ar@{^(->}[r]
        \ar@{}[d]|{\rotatebox{90}{$\cong$}}&
        \Fun_{\gC\Del^n}(0,n)\ar@{}[d]|{\rotatebox{90}{$\cong$}} \\
        \CHorn^{n-1,k}_1\ar@{^(->}[r] & \Cube^{n-1}}
    \]
    is isomorphic to the cubical horn inclusion.
    
  \item\label{itm:outer-horn}\textbf{outer horns:} Because $\Horn^{n,n}$
    contains all of the vertices of $\Del^n$, the simplicial subcomputad
    $\gC\Horn^{n,n}$ contains all of the objects of $\gC\Del^n$. The only
    non-degenerate simplices in $\Del^n$ that are not present in the horn
    $\Horn^{n,n}$ are the top dimensional $n$-simplex and its $n$th face. In the
    former case, the missing atomic $r$-arrows are elements of
    $\Fun_{\gC\Del^n}(0,n)$, and in the latter case they are elements of
    $\Fun_{\gC\Del^n}(0,n-1)$; so $\Fun_{{\gC\Horn^{n,n}}}(i,j)=
    \Fun_{\gC\Del^n}(i,j)$ for all but those two function complexes. What is
    more, under the isomorphism of Example \ref{ex:homotopy-coherent-simplex}
    the inclusions
    \[
      \mkern30mu
      \xymatrix@R=1em{
        \Fun_{\gC\Horn^{n,n}}(0,n-1)\ar@{^(->}[r]
        \ar@{}[d]|{\rotatebox{90}{$\cong$}}&
        \Fun_{\gC\Del^n}(0,n-1)\ar@{}[d]|{\rotatebox{90}{$\cong$}} \\
        \boundary\Cube^{n-2}\ar@{^(->}[r] &\Cube^{n-2}}
      \mkern10mu
      \xymatrix@R=1em{
        \Fun_{{\gC\Horn^{n,n}}}(0,n)\ar@{^(->}[r]
        \ar@{}[d]|{\rotatebox{90}{$\cong$}}&
        \Fun_{\gC\Del^n}(0,n)\ar@{}[d]|{\rotatebox{90}{$\cong$}} \\
        \CHorn^{n-1,n-1}_0\ar@{^(->}[r] & \Cube^{n-1}}
    \]
    are isomorphic to the cubical boundary and cubical horn inclusions. The
    analysis for $\gC\Horn^{n,0} \inc\gC\Del^n$ is similar, with
    $\Fun_{{\gC\Horn^{n,0}}}(0,n)\inc \Fun_{\gC\Del^n}(0,n)$ isomorphic to the
    cubical horn inclusion $\CHorn^{n-1,1}_0\subset\Cube^{n-1}$.
  \end{enumerate}
\end{ex}

From the explicit description of the homotopy coherent inner horn inclusions given in \ref{ex:subcomputad-of-simplex}\ref{itm:inner-horn} it is straightforward to prove the following well-known result:

\begin{prop}\label{prop:qcat-from-kan-enriched} If $\eC$ is a Kan complex enriched category, then its homotopy coherent nerve $\qC\defeq\nrvhc\eC$ is a quasi-category.
\end{prop}


\begin{defn}[bead shapes]
  We shall call those atomic arrows $T^\bullet\colon 0 \to n$ of $\gC{\Del^{n}}$
  which are not members of $\gC\boundary\Del^n$ \emph{bead shapes}. An $r$-dimensional bead shape $T^\bullet\colon
  0\to n$ is given by a sequence of subsets
  \[ \{0,n\} = T^0 \subset T^1 \subset \cdots \subset T^r = [0,n].\] 
  \end{defn}
  
  \begin{lem}\label{lem:bead-shapes} An atomic $r$-arrow in $\gC\Del^n$ is a bead shape if and only if it is represented in the notation of \ref{ntn:compact-arrows} by a sequence $\langle {I_0\mid I_1\mid\ldots\mid I_r} \rangle$ with $I_0=\{0,n\}$ and
  $I_0,I_1,\ldots,I_r$ a partition of $[0,n]$ into non-empty subsets. Thus bead shapes of $\gC\Del^n$ stand in bijection with partitions $I_1,\ldots, I_r$ of $[1,n-1]$ into non-empty subsets.
 \end{lem} 
 \begin{proof} By the calculation of Example \ref{ex:subcomputad-of-simplex}\ref{itm:boundary}, an atomic $r$-arrow $T^\bullet\colon 0 \to n$ of $\gC{\Del^{n}}$ is not in $\gC{\boundary\Del^n}$ if and only if $T^r=[0,n]$; thus the union $I_0 \cup \cdots \cup I_n = [0,n]$. Atomicity demands further that $T^0 = I_0 = \{0,n\}$ and non-degeneracy corresponds to the requirement that each of $I_1,\ldots, I_n$ is non-empty.
  \end{proof}

  Observe
  now that in any pushout
  \begin{equation*}
    \xymatrix@=2em{
      {\gC\boundary\Del^n}\ar@{^(->}[r]\ar[d]_{b} &
      {\gC{\Del^{n}}}\ar[d] \\
      {\eB}\ar@{^(->}[]!R(1.2);[r] & {\eC}\poexcursion
    }
  \end{equation*}
  in $\sCptd$ the graph of atomic $r$-arrows of $\eC$ is obtained from that of
  $\eB$ by adjoining one atomic $r$-arrow to the function complex
  $\Fun_{\eB}(b_0,b_n)$ for each $r$-dimensional bead shape $T^\bullet\colon
  0\to n$.  We leave it to the reader to determine how the face and degeneracy operators act on these new atomic arrows in $\eC$.

\begin{prop}[$\gC X$ as a simplicial computad]\label{prop:gothic-C}
The homotopy coherent realisation $\gC{X}$ of a simplicial set $X$ is a simplicial computad with
  \begin{itemize}
  \item   objects the vertices of $X$ and 
  \item non-degenerate atomic $r$-arrows given by pairs $(x,T^\bullet)$, wherein
  $x$ is a non-degenerate $n$-simplex of $X$ for some $n > r $ and
  $T^\bullet\colon 0\to n$ is an $r$-dimensional bead shape. 
  \end{itemize}
  The domain of $(x,T^\bullet)$ is the initial vertex $x_0$ of $x$ while the codomain is the 
terminal vertex $x_n$.
 \end{prop}
 
The pairs $(x,T^\bullet)$ are called \emph{beads in $X$}. As a consequence of this result we find that 
  $r$-simplices of $\gC{X}$ correspond to sequences of abutting beads,
  structures which are called \emph{necklaces} in the work of Dugger and
  Spivak~\cite{DuggerSpivak:2011ms} and Riehl~\cite{Riehl:2011ot}. In this terminology, 
   $\gC X$ is a simplicial computad in which the atomic arrows are those necklaces that consist of a single bead with non-degenerate image.
   
 \begin{proof}
 As observed in the proof of Lemma \ref{lem:subcomputad-inclusion}, the
  homotopy coherent realisation $\gC{X}$ is
  constructed by a process which adjoins one copy of $\gC{\Del^{n}}$ along a map
  of its boundary $\gC\boundary\Del^n$ for each non-degenerate $n$-simplex
  $x\in X$. By Lemma \ref{lem:sub-computad-stability}, each atomic $r$-arrow of $\gC{X}$ arises from a unique pushout of this form, as the image of an atomic non-degenerate $r$-arrow of $\gC{\Del^n}$ not in $\gC{\boundary\Del^n}$. 
  \end{proof}


\begin{ntn}[compact notation for beads]\label{ntn:beads}
Recall a \emph{bead} $(x,T^\bullet)$ representing a non-degenerate atomic $r$-arrow in $\gC{X}$ is given by a non-degenerate simplex $x \in X_n$ with $n > r$ together with a $r$-dimensional bead shape $T^\bullet \colon 0 \to n$ in $\gC{\Del^n}$. In the notation introduced in \ref{ntn:compact-arrows}, Lemma \ref{lem:bead-shapes} tells us that the bead shape $T^\bullet$ is encoded by a sequence 
  $\langle {I_0\mid I_1\mid\ldots\mid I_r} \rangle$ with $I_0=\{0,n\}$ and
  $I_0,I_1,\ldots,I_r$ a partition of $[0,n]$ into non-empty subsets. In the case $r=0$, we must have $n=1$; that is, the 0-dimensional beads are given by non-degenerate 1-simplices of $X$. In the case $r=1$, we must have $I_0 = \{0,n\}$ and $I_1 = \{1,\ldots, n-1\}$ for some $n \geq 2$; hence, the 1-dimensional beads are given by non-degenerate $n$-simplices of $X$ for $n \geq 2$.
  
Here $I_0$ provides no information beyond that contained
  the dimension of $x$, so we may therefore write the bead $(x,T^\bullet)$ as
  $\langle x;I_1\mid\ldots\mid I_r \rangle$ where $I_1,\ldots,I_r$ is a partition of $[1,n-1]$ into non-empty subsets. Its domain and codomain
  objects are $x\cdot\face^{\fbv{0}}$ and $x\cdot\face^{\fbv{n}}$ respectively.

\end{ntn}

The utility of the simplicial computad characterisation of the simplicial categories $\gC{X}$ defined by homotopy coherent realisation is that it provides an immediate proof of the following extension lemma, a specialisation of Proposition \refII{prop:simp-computad-maps} to the simplicial computads described in Proposition \ref{prop:gothic-C}.

\begin{lem}\label{lem:computad-extension} To define a simplicial functor $F \colon \gC{X} \to \eK$ is to:
\begin{itemize}
\item Specify an object $Fx \in \eK$ for each vertex $x \in X_0$.
\item Specify a 0-simplex $F\alpha \in \Fun(Fx_0,Fx_1)$ for each non-degenerate 1-simplex $ \alpha\colon x_0 \to x_1 \in X_1$.
\item Specify a 1-simplex $F\beta \in \Fun(Fx_0,Fx_n)$ for each non-degenerate $n$-simplex $\beta \in X_n$ whose initial vertex is $x_0$ and whose final vertex is $x_n$. The source of $F\beta$ must be the image of the diagonal edge of $\beta$, while its target must be the image of its \emph{spine}, the composite of the path of edges from $x_i$ to $x_{i+1}$ for each $0 \leq i < n$.
\item Specify a $r$-simplex $F\sigma \in \Fun(Fx_0,Fx_n)$ for each bead $(\sigma,T^\bullet)$ comprised of a non-degenerate $n$-simplex $\sigma \in X_n$, whose initial vertex is $x_0$ and whose final vertex is $x_n$, and $r$-dimensional bead shape $T^\bullet$ in such a way that the faces of $F\sigma$, which decomposed as a unique composite of degenerated atomic arrows as in \eqref{eq:computad-arrow-decomp}, coincide with the previously specified data.
\end{itemize}
\end{lem}

\subsection{Extensions from simplicial subcomputads}\label{sec:relative-computad}

More generally, Lemma \ref{lem:subcomputad-inclusion} proves that any inclusion of simplicial sets $X\inc Y$ defines a simplicial subcomputad inclusion $\gC X \inc \gC Y$. In this section, we describe the evident relative analogue of Lemma \ref{lem:computad-extension}.


\begin{lem}\label{lem:subcomputad-factorization} If $\eA\inc\eB$ is a simplicial subcomputad inclusion then $\eA$ is closed under factorisations in $\eB$:
i.e., if whenever $f$ and $g$ are composable arrows of $\eB$ and $g \circ f \in \eA$ then $f$ and $g$ are in $\eA$.
\end{lem}
\begin{proof}
Since the inclusion $\eA\inc\eB$ is a faithful functor of simplicial computads, the atomic factorisation of $g \circ f$ in $\eA$ is carried to an atomic factorisation in $\eB$, which, by uniqueness, must equal the composite of the atomic factorisations of $g$ and $f$ in $\eB$. This tells us that the maps $f$ and $g$ must themselves lie in $\eA$.
\end{proof}

On account of Lemma \ref{lem:subcomputad-factorization} extending along simplicial subcomputads is particularly simple: as in Lemma \ref{lem:computad-extension}, all that is needed is to specify the images of the new atomic $n$-arrows from the codomain compatibly with lower-dimensional data. Because subcomputads are closed under factorisation, it will never be the case that the extension problem imposes coherence relations on composites of atomic arrows present in its domain computad.

\begin{lem}\label{lem:subcomputad-extension}
To define an extension of a simplicial functor along a simplicial subcomputad inclusion
\[ \xymatrix{ \gC{X} \ar[r]^-F \ar@{^(->}[d] & \eK \\ \gC{Y} \ar@{-->}[ur]_-F}
\] is to:
\begin{itemize}
\item Specify an object $Fy \in \eK$ for each vertex $y$ in $Y$ that is not in $X$.
\item Specify a 0-simplex $F\alpha \in \Fun(Fy_0,Fy_1)$ for each non-degenerate 1-simplex $ \alpha\colon y_0 \to y_1 \in Y$ that is not in $X$.
\item Specify a 1-simplex $F\beta \in \Fun(Fy_0,Fy_n)$ for each non-degenerate $n$-simplex $\beta \in Y$ that is not in $X$ whose initial vertex is $y_0$ and whose final vertex is $y_n$. The source of $F\beta$ must be the image of the diagonal edge of $\beta$, while its target must be the image of its \emph{spine}, the composite of the path of edges from $y_i$ to $y_{i+1}$ for each $0 \leq i < n$.
\item Specify an $r$-simplex $F\sigma \in \Fun(Fy_0,Fy_k)$ for each bead $(\sigma,T^\bullet)$ comprised of a non-degenerate $n$-simplex $\sigma \in Y$ that is not in $X$ whose initial vertex is $y_0$ and whose final vertex is $y_n$, 
and $r$-dimensional bead shape $T^\bullet$ in such a way that the faces of $F\sigma$, which decomposed as a unique composite of degenerated atomic arrows as in \eqref{eq:computad-arrow-decomp}, coincide with the previously specified data.
\end{itemize}
\end{lem}



%% file: cocones.tex

\section{Extending cocartesian cocones}\label{sec:cocones}

The mechanics of the comprehension construction involve extensions of
\emph{cocartesian cocones} valued in an $\infty$-cosmos. Our first introduction
to this construction takes the form of a special case in
\S\ref{sec:global-cocart}. Cocartesian cocones are particular simplicial natural
transformations between certain functors, which we call \emph{lax cocones};
these are introduced in \S\ref{sec:lax-cocone}. Cocartesian cocones themselves
are defined in \S\ref{sec:cocart-cocone-def}, and the promised extension result
is proven in \S\ref{sec:cocart-cone}. This section concludes with a proof that
the space of lifts of a lax cocone along a fixed cocartesian fibration defines a
contractible Kan complex.

\subsection{A global universal property for cocartesian fibrations}\label{sec:global-cocart}

Using the tools developed in \S\ref{sec:coherent-nerve}, we may extend the
\emph{local\/} lifting property of cocartesian fibrations given in Corollary
\ref{cor:representable-cocart}, positing the existing of cocartesian lifts of
1-arrows with specified domain, to a \emph{global} lifting property of some
significant utility. Suppose that we are given simplicial functors
\begin{equation*}
  \xymatrix@R=0em@C=8em{
    {\gC\Horn^{n,n}}\ar[r]^{E} & {\eK}
  }
\end{equation*}
and
\begin{equation*}
  \xymatrix@R=0em@C=8em{
    {\gC\Del^n}\ar[r]^{B} & {\eK}
  }
\end{equation*}
and a simplicial natural transformation
\begin{equation}\label{eq:domain-transformation}
  \xymatrix@R=0em@C=8em{
    {\gC\Horn^{n,n}}\ar@/_2ex/[]!R(0.5);[r]_{B}\ar@/^2ex/[]!R(0.5);[r]^{E}
    \ar@{}[r]|-{\textstyle \Downarrow p} & {\eK}
  }
\end{equation}
whose final component $p_n \colon E_n \tfib B_n$ is a cocartesian fibration. We
will provide sufficient conditions under which it is possible to extend $p$ and
its domain $E$ from $\gC\Horn^{n,n}$ to give a simplicial natural transformation
whose codomain is the specified simplicial functor $B$ on $\gC\Del^n$, as shown
in the diagram:
\begin{equation}\label{eq:outer-horn-extension}
  \xymatrix@R=3em@C=8em{
    {\gC\Horn^{n,n}}\ar@{^(->}[dd]
    \ar@/_2ex/[]!DR(0.5);[dr]_{B}^{}="two" \ar@/^2ex/[]!DR(0.5);[dr]^{E}_{}="one"
    \ar@{} "one";"two" \ar@{=>}?(0.25);?(0.75)^{p}\\
    & {\eK} \\
    {\gC\Del^n}\ar@/_2ex/[]!UR(0.5);[ru]_{B}^{}="four"
    \ar@{.>}@/^2ex/[]!UR(0.5);[ru]_{}="three"
    \ar@{} "three";"four" \ar@{:>}?(0.25);?(0.75)\\
  }
\end{equation}

To gain an intuition for the construction we are imagining, and  to see that it is sensibly described as a ``lifting property,'' it is instructive
to consider a couple of low dimensional examples:

\begin{obs}[low dimensional outer horn extensions]\label{obs:low-dim-horn-ext} $\quad$
  \begin{enumerate}[label=(\roman*)]
  \item \textbf{dimension $\mathbf{n=1}$:} The information provided by \eqref{eq:outer-horn-extension} comprises
    the two solid arrows and their domain and codomain objects in the following
    diagram:
    \begin{equation*}
      \xymatrix@R=2em@C=4em{
        {E_0}\pbexcursion\ar@{.>}[r]^{E_{\langle 0,1 \rangle}}\ar@{.>}[d]_{p_0} &
        {E_1}\ar[d]^{p_1} \\
        {B_0}\ar[r]_{B_{\langle 0,1 \rangle}} & {B_1}
      }
    \end{equation*}
    Without placing any condition on this data, we may extend it to the
    requested natural transformation on $\gC\Del^1$ simply by forming the
    displayed pullback.
  \item \textbf{dimension $\mathbf{n=2}$:}\label{item:dim2} Here again we depict the information
    provided by \eqref{eq:outer-horn-extension} as the solid arrows in the following diagram:
    \begin{equation*}
      \begin{xy}
        0;<1.25cm,0cm>:<0cm,1cm>::
        *{\xybox{
            \POS(1.5,0)*+{B_1}="one"
            \POS(0,1)*+{B_0}="two"
            \POS(3.5,0.5)*+{B_2}="three"
            \ar "two";"one"_{B_{\langle {0,1} \rangle}}
            \ar@/_5pt/ "one";"three"_{B_{\langle {1,2} \rangle}}^(0.1){}="otm"
            \ar@/^10pt/ "two";"three"^(0.6){B_{\langle {0,2} \rangle}}_(0.5){}="ttm"|(0.42){\hole}
            \ar@{=>} "ttm"-<0pt,7pt> ; "otm"+<0pt,10pt> ^(0.3){B_{\langle {0,2 \mid 1} \rangle}}
            \POS(1.5,2.5)*+{E_1}="one'"
            \POS(1.5,2.5)*{\pbcorner}
            \POS(0,3.5)*+{E_0}="two'"
            \POS(3.5,3)*+{E_2}="three'"
            \ar@/_5pt/ "one'";"three'"_{E_{\langle {1,2} \rangle}}
            \ar@/^10pt/ "two'";"three'"^{E_{\langle {0,2} \rangle}}_(0.55){}="ttm'"
            \ar@{->>} "one'";"one"_(0.325){p_1}
            \ar "two'";"two"_{p_0}
            \ar@{->>} "three'";"three"^{p_2}
            \ar@{..>} "two'";"one'"_*!/^2pt/{\scriptstyle E_{\langle {0,1} \rangle}}
            \ar@{..>}@/_10pt/ "two'";"three'"^(0.44){}="otm'"
            \ar@{:>} "ttm'"-<0pt,4pt> ; "otm'"+<0pt,4pt> ^(0.3){E_{\langle {0,2 \mid 1} \rangle}}
          }}
      \end{xy}
    \end{equation*}
    The desired extension comprises the dotted arrows, and in order to construct
    those we shall make two assumptions regarding $p$:
    \begin{enumerate}[label=(\roman*)]
    \item its component at the object $2\in\gC\Horn^{2,2}$, located at the right of the
      diagram, is a cocartesian fibration, and
    \item its naturality square associated with the $0$-arrow
      $\langle {1,2} \rangle\in\Fun_{\gC\Horn^{2,2}}(1,2)$, located at the front of the
      diagram, is a pullback.
    \end{enumerate}
    Under these conditions, our first step is to lift the
    whiskered $1$-arrow $B_{\langle {0,2 \mid 1} \rangle}\circ p_0$ in
    $\Fun_{\eK}(E_0,B_2)$ to give a $p_2$-cocartesian arrow $E_{\langle {0,2
        \mid 1} \rangle}$ in $\Fun_{\eK}(E_0,E_2)$. Observe now that by
    construction the dotted arrow from $E_0$ to $E_2$, comprising the target of
    that lift, composes with $p_2$ to give an $0$-arrow equal to the composite
    $B_{\langle {1,2} \rangle}\circ B_{\langle {0,1} \rangle}\circ p_0$. It
    follows, therefore, that we may apply the pullback property of the
    naturality square at the front to induce the $0$-arrow $E_{\langle {0,1}
      \rangle}$. This makes the square on the left commute and thus completes
    $E$ and $p$ to structures on $\gC\Del^2$.
  \end{enumerate}
  The following proposition furnishes us with corresponding extension results
  for right outer horns of higher dimension.
\end{obs}

\begin{prop}\label{prop:outer-horn-extension}
  Fix an $n> 2$ and suppose that we are given simplicial functors
  $E\colon\gC\Horn^{n,n}\to\eK$ and $B\colon\gC\Del^n\to\eK$ and
  a simplicial natural transformation $p$ from $E$ to
  $B$ on $\gC\Horn^{n,n}$ as in \eqref{eq:domain-transformation}. Suppose
  further that:
  \begin{itemize}
  \item the component $p_n\colon E_n\tfib B_n$ is a cocartesian fibration,
  \item the naturality square
    \begin{equation}\label{eq:pullback-condition}
      \xymatrix@R=2em@C=4em{
        {E_{n-1}}\pbexcursion\ar@{->>}[d]_-{p_{n-1}}\ar[r]^{E_{\langle {n-1,n} \rangle}} &
        {E_n}\ar@{->>}[d]^-{p_n} \\
        {B_{n-1}}\ar[r]_{B_{\langle {n-1,n} \rangle}} & {B_n}
      }
    \end{equation}
    is a pullback, and
  \item for each $k=0,\ldots,n-2$ the $1$-arrow $E_{\langle {k,n \mid n-1} \rangle}$ in the diagram
    \begin{equation}\label{eq:cocartesian-arrow-condition}
      \begin{xy}
        0;<1.25cm,0cm>:<0cm,1cm>::
        *{\xybox{
            \POS(1.5,0)*+{B_{n-1}}="one"
            \POS(-0.5,1)*+{B_k}="two"
            \POS(3.5,0.5)*+{B_n}="three"
            \ar "two";"one"_{B_{\langle {k,n-1} \rangle}}
            \ar@/_5pt/ "one";"three"_{B_{\langle {n-1,n} \rangle}}^(0.1){}="otm"
            \ar@/^10pt/ "two";"three"^(0.6){B_{\langle {k,n} \rangle}}_(0.55){}="ttm"|(0.5){\hole}
            \ar@{=>} "ttm"-<0pt,7pt> ; "otm"+<0pt,10pt> ^(0.5){B_{\langle {k,n \mid n-1} \rangle}}
            \POS(1.5,2.5)*+{E_{n-1}}="one'"
            \POS(1.5,2.5)*{\pbcorner}
            \POS(-0.5,3.5)*+{E_k}="two'"
            \POS(3.5,3)*+{E_n}="three'"
            \ar@/_5pt/ "one'";"three'"_{E_{\langle {n-1,n} \rangle}}^(0,1){}="otm'"
            \ar@/^10pt/ "two'";"three'"^{E_{\langle {k,n} \rangle}}_(0.55){}="ttm'"
            \ar@{->>} "one'";"one"_(0.325){p_{n-1}}
            \ar "two'";"two"_{p_k}
            \ar@{->>} "three'";"three"^{p_n}
            \ar "two'";"one'"_*!/^2pt/{\scriptstyle E_{\langle {k,n-1} \rangle}}
            \ar@{=>} "ttm'"-<0pt,7pt> ; "otm'"+<0pt,10pt> ^(0.5){E_{\langle {k,n \mid n-1} \rangle}}
          }}
      \end{xy}
    \end{equation}
    is $p_n$-cocartesian.
  \end{itemize}
  Then we may construct the extension of $E$ and $p$ depicted
  in~\eqref{eq:outer-horn-extension}.
\end{prop}

\begin{proof}
  We know from Example \ref{ex:subcomputad-of-simplex}\ref{itm:outer-horn} that $\gC\Horn^{n,n}$ may be
  regarded as a simplicial subcomputad of $\gC\Del^n$ and that they differ
  only in their function complexes from the object $0$ to the objects $n-1$ and $n$
  respectively. In particular they share the same set of objects, so $p$
  already has a suitable component for each object of $\gC\Del^n$. Consequently,
  all we are required to do is show that we may extend the action of $E$
  to the function complexes $\Fun_{\gC\Del^n}(0,n)$ and $\Fun_{\gC\Del^n}(0,n-1)$ in
  a way which ensures both the functoriality of the extended $E$ and the
  naturality of the existing components of $p$ with respect to that
  extension.

On contemplating the structure of $\gC\Del^n$ a
  little further, it is clear that we may construct the required extension in two steps:
  \begin{enumerate}[label=(\roman*)]
  \item extend $E$ to the function complex $\Fun_{\gC\Del^n}(0,n)$ in a way
    that ensures naturality with respect to $p_0$ and $p_n$, by solving the
    following lifting problem:
    \begin{equation}\label{eq:2}
      \xymatrix@R=2em@C=3em{
        {\Fun_{\gC\Horn^{n,n}}(0,n)}\ar@{^(->}[d]\ar[rr]^{E} & &
        {\Fun_{\eK}(E_0,E_n)}\ar@{->>}[d]^{p_n\circ{-}} \\
        {\Fun_{\gC\Del^n} (0,n)} \ar[r]_{B}\ar@{.>}[urr]^{E} &
        {\Fun_{\eK}(B_0,B_n)}\ar[r]_{{-}\circ p_0} & {\Fun_{\eK}(E_0,B_n)}
      }
    \end{equation}
  \item extend $E$ to the function complex $\Fun_{\gC\Del^n}(0,n-1)$ in a
    way that ensures functoriality with respect to composition by the $0$-arrow
    $E_{\langle {n-1,n} \rangle}\colon E_{n-1}\to E_{n}$ and naturality with respect to $p_0$
    and $p_{n-1}$, by solving the following problem:
    \begin{equation}\label{eq:lift3}
      \xymatrix@R=2em@C=5em{
        {\Fun_{\eK}(B_0,B_{n-1})}\ar[r]^{{-}\circ p_0} &
        {\Fun_{\eK}(E_0,B_{n-1})} \\
        {\Fun_{\gC\Del^n}(0,n-1)}\ar@{.>}[r]^{E}\ar[u]^{B}\ar[d]_{{-}\circ\langle {n-1,n} \rangle} &
        {\Fun_{\eK}(E_0,E_{n-1})}\ar[u]_{p_{n-1}\circ{-}}\ar[d]^{E_{\langle {n-1,n} \rangle}\circ{-}} \\
        {\Fun_{\gC\Del^n}(0,n)}\ar[r]_{E} & {\Fun_{\eK}(E_0,E_n)}
      }
    \end{equation}
  \end{enumerate}
  On consulting \ref{ex:subcomputad-of-simplex}\ref{itm:outer-horn}  again, we see that the inclusion
  $\Fun_{\gC\Horn^{n,n}}(0,n-1)\inc \Fun_{\gC\Del^n}(0,n-1)$ is isomorphic to
  $\boundary\Cube^{n-2}\subseteq\Cube^{n-2}$ and that the inclusion
  $\Fun_{\gC\Horn^{n,n}}(0,n)\inc \Fun_{\gC\Del^n}(0,n)$ is isomorphic to
  $\CHorn^{n-1,n-1}_0\subseteq \Cube^{n-1}$. The second of these inclusions, as noted in
  \ref{ntn:cubes}, is isomorphic to the Leibniz product
  $(\boundary\Cube^{n-2}\subseteq\Cube^{n-2})\leib\times
  (\Del^{\fbv{0}}\subseteq\Del^1)$, so Lemma \ref{lem:cocart-cylinder-extensions} proves that
  we may solve the lifting problem in~\eqref{eq:2} so long as the cylinder
  \begin{equation*}
    \boundary\Cube^{n-2}\times\Del^1 \subseteq
    (\boundary\Cube^{n-2}\times\Del^1)\cup(\Cube^{n-2}\times\Del^{\fbv{0}})\cong
    \Fun_{\gC\Horn^{n,n}}(0,n)\xrightarrow{\mkern20mu E\mkern20mu}
    \Fun_{\eK}(E_0,E_n)
  \end{equation*}
  is pointwise $\Fun_{\eK}(E_0,p_n)$-cocartesian. 
  
 To verify this, consider a vertex
  $\vec{a}=(a_1,\ldots,a_{n-2})$ of $\Cube^{n-2}$. 
Under the
  isomorphism $\Cube^{n-2}\times\Del^1\cong\Fun_{\gC\Del^n}(0,n)$,
  the $1$-simplex
  $((a_1,\ldots,a_{n-2})\cdot\degen^0,\id_{[1]})$ corresponds to the $1$-arrow $\langle I_{\vec{a}}\mid n-1 \rangle$ where $I_{\vec{a}} =
  \{0,n\}\cup\{i\in[1,n-2]\mid a_i=1\}$. Take $k$ to be the largest integer in
  $J_{\vec{a}}=I_{\vec{a}}\setminus\{n\}$ and then observe that $\langle
  J_{\vec{a}} \rangle$ is a $0$-arrow in $\Fun_{\gC\Del^n}(0,k)$ and that
  $\langle I_{\vec{a}} \mid n-1 \rangle$ may be expressed as the whiskered
  composite $\langle J_{\vec{a}} \rangle\circ\langle k,n \mid n-1 \rangle$. By
  the functoriality of $E$ it follows that the $1$-arrow $E_{\langle
    {I_{\vec{a}} \mid n-1} \rangle}$ can be written as a whiskered composite
  $E_{\langle k,n \mid n-1 \rangle}\circ E_{\langle J_{\vec{a}} \rangle}$. In
  the statement of this proposition we assumed that each $E_{\langle {k,n\mid
      n-1} \rangle}$, depicted in \eqref{eq:cocartesian-arrow-condition}, is $p_n$-cocartesian and we
  know, from Corollary \ref{cor:representable-cocart}, that these are preserved under whisker
  precomposition by any $0$-arrow. It follows therefore that $E_{\langle
    I_{\vec{a}}\mid n-1 \rangle}$ is $p_n$-cocartesian and thus that the
  required lift exists in~\eqref{eq:2}.

  To construct the lift displayed in \eqref{eq:lift3} consider the following diagram:
  \begin{equation}\label{eq:cocartesian-missing-face}
    \xymatrix@R=2em@C=3em{
      & {\Fun_{\eK}(B_0,B_{n-1})}\ar[r]^{{-}\circ p_0} &
      {\Fun_{\eK}(E_0,B_{n-1})}\ar[dr]!L(0.75)^{B_{\langle {n-1,n} \rangle}\circ{-}} \\
      {\Fun_{\gC\Del^n}(0,n-1)}\ar@{.>}[r]\ar[]!R(0.75);[ur]^{B}
      \ar[]!R(0.75);[dr]_{{-}\circ\langle {n-1,n} \rangle} &
      {\Fun_{\eK}(E_0,E_{n-1})}\ar[]!R(0.75);[ur]_{p_{n-1}\circ{-}}
      \ar[]!R(0.75);[dr]^{E_{\langle {n-1,n} \rangle}\circ{-}}{\save[]+<50pt,0pt>*{\rotatebox{45}{$\pbcorner$}}\restore}& &
      {\Fun_{\eK}(E_0,B_n)}\\
      & {\Fun_{\gC\Del^n}(0,n)}\ar[r]_{E} & {\Fun_{\eK}(E_0,E_n)}
      \ar[ur]!L(0.75)_{p_n\circ{-}}
    }
  \end{equation}
  It is straightforward to check that the outer hexagon commutes, this being an
  immediate consequence of the commutativity of the lower triangle
  in~\eqref{eq:2}. What is more, the assumption in the statement that the
  naturality square in \eqref{eq:pullback-condition} is a pullback in $\eK$ implies that the
  right hand diamond above is a pullback, whose universal property induces the
  required (dotted) solution to the problem in \eqref{eq:lift3} as required.
\end{proof}

The outer horn inclusions $\Horn^{n,n}\inc\Del^n$ of Proposition \ref{prop:outer-horn-extension} are isomorphic to the maps obtained by joining a sphere boundary inclusion $\partial\Del^{n-1}\inc\Del^{n-1}$ with a ``cocone vertex'' $\Del^0$ on the right. Proposition \ref{prop:extending-cocart-cones} will generalise this global lifting property, extending natural transformations along an inclusion of homotopy coherent diagrams of shape $X\join\Del^0\inc Y\join\Del^0$ where $X \inc Y$ is an arbitrary inclusion of simplicial sets. First we characterise the class of simplicial functors $\gC[X \join\Del^0] \to \eK$ that will be of interest: namely, those that define a \emph{lax cocone} under a $\gC[X]$-shaped diagram valued in the Kan complex enriched core of $\eK$.

\subsection{Lax cocones}\label{sec:lax-cocone}

  The groupoidal core functor $g\colon\qCat\to\Kan$, which carries each
  quasi-category to the maximal Kan complex spanned by its invertible arrows,
  preserves products. So if $\eC$ is a category enriched in quasi-categories
  then we may take the groupoidal core of its function complexes to construct a
  Kan complex enriched category denoted $g_*\eC$. 
  
  \begin{ntn}\label{ntn:qcat-from-cosmos} For a quasi-categorically enriched category $\eC$, we shall use $\qC$ to denote
  the quasi-category constructed by applying the homotopy coherent nerve
  construction to $g_*\eC$. In particular, we write $\qK$ for the quasi-category constructed in this way from an $\infty$-cosmos $\eK$.
\end{ntn}

The following lemma is used to detect those simplicial functors $\gC\Del^n \to \eC$ that represent $n$-simplices in $\qC$.

\begin{lem}\label{lem:core-condition}
Suppose that we are
  given a homotopy coherent simplex $X\colon\gC\Del^n\to\eC$ for some $n\geq 2$ with the property that all of its $2$-dimensional faces
  $X\cdot\face^{\fbv{i,j,k}}\colon \gC\Del^2\to\eC$ factor through the simplicial
  subcategory inclusion $g_*\eC\subseteq \eC$. Then $X$ itself factors through
  that inclusion.
\end{lem}

\begin{proof}
  The condition given in the statement is equivalent to postulating that for
  each 3-element subset $\{i<j<k\}\subseteq[n]$ the $1$-arrow $\langle i,k\mid j
  \rangle$ in $\Fun_{\gC\Del^n}(i,k)$ is mapped by $X$ to a $1$-arrow
  $X_{\langle i,k \mid j \rangle}$ which is invertible in the function complex
  $\Fun_{\eC}(X_i,X_k)$. So consider an arbitrary $1$-arrow $\langle I_0\mid I_1
  \rangle$ in $\Fun_{\gC\Del^n}(i,k)$ where $I_1=\{j_1<j_2 < \ldots <j_r\}$,
  and consider the $r$-arrow $\langle I_0\mid j_1\mid j_2\mid\ldots\mid j_r\rangle$
  in $\Fun_{\gC\Del^n}(i,k)$. Notice that this derived $r$-arrow has as its spine
  the sequence of $1$-arrows of the form $\langle I_0\cup\{j_1,\ldots,j_{l-1}\}\mid
  j_l \rangle$ for $1 \leq l \leq r$ and it witnesses that their composite in
  $\Fun_{\gC\Del^n}(i,k)$ is the original $1$-arrow $\langle {I_0,I_1} \rangle$.
  What is more, if we let $j_0$ (resp.~$j_{r+1}$) be the largest (resp.~smallest) element of
  $I_0$ which is less than $j_1$ (resp.~greater than $j_r$) then we can write the
  $1$-arrows on the spine of $\langle I_0\mid j_1\mid j_2\mid\ldots\mid j_r\rangle$
  as whiskered composites
  \begin{equation*}
    \langle  I_0\cup\{j_1,\ldots,j_{l-1}\}\mid j_l \rangle =
    \langle (I_0\cap[i,j_0])\cup\{j_1,\ldots,j_{l-1}\} \rangle\circ
    \langle j_{l-1},j_{l+1}\mid j_l \rangle \circ
    \langle I_0\cap[j_{l+1},k] \rangle
  \end{equation*}
  for $l=1,\ldots,r$. Applying the simplicial functor $X$ to this data, we find
  that the $1$-arrow $X_{\langle {I_0\mid I_1} \rangle}$ can be expressed as a
  composite of $1$-arrows in the quasi-category $\Fun_{\eC}(X_i,X_k)$, each one
  of which is constructed by whiskering a $1$-arrow $X_{\langle
    {i_{l-1},i_{l+1}\mid i_l \rangle}}$. By assumption these latter arrows are
  all isomorphisms and whisker composition preserves isomorphisms, so it follows
  that we have expressed $X_{\langle {I_0\mid I_1 \rangle}}$ as a composite of
  isomorphisms in the quasi-category $\Fun_{\eC}(X_i,X_k)$ and thus that it too
  is an isomorphism. In this way we have shown that $X$ maps every $1$-arrow in
  $\Fun_{\gC\Del^n}(i,k)$ to an isomorphism in $\Fun_{\eC}(X_i,X_k)$  and thus that it lands in $g_*\eC$ as postulated.
\end{proof}

\begin{rec}[cocones of simplicial sets and their coherent realizations]\label{rec:cocone-notation}
  Given a simplicial set $X$ we consider its join $X\join\Del^0$ with the point, a simplicial set that describes the shape of a cocone under $X$. In forming the join, we have, as usual,  assumed that $X$ and $\Del^0$ are trivially augmented
  with a single $(-1)$-simplex called $*$. As is customary, we
  identify each simplex $x\in X$ with the simplex $(x,*)$ in $X\join\Del^0$,
  thereby regarding $X$ as a simplicial subset of $X\join\Del^0$.

  We note that each non-degenerate $n$-simplex $x\in X$ gives rise to two
  non-degenerate simplices in the join $X\join\Del^0$, these being $(x,*)$ of
  dimension $n$ (already identified with $x$ itself) and $(x,\id_{[0]})$ of
  dimension $n+1$, and these enumerate all of its non-degenerate simplices. We shall
  adopt the notation $\hat{x}$ for the non-degenerate $(n+1)$-simplex
  $(x,\id_{[0]})$ and $\top$ for the $0$-simplex $(*,\id_{[0]})$, the only
  $0$-simplex of $X\join\Del^0$ which is not a member of $X$.

  We may apply homotopy coherent realisation to the canonical inclusion
  $X\subset X\join\Del^0$ to obtain a simplicial subcomputad $\gC{X} \subset
  \gC[X\join\Del^0]$ as discussed in Lemma \ref{lem:subcomputad-inclusion}. The notational conventions of Notation~\ref{ntn:beads} provide a ``naming of parts'' for $\gC[X\join \Del^0]$. Its objects are
  named after the $0$-simplices of $X$ along with one extra object named $\top$.
  Its $r$-beads are either named
  \begin{itemize}
  \item $\langle {x;I_1\mid\ldots\mid I_r} \rangle$ for some non-degenerate
    $n$-simplex $x\in X$ with $n > r$ and some partition $\{I_k\}_{k=1}^r$ of
    $[1,n-1]$ into non-empty subsets, defining an $r$-arrow in the functor
    space $\Fun_{\gC{X}}( x\cdot\face^{\fbv{0}},x\cdot\face^{\fbv{n}})$, or
  \item $\langle {\hat{x};I_1\mid\ldots\mid I_r} \rangle$ for some non-degenerate
    $n$-simplex $x\in X$ with $n \geq r$ and some partition $\{I_k\}_{k=1}^r$ of
    $[1,n]$ into non-empty subsets, defining an $r$-arrow in
    the function complex $\Fun_{\gC{X}}( x\cdot\face^{\fbv{0}},\top)$.
  \end{itemize}
\end{rec}

\begin{defn}[lax cocones]\label{defn:lax-cocone}
  Suppose that $X$ is a simplicial set and that $\eC$ is a category enriched in
  quasi-categories. Then a \emph{lax cocone of shape $X$\/} in $\eC$ is defined
  to be a simplicial functor $\ell^B\colon\gC[X\join\Del^0]\to\eC$ with the property
  that its composite with the inclusion $\gC{X}\subset\gC[X\join\Del^0]$ factors
  through the inclusion $g_*\eC\subseteq \eC$. 
  \[
    \xymatrix{
      & \gC{X} \ar[r]^-{B_\bullet} \ar@{^(->}[d] & g_*\eC \ar@{^(->}[d] \\
      \gC\Del^0 \ar@{^(->}[r]^-{\top} \ar@/_3ex/[rr]_-{\ell^B_\top} &
      \gC[X\join\Del^0] \ar[r]^-{\ell^B} & \eC}
  \]
  The restriction of a lax cocone $\ell^B\colon\gC[X\join\Del^0]\to\eC$ to a
  functor $B_{\bullet}\colon\gC{X}\to g_*\eC$ is called its \emph{base\/} and
  its transpose under the adjunction between homotopy coherent nerve and
  realisation $\gC\dashv\hN$ is denoted $b\colon X\to\qC$. We say that $\ell^B$
  is a lax cocone \emph{under the diagram\/} $b$; the object $B\in\eC$ obtained
  by evaluating $\ell^B$ at the object $\top$ is called the \emph{nadir} of that
  lax cocone.
\end{defn}

\begin{ntn}[naming the parts of a lax cocone]
  The components of a lax cocone $\ell^{B}\colon\gC[X\join\Del^0]\to\eK$ inherit
  the following nomenclature from the conventions established in
  Recollection~\ref{rec:cocone-notation}. This maps:
  \begin{itemize}
  \item the objects of $\gC[X\join\Del^0]$ to objects denoted $B_x$ (for
    each $0$-simplex $x\in X$) and $B_\top$ (or sometimes just $B$),
  \item an $r$-arrow $\langle x; I_1\mid...\mid I_r \rangle$ in
    $\Fun_{\gC[X\join\Del^0]}(x_0,x_n)$, where $x_0$ and $x_n$ are the initial
    and final vertices of the $n$-simplex $x\in X$, to an $r$-arrow $B_{\langle
      x; I_1\mid...\mid I_r \rangle}$ in $\Fun_{\eK}(B_{x_0},B_{x_n})$, and
  \item an $r$-arrow $\langle \hat{x}; I_1\mid...\mid I_r \rangle$ in
    $\Fun_{\gC[X\join\Del^0]}(x_0,\top)$, where $x_0$ is the initial vertex of
    the $n$-simplex $x\in X$, to an $r$-arrow $\ell^B_{\langle
      {\hat{x};I_1\mid...\mid I_r} \rangle}$ in $\Fun_{\eK}(B_{x_0},B_\top)$.
  \end{itemize}
  To simplify this notation in common cases, especially at low dimension, if
  $x\in X$ is an $n$-simplex we shall write $B_x$ for the $(n-1)$-arrow
  $B_{\langle {x;1\mid 2\mid...\mid n-1} \rangle}$ and $\ell^B_{x}$ for the
  $n$-arrow $\ell^B_{\langle \hat{x};1\mid 2\mid...\mid n \rangle}$. Consequently, if
  $f\in X$ is a $1$-simplex with vertices $x$ and $y$ then the associated
  components of the lax cocone $\ell^B$ are named and related as depicted in the
  following diagram:
  \begin{equation*}
    \xymatrix@C=4em@R=0.75em{
      {B_x}\ar[dd]_{B_f}\ar[dr]^{\ell^B_x} & \\
      {}\ar@{}[r]|(0.35){\Downarrow\,\ell^B_f} & {B} \\
      {B_y}\ar[ur]_{\ell^B_y} & 
    }
  \end{equation*}
\end{ntn}

\begin{obs}[whiskering lax cocones]\label{obs:whiskering-cocone}
Let $\ell^A \colon \gC[X \join\Del^0] \to \eC$ be a lax cocone with nadir $\ell^A_\top = A$ and let $f \colon A \to B$ be any map in $\eC$. Then there is a \emph{whiskered lax cocone} $f \cdot \ell^A \colon \gC[X \join\Del^0] \to \eC$ with the same base diagram $A_\bullet \colon \gC{X} \to g_*\eC$ and with nadir $B$, whose components from $x \in X$ to $\top$ are defined by whiskering with $f$:
\[ \Fun_{\gC[X\join\Del^0]}(x,\top) \xrightarrow{\ell^A} \Fun_{\eC}(A_x,A) \xrightarrow{f \circ -} \Fun_{\eC}(A_x,B)\]
\end{obs}

\begin{defn}\label{defn:lax-cocone-sset}
There is an augmented simplicial set $\qop{cocone}_{\eC}(X)$ of lax cocones of shape $X$
  in $\eC$, which has:
  \begin{itemize}
  \item $n$-\textbf{simplices}, for $n \geq 0$, lax cocones in $\eC$ of shape $X\times\Del^n$,
    and
  \item \textbf{action} of a simplicial operator $\alpha\colon[m]\to[n]$ on a
    the lax cocones of shape $X\times\Del^n$ given by precomposition with the
    simplicial functor \[\gC[(X\times\Del^\alpha)\join \Del^0]
    \colon\gC[(X\times\Del^m)\join\Del^0]\to\gC[(X\times\Del^n)\join \Del^0],\]
    \item \textbf{$(-1)$-simplices}, objects of $\eC$ with action of the unique operator $[-1] \to [n]$ given by taking a lax cocone of shape $X \times \Del^n$ to its nadir.
  \end{itemize}
  \end{defn}

It is worth mentioning the following result, though we shall not make use of it here:

\begin{prop}\label{prop:cocone-quasi-category}
  Suppose that $X$ is a simplicial set and that $\eC$ is a category enriched in
  quasi-categories. The simplicial set $\qop{cocone}_{\eC}(X)$ of lax cocones of
  shape $X$ in $\eC$ is a quasi-category and the diagram functor
  $\mathrm{diag}\colon \cocone_{\eC}(X)\to \qC^{X}$, which projects each lax cocone to the
  diagram under which it lives, is an isofibration.
\end{prop}

\begin{proof}
  This follows by much the same homotopy coherent horn filling argument deployed to prove Proposition \ref{prop:qcat-from-kan-enriched}. We leave the details to the reader.
\end{proof}

\subsection{Cocartesian cocones}\label{sec:cocart-cocone-def}

We now axiomatise the properties of the simplicial natural transformation
  \begin{equation*}
    \xymatrix@R=0em@C=8em{
      {\gC[\boundary\Del^n\join\Del^0]}\ar@/_2ex/[]!R(0.5);[r]_{\ell^B}\ar@/^2ex/[]!R(0.5);[r]^{\ell^E}
      \ar@{}[r]|-{\Downarrow\, p} & {\eK}
    }
  \end{equation*}
 of Proposition \ref{prop:outer-horn-extension}: in terminology we presently introduce, the hypotheses of that result ask that $p$ defines a \emph{cocartesian cocone}.

\begin{defn}[cocartesian cocones]\label{defn:cocart-cocone}
    Suppose we are given a simplicial set $X$ and lax cocones $\ell^E,\ell^B
  \colon\gC[X\join\Del^0]\to \eK$ of shape $X$ in an $\infty$-cosmos $\eK$ 
  with bases $E_\bullet$ and $B_\bullet$ respectively. Suppose also that we are given a
  simplicial natural transformation
  \begin{equation*}
    \xymatrix@R=0em@C=8em{
      {\gC[X\join\Del^0]}\ar@/_2ex/[]!R(0.5);[r]_{\ell^B}\ar@/^2ex/[]!R(0.5);[r]^{\ell^E}
      \ar@{}[r]|-{\Downarrow\, p} & {\eK.}
    }
  \end{equation*}
  Then we say that the triple $(\ell^E,\ell^B,p)$ is a \emph{cocartesian cocone} if
  and only if
  \begin{enumerate}[label=(\roman*)]
  \item\label{itm:cocart-cocone-i} the nadir of the natural transformation $p$, that being its component
    $p\colon E\tfib B$ at the object $\top$, is a cocartesian fibration,
  \item\label{itm:cocart-cocone-ii} for all $0$-simplices $x\in X$ the naturality square
    \begin{equation*}
      \xymatrix@=2em{
        {E_x}\ar@{->>}[d]_{p_x}\ar[r]^{\ell^E_x} \pbexcursion & {E}\ar@{->>}[d]^{p} \\
        {B_x}\ar[r]_{\ell^B_x} & {B}
      }
    \end{equation*}
    is a pullback, and
  \item\label{itm:cocart-cocone-iii} for all non-degenerate $1$-simplices $f\colon x\to y\in X$ the $1$-arrow
    \begin{equation*}
      \xymatrix@C=4em@R=0.75em{
        {E_x}\ar[dd]_{E_{f}}\ar[dr]^{\ell^E_x} & \\
        {}\ar@{}[r]|(0.35){\Downarrow\,\ell^E_f} & {E} \\
        {E_y}\ar[ur]_{\ell^E_y} & 
      }
    \end{equation*}
    is $p$-cocartesian.
  \end{enumerate}
  In this situation we also say that the pair $(\ell^E,p)$ defines a \emph{cocartesian
  cocone over} $\ell^B$.
\end{defn}

\begin{obs}\label{obs:cocart-cone-restriction}
Because the defining conditions for cocartesian cocones are given
  componentwise, the class of cocartesian cocones is stable under reindexing: given a cocartesian cocone $(\ell^E,\ell^B,p)$ of shape $Y$ and a simplicial map $f \colon X \to Y$, the 
  whiskered natural transformation
  \begin{equation*}
    \xymatrix@R=0em@C=8em{
      {\gC[X\join\Del^0]}\ar[r]^-{\gC[f\join\Del^0]} &
      {\gC[Y\join\Del^0]}\ar@/_2ex/[]!R(0.5);[r]_{\ell^B}\ar@/^2ex/[]!R(0.5);[r]^{\ell^E}
      \ar@{}[r]|-{\textstyle \Downarrow p} & {\eK}
    }
  \end{equation*}
  defines a cocartesian cocone of shape $X$.
\end{obs}

\begin{lem}[pullbacks of cocartesian cocones]\label{lem:cocart-cocone-pb} Suppose given:
\begin{itemize}
\item a pullback diagram
\begin{equation}\label{eq:cocart-cocone-pb} \xymatrix{ F \ar@{->>}[d]_q \ar[r]^g \pbexcursion & E \ar@{->>}[d]^p \\ A \ar[r]_f & B}\end{equation} in which $p$ and $q$ are cocartesian fibrations;
\item a lax cocone $\ell^A \colon \gC[X\join\Del^0] \to \eK$  with nadir $A$; and
\item a cocartesian cocone $(\ell^E,\ell^B,p)$ whose nadir is $p \colon E \tfib B$ and whose codomain cocone $\ell^B = f \cdot \ell^A$ is obtained from the lax cocone $\ell^A$ by whiskering with $f\colon A \to B$ as in Observation \ref{obs:whiskering-cocone}.
\end{itemize}
Then there is a cocartesian cone $(\ell^F,\ell^A,q)$ whose codomain is $\ell^A$, whose nadir component is $q \colon F \tfib A$, and whose domain component is a lax cocone $\ell^F$ that whiskers with $g$ to the lax cone $\ell^E = g \cdot \ell^F$.
\end{lem}

Conversely, a cocartesian cocone $(\ell^F,\ell^A,q)$ with nadir component $q
\colon F \tfib A$ can be whiskered with a pullback square
\eqref{eq:cocart-cocone-pb} to define a cocartesian cocone $(g\cdot \ell^F,
f\cdot \ell^A, p)$ with nadir $p \colon E \tfib B$ and whose domain and codomain
are whiskered lax cocones as defined in Observation \ref{obs:whiskering-cocone}.

\begin{proof}
  Define the lax cocone $\ell^F \colon \gC[X\join\Del^0] \to \eK$ to agree with
  $\ell^E$ on the full subcategory $\gC{X} \inc \gC[X\join\Del^0]$, i.e., to
  have the same base diagram $E_\bullet \colon \gC{X} \to g_*\eK$, and to have
  nadir $\ell^F_\top \defeq F$. The remaining components of this simplicial
  functor are induced by the pullback diagram
  \[
    \xymatrix{ \Fun_{\gC[X\join\Del^0]}(x,\top) \ar@/^2ex/[drr]^{\ell^E}
      \ar@/_2ex/[ddr]_{\ell_A} \ar@{-->}[dr]^{\ell^F} \\ & \Fun_{\eK}(E_x,F)
      \pbexcursion \ar[r]^{g \circ -} \ar[d]_{q \circ -} \pbexcursion &
      \Fun_{\eK}(E_x,E) \ar@{->>}[d]^{p \circ -} \\ & \Fun_{\eK}(E_x,A) \ar[r]_{f
        \circ -} & \Fun_{\eK}(E_x,B)}
  \]
  The simplicial natural transformation $q$ is defined by $q_x \defeq p_x$ and
  $q_\top \defeq q$; naturality follows easily from the universal property of
  the above pullback. At low dimension we might otherwise depict this
  construction as follows:
  \begin{equation*}
    \begin{xy}
      0;<1.4cm,0cm>:<0cm,0.75cm>::
      *{\xybox{
          \POS(1,-0.5)*+{A_y}="one"
          \POS(0,1.5)*+{A_x}="two"
          \POS(2.25,0.5)*+{A}="three"
          \POS(4,0.5)*+{B}="four"
          \POS(1,4)*+{E_y}="one'"
          \POS(1,3.75)*{\pbcorner}
          \POS(0,6)*+{E_x}="two'"
          \POS(0,6)*{\pbcorner}
          \ar "two"; "one"
          \ar@/_5pt/ "one";"three"_{\ell^A_y}^(0.15){}="otm"
          \ar@/^10pt/ "two";"three"^{\ell^A_x}_(0.6){}="ttm"|(0.42){\hole}
          \ar "three";"four"_(0.5){f}
          \ar@{=>} "ttm"-<0pt,7pt> ; "otm"+<0pt,10pt> ^(0.3){\ell^A_k}
          \POS(2.25,2.5)*+{F}="three'"
          \POS(4,2.5)*+{E}="four'"
          \POS(2.25,2.5)*{\pbcorner}
          \ar "three'";"four'"_{g}
          \ar@{->>} "four'";"four"^{p}
          \ar@/_5pt/ "one'";"four'"^(0.6){\ell^E_{fx'}}|(0.42){}="otm'"
          \ar@/^15pt/ "two'";"four'"^{\ell^E_{fx}}_(0.58){}="ttm'"
          \ar@{->>} "one'";"one"_(0.325){q_y}
          \ar@{->>} "two'";"two"_{q_x}
          \ar@{->>} "three'";"three"^{q}
          \ar "two'";"one'"
          \ar@{=>} "ttm'"-<0pt,4pt> ; "otm'"+<0pt,4pt> _(0.3){\ell^E_{fk}}
          \ar@{.>}@/_6pt/ "one'";"three'"^(0.4){}="otm''"|(0.6){\ell^F_x}
          \ar@{.>}@/^12pt/ "two'";"three'"_(0.65){}="ttm''"|(0.4){\ell^F_y}
          \ar@{:>} "ttm''";"otm''"_(0.2){\ell^F_k}
        }}
    \end{xy}
  \end{equation*}
It remains to verify that the triple $(\ell^F,\ell^A,q)$ so constructed defines a cocartesian cone. Condition \ref{itm:cocart-cocone-i} of Definition \ref{defn:cocart-cocone} is clear by construction, since the pullback $q$ is a cocartesian fibration. For condition \ref{itm:cocart-cocone-ii}, the naturality squares for the simplicial natural transformation $q$ factor through the corresponding ones for $p$ via the pullback  \eqref{eq:cocart-cocone-pb} and hence are pullback squares by the cancellation property for pullbacks. Finally, property \ref{itm:cocart-cocone-iii} is a consequence of Proposition \ref{prop:cart-fib-pullback} since the relevant 1-arrow components of $\ell^F$ whisker with $g$ to the corresponding 1-arrow components of $p$, which are known to be $p$-cocartesian.
\end{proof}

\subsection{Extending cocartesian cocones}\label{sec:cocart-cone}

The outer horn extension result discussed in
Proposition~\ref{prop:outer-horn-extension} immediately provides the following
extension result for cocartesian cocones along inclusions of any shape:

\begin{prop}[extending cocartesian cocones]\label{prop:extending-cocart-cones}
  Suppose that $X$ is a simplicial subset of $Y$ and that we are given a lax
  cocone $\ell^B\colon \gC[Y\join\Del^0]\to \eK$ of shape $Y$ valued in an $\infty$-cosmos $\eK$. Assume also that
 $\ell^E\colon\gC[X\join\Del^0]\to\eK$ and $p\colon \ell^E\Rightarrow \ell^B$
  comprise a cocartesian cocone over the restriction of $\ell^B$ to a lax cocone of
  shape $X$. 
\[
    \xymatrix@R=3em@C=8em{
      {\gC[X \join \Del^0]}\ar@{^(->}[dd]
      \ar@/_2ex/[]!DR(0.5);[dr]_{\ell^B}^{}="two" \ar@/^2ex/[]!DR(0.5);[dr]^{\ell^E}_{}="one"
      \ar@{} "one";"two" \ar@{=>}?(0.25);?(0.75)^{p}\\
       & {\eK} \\
       {\gC[Y\join\Del^0]}\ar@/_2ex/[]!UR(0.5);[ru]_{\ell^B}^{}="four"
       \ar@{.>}@/^2ex/[]!UR(0.5);[ru]_{}="three"
       \ar@{} "three";"four" \ar@{:>}?(0.25);?(0.75)\\
    }
\]
Then we may extend $(\ell^E,p)$ to a cocartesian cocone of shape $Y$ over the
  unrestricted diagram $\ell^B$.
\end{prop}

\begin{proof}
  A standard skeleton-wise argument reduces this result to a verification that
  it holds for each boundary inclusion $\boundary\Del^n\inc\Del^n$. Now the joined inclusion $\boundary\Del^n\join\Del^0\subset
  \Del^n\join\Del^0$ is isomorphic to $\Horn^{n+1,n+1}\subset\Del^{n+1}$ so that
  case reduces to an application of Proposition~\ref{prop:outer-horn-extension}.

  We do have one further thing left to check, this being that the restriction of
  the extended cocone $\ell^E\colon\gC[Y\join\Del^0]\to\eK$ to the simplicial
  subcategory $\gC{Y}\subset\gC[Y\join\Del^0]$ factors though the inclusion
  $g_*\eK\subseteq\eK$. By Lemma~\ref{lem:core-condition}, the following
  lemma, which demonstrates the required result in the case $Y=\Del^2$, suffices.
\end{proof}

\begin{lem}\label{lem:extension-conservativity}
  Suppose that we are given homotopy coherent simplices  $E,B\colon\gC\Del^3\to\eK$ and a
  natural transformation $p\colon E\Rightarrow B$ so that 
  \begin{itemize}
  \item the component $p_3\colon E_3\tfib B_3$ of the natural transformation $p$
    at the object $3\in\gC\Del^3$ is a cocartesian fibration,
  \item the naturality square
    \begin{equation*}
      \xymatrix@=2em{
        {E_2}\pbexcursion\ar[r]^{E_{\langle {2,3} \rangle}}\ar@{->>}[d]_{p_2} &
        {E_3}\ar@{->>}[d]^{p_3} \\
        {B_2}\ar[r]_{B_{\langle {2,3} \rangle}} & {B_3}
      }
    \end{equation*}
    is a pullback,
  \item the $1$-arrows $E_{\langle 1,3\mid 2 \rangle}, E_{\langle {0,3\mid 1}
      \rangle}$ and $E_{\langle {0,3\mid 2} \rangle}$ are $p_3$-cocartesian; and
  \item the $1$-arrow $B_{\langle {0,2\mid 1} \rangle}$ is an isomorphism in the
    function complex $\Fun_{\eK}(B_0,B_2)$.
  \end{itemize}
  Then the $1$-arrow $E_{\langle {0,2\mid 1} \rangle}$ is an isomorphism in the
  function complex $\Fun_{\eK}(E_0,E_2)$.
\end{lem}

Lemma \ref{lem:extension-conservativity} asserts that given the data
$(\ell^E,\ell^B,p)$ of a cocartesian cocone of shape $\Del^2$ where only
$\ell^B$ is assumed to define a lax cocone, then $\ell^E$ is also lax cocone.

\begin{proof} 
  On account of the pullback
  \[
    \xymatrix@=2em{
      {\Fun_{\eK}(E_0,E_2)}\pbexcursion\ar[r]^{E_{\langle {2,3} \rangle}\circ-}\ar@{->>}[d]_{p_2\circ-} &
      {\Fun_{\eK}(E_0,E_3)}\ar@{->>}[d]^{p_3\circ -} \\
      {\Fun_{\eK}(E_0,B_2)}\ar[r]_{B_{\langle {2,3} \rangle}\circ-} & {\Fun_{\eK}(E_0,B_3)}
    }
  \]
  to prove that the 1-arrow $E_{\langle{0,2\mid 1}\rangle}$ is invertible, it
  suffices to demonstrate that its two projections are isomorphisms. By simplicial
  naturality, $p_2 \circ E_{\langle{0,2\mid 1}\rangle} = B_{\langle{0,2\mid
      1}\rangle} \circ p_0$, which is an isomorphism by hypothesis. The other
  projection is the displayed dashed 1-arrow in the diagram $\Fun_{\gC\Del^3}(0,3)
  \to \Fun_{\eK}(E_0,E_3)$ of 0-,1-, and 2-arrows:
  \[
    \xymatrix{ E_{\langle 0,3\rangle} \ar[r]^{E_{\langle{0,3\mid 1}\rangle}}
      \ar[d]_{E_{\langle{0,3\mid 2}\rangle}} \ar[dr]^{E_{\langle{0,3\mid
            1,2}\rangle}} &  E_{\langle 0,1,3\rangle} \ar[d]^{E_{\langle{0,1,3\mid
            2}\rangle}} \\  E_{\langle 0,2,3\rangle}
      \ar@{-->}[r]_{E_{\langle{0,2,3\mid 1}\rangle}} &  E_{\langle 0,1,2,3\rangle}}
  \]
  By hypothesis and Lemma \ref{lem:cocart-closures}\ref{itm:cocart-compose} each
  of the solid arrows are $p_3$-cocartesian. Since $p_3 \circ E_{\langle{0,2,3\mid
      1}\rangle} = B_{\langle{0,2,3\mid 1}\rangle} \circ p_0 $, a whiskered
  composite of $B_{\langle {0,2\mid 1}\rangle}$, is an isomorphism, Lemma
  \ref{lem:cocart-closures}\ref{itm:cocart-conserv} implies that
  $E_{\langle{0,2,3\mid 1}\rangle}$ is an isomorphism; hence $E_{\langle{0,2\mid
      1}\rangle}$ is invertible as claimed.
\end{proof}

\begin{defn}[the space of cocartesian lifts of a lax cocone]\label{defn:cocart-lifts-lax}
  Fix a simplicial set $X$ and define an augmented simplicial set
  $\cocone^{\cattwo}_{\eK}(X)$, the \emph{space of cocartesian cocones of shape
    $X$ in $\eK$}, to have
  \begin{itemize}
  \item $n$-\textbf{simplices}, for $n \geq 0$, the cocartesian cocones of shape $X\times\Del^n$ in
    $\eK$, 
  \item \textbf{action} of a simplicial operator $\alpha\colon[m]\to[n]$
    given by precomposition with the simplicial functor
    \[\gC[(X\times\Del^\alpha)\join\Del^0]\colon \gC[(X\times\Del^m)\join\Del^0]
      \to \gC[(X\times\Del^n)\join\Del^0],\]
  \item $(-1)$-\textbf{simplices} cocartesian fibrations $E \tfib B$ in $\eK$ with
    action of the unique operator $[-1] \to [n]$ given by carrying a cocartesian
    cocone to its component at $\top$.
  \end{itemize}
  The augmented simplicial set $\cocone^{\cattwo}_{\eK}(X)$ decomposes as a
  coproduct of terminally augmented simplicial sets $\cocone^\cattwo_{\eK}(X)
  \cong \coprod_p \cocone^\cattwo_{\eK}(X)_p$ indexed by the cocartesian
  fibrations $p \colon E \tfib B$ appearing as the nadir components.
\end{defn}
  
Projection of a cocartesian cocone $(\ell^E,\ell^B,p)$ onto the lax cocone
$\ell^B$ over which it lives defines a simplicial map
$\cod\colon\cocone^{\cattwo}_{\eK}(X)\to\qop{cocone}_{\eK}(X)$, and an inclusion
of simplicial sets $i \colon X \inc Y$ gives rise to a commutative square
\begin{equation*}
  \xymatrix@R=2em@C=4em{
    {\cocone^{\cattwo}_{\eK}(Y)}\ar[r]^{\cocone^{\cattwo}_{\eK}(i)}\ar[d]_{\cod} &
    {\cocone^{\cattwo}_{\eK}(X)}\ar[d]^{\cod} \\
    {\cocone_{\eK}(Y)}\ar[r]_{\cocone_{\eK}(i)} & {\cocone_{\eK}(X)}
  }
\end{equation*}
of simplicial maps, whose horizontals are given by precomposition as in
Observation~\ref{obs:cocart-cone-restriction} and whose verticals are the
codomain projections of Definition \ref{defn:cocart-lifts-lax}.
  
\begin{thm}\label{thm:cocart-lifts-lax}$\quad$
  \begin{enumerate}[label=(\roman*)]
  \item\label{itm:triv-fib-relative} For any inclusion of simplicial sets $i
    \colon X \inc Y$, the induced
    map \begin{equation}\label{eq:uniqueness-trivial-fibration}
      \xymatrix@R=0em@C=5em{ {\cocone^{\cattwo}_{\eK}(Y)}\ar[r] &
        {\cocone^{\cattwo}_{\eK}(X)\times_{\cocone_{\eK}(X)}{\cocone_{\eK}(Y)}}
      }
    \end{equation}
    is a trivial fibration of augmented simplicial sets.
  \item\label{itm:triv-fib} In particular, taking $X=\emptyset$, the fibre
    $\cocone^\cattwo_{\eK}(Y)_{\langle\ell^B,p\rangle}$ of
    \eqref{eq:uniqueness-trivial-fibration} over a fixed lax cocone $\ell^B$ of
    shape $Y$ and fixed cocartesian fibration $p \colon E \tfib B$ over its
    base, is a contractible Kan complex.
  \end{enumerate}
\end{thm}

We refer to the fibre $\cocone^\cattwo_{\eK}(Y)_{\langle\ell^B,p\rangle}$
described in \ref{itm:triv-fib} as the \emph{space of $p$-cocartesian lifts} of
the lax cocone $\ell_B \colon \gC[Y\join\Del^0] \to \eK$.

\begin{rec}
  Herein a map $p\colon X\to Y$ of augmented simplicial sets is said to be a
  \emph{trivial fibration} if and only if it has the right lifting property with
  respect to all monomorphisms between augmented simplicial sets. Indeed, it is
  enough for it to enjoy the right lifting property with respect to all
  boundary inclusions $\boundary\Del^n\inc\Del^n$ for $n\geq -1$. Equivalently,
  $p$ is a trivial fibration if and only if it acts surjectively on
  $(-1)$-simplices and for each each $(-1)$-simplex $a\in X$ the restricted map of
  components $p_a\colon X_a\to Y_{p(a)}$ is a trivial fibration of simplicial
  sets. 
\end{rec}

\begin{proof}(of Theorem~\ref{itm:triv-fib-relative}) Statement
  \ref{itm:triv-fib} follows from \ref{itm:triv-fib-relative} by specialising to
  the inclusion $\emptyset\inc Y$. To verify
  statement~\ref{itm:triv-fib-relative}, we must show that the map
  in~\eqref{eq:uniqueness-trivial-fibration} acts enjoys the right lifting
  property with respect to the boundary inclusions of the standard simplices at
  dimension $n\geq -1$.

  For any simplicial set $X$, the $(-1)$-simplices of $\cocone_{\eK}(X)$
  correspond bijectively to objects of $\eK$ and those of
  $\cocone^\cattwo_{\eK}(X)$ stand in bijection to the cocartesian fibrations of
  $\eK$. It follows immediately that the map
  in~\eqref{eq:uniqueness-trivial-fibration} acts bijectively on sets of
  $(-1)$-simplices, so it possesses the right lifting property with respect
  to the boundary $\emptyset = \boundary\Del^{-1}\inc\Del^{-1}$.
  
  In the case $n=0$, the lifting property against the inclusion $\Del^{-1}\cong
  \partial\Del^0\inc\Del^0$ is exactly the assertion of
  Proposition~\ref{prop:extending-cocart-cones}: that for any lax cocone $\ell^B
  \colon \gC[Y\join\Del^0] \to \eK$ restricting to define a cocartesian cocone
  of shape $X$ with nadir $p \colon E \tfib B$, this data extends to a
  cocartesian cocone of shape $Y$.

  For $n \geq 1$, we might as well consider a general inclusion of terminally
  augmented simplicial sets $j \colon U\inc V$. Since $U$ and $V$ are terminally
  augmented, maps $U \to \qop{cocone}^{\cattwo}_{\eK}(Y)$ and $V \to
  \qop{cocone}_{\eK}(Y)$ stand in bijection with vertices of
  $\qop{cocone}^{\cattwo}_{\eK}(Y \times U)$ and $\qop{cocone}_{\eK}(Y \times
  V)$ respectively. Applying Proposition~\ref{prop:extending-cocart-cones} to
  the Leibniz product inclusion
  \[i\leib\times j \colon (X\times V)\cup(Y\times U)\inc Y\times V,\]
  we see that the desired extension exists.
\end{proof}

We will repeatedly apply the uniqueness result of Theorem \ref{thm:cocart-lifts-lax} to conclude that functors defined as the bases of lax cocones that define the domain component of an extended cocartesian cocone are unique:

\begin{cor}[uniqueness of cocartesian lifts]\label{cor:unique-cocart-lifts}
Let $i\colon X\inc Y$ be an inclusion of simplicial sets and
  suppose that we are given a lax cocone $\ell^B\colon\gC[Y\join\Del^0]\to\eK$
  of shape $Y$ and a cocartesian cocone $p\colon\ell^E\Rightarrow \ell^B$ of shape
  $X$ over the restriction of $\ell^B$.
  \[
    \xymatrix@R=3em@C=8em{
      {\gC[X \join \Del^0]}\ar@{^(->}[dd]
      \ar@/_2ex/[]!DR(0.5);[dr]_{\ell^B}^{}="two" \ar@/^2ex/[]!DR(0.5);[dr]^{\ell^E}_{}="one"
      \ar@{} "one";"two" \ar@{=>}?(0.25);?(0.75)^{p}\\
      & {\eK} \\
      {\gC[Y\join\Del^0]}\ar@/_2ex/[]!UR(0.5);[ru]_{\ell^B}^{}="four"
      \ar@{.>}@/^2ex/[]!UR(0.5);[ru]_{}="three"
      \ar@{} "three";"four" \ar@{:>}?(0.25);?(0.75)\\
    }
  \]
  Then any pair of functors $\bar{e}, \hat{e}\colon Y\to \qK$ that define the
  bases of cocartesian cocones of shape $Y$ over $\ell^B$ that extend the given
  cocartesian cocone of shape $X$ as in the displayed diagram are isomorphic as
  objects of the functor quasi-category $\qK^Y$. Moreover, this isomorphism lies
  in the fibre of $\qK^Y\tfib\qK^X$ over base $e \colon X \to \qK$ of the domain
  of the given $X$-shaped lax cocone $\ell^E$.
\end{cor}

 In other words, we may regard $\bar{e}$
  and $\hat{e}$ as maps in the simplicial slice category
  $\prescript{X/}{}{\SSet}$ between the objects $i\colon X\inc Y$ and $e\colon
  X\to\qK$, wherein they are related by an isomorphism in the function complex
  between those two objects.

\begin{proof}
Together, the lax cocone $\ell^B\colon\gC[Y\join\Del^0]\to\eK$ of shape $Y$ and cocartesian cocone $p\colon\ell^E\Rightarrow \ell^B$ of shape  $X$ over the restriction of $\ell^B$ determine a vertex of the
  codomain of \eqref{eq:uniqueness-trivial-fibration}. The vertices of the fibre over that point
  are extensions of $p\colon \ell^E\Rightarrow \ell^B$ to a cocartesian cocone
  of shape $Y$ over $\ell^B$. That fibre is a contractible Kan complex, since
  the map in~\eqref{eq:uniqueness-trivial-fibration} is a trivial fibration, so it follows that any two
  such extensions $\bar{p}\colon\bar{\ell}^E\Rightarrow\ell^B$ and
  $\hat{p}\colon\hat{\ell}^E\Rightarrow\ell^B$ are related by an isomorphism
  which restricts to an identity on $X$.

  Projecting these cocartesian cocones to their domain lax cocones and then to their bases, we obtain diagrams $\bar{e},
  \hat{e}\colon Y\to \qK$ which both extend the diagram $e\colon X\to\qK$ and an
  isomorphism $\bar{e}\cong\hat{e}$ in the fibre of the restriction isofibration
  $\qK^Y\tfib\qK^X$ over the vertex $e$ as claimed.
  \end{proof}


%% file: comprehend.tex

\section{The comprehension construction}\label{sec:comprehension}

For any fixed cocartesian fibration $p\colon E\tfib B$ in an $\infty$-cosmos
$\eK$, the \emph{comprehension construction\/} produces a functor $c_p$ from the
underlying quasi-category of $B$ to the quasi-category of $\infty$-categories
and $\infty$-functors in $\eK$.

The comprehension functor $c_p$ extends pullback in the sense that it maps an
element $b\colon 1\to B$, a vertex in the underlying quasi-category of $B$, to
the corresponding fibre $E_b$ of $p\colon E\tfib B$, the $\infty$-category of
$\eK$ formed by taking the pullback:
\begin{equation*}
  \xymatrix@=2em{
    {E_b}\pbexcursion\ar@{->>}[d]\ar[r]^{\ell_{\hat{b}}} &
    {E}\ar@{->>}[d]^{p} \\
    {1}\ar[r]_{b} & {B}
  }
\end{equation*}
Its action on 1-arrows $f\colon a\to b$ is defined by lifting $f$ to a
$p$-cocartesian 1-arrow as displayed in the diagram
\begin{equation*}
  \begin{xy}
    0;<1.4cm,0cm>:<0cm,0.75cm>::
    *{\xybox{
        \POS(1,0)*+{1}="one"
        \POS(0,1)*+{1}="two"
        \POS(3,0.5)*+{B}="three"
        \ar@{=} "one";"two"
        \ar@/_5pt/ "one";"three"_{b}^(0.1){}="otm"
        \ar@/^10pt/ "two";"three"^{a}_(0.5){}="ttm"|(0.325){\hole}
        \ar@{=>} "ttm"-<0pt,7pt> ; "otm"+<0pt,10pt> ^(0.3){f}
        \POS(1,2.5)*+{E_{b}}="one'"
        \POS(1,2.5)*{\pbcorner}
        \POS(0,3.5)*+{E_{a}}="two'"
        \POS(0,3.6)*{\pbcorner}
        \POS(3,3)*+{E}="three'"
        \ar@/_5pt/ "one'";"three'"_{\ell_{b}}
        \ar@/^10pt/ "two'";"three'"^{\ell_{a}}_(0.55){}="ttm'"
        \ar@{->>} "one'";"one"
        \ar@{->>} "two'";"two"
        \ar@{->>} "three'";"three"^{p}
        \ar@{..>} "two'";"one'"_*!/^2pt/{\scriptstyle E_f}
        \ar@{..>}@/_10pt/ "two'";"three'"^(0.44){}="otm'"
        \ar@{=>} "ttm'"-<0pt,4pt> ; "otm'"+<0pt,4pt> ^(0.3){\ell_{f}}
      }}
  \end{xy}
\end{equation*}
and then factoring its codomain to obtain the requisite $\infty$-functor $E_f
\colon E_a \to E_b$ between the fibres over $a$ and $b$.

The mechanics of the comprehension construction involve extensions of
\emph{cocartesian cocones} valued in an $\infty$-cosmos, as discussed in
\S\ref{sec:cocones}. The comprehension construction itself is given in
\S\ref{sec:comprehension-finally} first in the form described above and then,
in Theorem \ref{thm:general-comprehension}, in the more general form of a
functor
\[
  \xymatrix@R=0em@C=6em{
    {\Fun_{\eK}(A,B)}\ar[r]^{c_{p,A}} & {\coCart(\qK)_{/A}}
  }
\]
from a function complex in $\eK$ to the quasi-category of cocartesian fibrations
and cartesian functors over $A$.

As explicated in Remark \ref{rmk:unstraightening}, the comprehension functor is
a quasi-categorical analogue of the unstraightening functor of Lurie
\cite{Lurie:2009fk}, with his universal (co)cartesian fibration of
quasi-categories replaced by a generic (co)cartesian fibration in any
$\infty$-cosmos. Interestingly, our comprehension functor can also be
interpreted as defining the action-on-objects of Lurie's straightening functor,
as we explain in Remark \ref{rmk:straightening}.

In \S\ref{sec:yoneda}, we then  apply the comprehension construction of Theorem \ref{thm:general-comprehension} in order to construct the covariant and contravariant Yoneda embeddings. These functors will be studied further in \S\ref{sec:computing}.

\subsection{The comprehension construction}\label{sec:comprehension-finally}

Assume, for this section, that $\eK$ is a fixed $\infty$-cosmos. Recall from
\ref{ntn:qcat-from-cosmos} that we shall use $\qK$ to denote the homotopy
coherent nerve of the Kan complex enriched core
$g_*\eK$.

\begin{defn}[underlying quasi-category of an
  $\infty$-category]\label{defn:underlying-qcat} For any $\infty$-cosmos $\eK$,
  there is functor of $\infty$-cosmoi $\Fun_{\eK}(1,-) \colon \eK \to \qCat$.
  When $B$ is an object in our $\infty$-cosmos $\eK$, it is natural to think of
  the function complex $\Fun_{\eK}(1,B)$ as being the \emph{underlying
    quasi-category}\footnote{In the $\infty$-cosmoi of complete Segal spaces,
    Segal categories, or naturally marked simplicial sets, the underlying
    quasi-categories defined in this manner are the usual ones; cf.~Examples
    \refIV{ex:CSS-cosmos}, \refIV{ex:segal-cosmos}, \refIV{ex:marked-cosmos}.}
  of $B$, which we shall denote by $\qB$.
\end{defn}

Suppose now that $p\colon E\tfib B$ is a cocartesian fibration in $\eK$. The
\emph{comprehension construction\/} defines a functor $c_p\colon\qB\to\qK$ from
the underlying quasi-category of $\qB$ to the quasi-category $\qK$ associated
with our ambient $\infty$-cosmos, which we call the \emph{comprehension
  functor\/} associated with $p$. It maps a vertex $b\colon 1\to B$ of $\qB$ to
the corresponding fibre $E_b$ of $p\colon E\to B$. Our extension result for
cocartesian cocones, Proposition~\ref{prop:extending-cocart-cones}, provides
precisely the tool we need to extend these constructions to the higher simplices
of $\qB$. We proceed to describe that process formally, beginning by
constructing the lax cocone over which the cocartesian cocone whose base defines
the comprehension functor lives.
  
\begin{defn}\label{defn:cattwo-fun-adjunction}
  Given a simplicial set $U$, let $\cattwo[U]$ denote the simplicial
  category with objects $-$ and $+$ and a single non-trivial function complex
  $\Fun_{\cattwo[U]}(-,+)\defeq U$. This construction is clearly functorial,
  providing us with a functor
  $\cattwo[-]\colon\sSet\to\prescript{\catone+\catone/}{}{\sCat}$ which maps $U$
  to the object $\langle {-,+} \rangle\colon\catone+\catone\to\cattwo[U]$ and
  which preserves all (small) colimits. Indeed, there exists an adjunction
  \begin{equation}\label{eq:cattwo-fun-adjunction}
    \xymatrix@R=0em@C=6em{
      \prescript{\catone+\catone/}{}{\sCat} \ar@/_2ex/[]!R(0.5);[r]_-{\Fun}
      \ar@{}[r]|-{\textstyle\bot} & {\sSet}\ar@/_2ex/[l]!R(0.5)_-{\cattwo[-]}
    }
  \end{equation}
  whose right adjoint carries an object $\langle A,B
  \rangle\colon\catone+\catone\to\eC$ to the function complex $\Fun_{\eC}(A,B)$.
  The component $k_{A,B}\colon\cattwo[\Fun_{\eC}(A,B)]\to\eC$ of its counit at
  the object $\langle A,B \rangle\colon\catone+\catone\to\eC$ is the manifest
  simplicial functor which maps $-\mapsto A$ and $+\mapsto B$ and which acts
  identically on $\Fun_{\cattwo[\Fun_{\eC}(A,B)]}(-,+)=\Fun_{\eC}(A,B)$.
\end{defn}

\begin{obs}\label{obs:lax-to-suspension}
  The functors $\cattwo[-]$ and $\gC[-\join\Del^0]$ may be
  regarded as having codomain $\prescript{\catone/}{}{\sCptd}$ by mapping each
  simplicial set $U$ to the pointed computads $+\colon\catone\to\cattwo[U]$ and
  $\top\colon\catone\to\gC[U\join\Del^0]$ respectively. The functor
  $\gC[-\join\Del^0]\colon\sSet\to\prescript{\catone/}{}{\sCptd}$ preserves all small colimits and is the left
  Kan extension of its restriction to the full subcategory of representables.
  It follows that we may construct a comparison natural transformation
  \begin{equation}\label{eq:t-defn}
    \xymatrix@R=0em@C=10em{
      {\sSet}\ar@/^2ex/[r]!L(0.6)^-{\gC[-\join\Del^0]}
      \ar@/_2ex/[r]!L(0.6)_-{\cattwo[-]}
      \ar@{}[r]|-{\Downarrow t}
      & {\prescript{\catone/}{}{\sCptd}}
    }
  \end{equation}
  by specifying it on representables and left Kan extending to all simplicial sets.
  
  To that end let $\cattwo[n]$ denote the locally ordered category whose objects
  are named $-$ and $+$ and whose only non-trivial hom-space is
  $\Hom_{\cattwo[n]}(-,+) \defeq[n]$. By taking nerves of its hom-spaces, we
  obtain the simplicial category $\cattwo[\Del^n]$, this being the \emph{generic
    $n$-arrow} in the sense that the $n$-arrows in a simplicial category $\eC$
  correspond bijectively to simplicial functors $\cattwo[\Del^n]\to\eC$. This notation is consistent with that just introduced in Definition \ref{defn:cattwo-fun-adjunction}.
  
  Now we
  may define a locally ordered comparison functor $t^n\colon\oSimp^{n+1}\to
  \cattwo[n]$ which:
  \begin{itemize}
  \item maps the objects $0,1,...,n\in\oSimp^{n+1}$ to the object
    $-\in\cattwo[n]$ and the object $n+1$ to the object $+\in\cattwo[n]$,
  \item maps the unique atomic arrow in $\Hom_{\oSimp}(i,j)$ for $0
    \leq i < j \leq n$ to the identity on $-$, and
  \item maps the unique atomic arrow in $\Hom_{\oSimp}(i,n+1)$ for
    $0\leq i \leq n$ to the arrow $i$ in $\Hom_{\cattwo[n]}(-,+)=[n]$.
  \end{itemize}
  This functor maps an arbitrary arrow $T$ in $\Hom_{\oSimp}(i,n+1)$ to the
  arrow $\max(T\setminus\{{n+1}\})$ in $\Hom_{\cattwo[n]}(-,+)=[n]$, so it
  clearly preserves the ordering on arrows. Hence, upon applying the nerve
  construction to hom-spaces we obtain simplicial computad morphisms $t^n \colon
  \gC[\Del^{n+1}] \to \cattwo[\Del^n]$. Moreover, these functors are natural in
  $[n]\in\Del$, so we obtain a natural transformation
  \begin{equation*}
    \xymatrix@R=0em@C=10em{
      {\Del}\ar@/^2ex/[r]!L(0.6)^-{\gC[\Del^{\bullet}\join\Del^0]}
      \ar@/_2ex/[r]!L(0.6)_-{\cattwo[\Del^\bullet]}
      \ar@{}[r]|-{\Downarrow t}
      & {\prescript{\catone/}{}{\sCptd}}
    }
  \end{equation*}
  which we may Kan extend to give the comparison depicted in~\eqref{eq:t-defn}.
\end{obs}

\begin{lem}\label{lem:constant-lax-cones}
  Given objects $A,B$ in a simplicial category $\eC$, a simplicial map $f\colon
  X\to\Fun_{\eC}(A,B)$ gives rise to a lax cocone
  \[
    \gC[X\join\Del^0] \xrightarrow{\ell^{f}} \eC
  \]
  of shape $X$, whose base is the constant functor at $A$ and whose nadir is
  $B$.
\end{lem}
\begin{proof}
  Under the adjunction depicted in~\eqref{eq:cattwo-fun-adjunction}, simplicial
  functors $\hat{f}\colon\cattwo[X]\to\eC$ with $\hat{f}(-)=A$ and
  $\hat{f}(+)=B$ stand in bijective correspondence to simplicial maps $f\colon
  X\to\Fun_{\eC}(A,B)$. Composing with the simplicial natural transformation
  defined in Observation \ref{obs:lax-to-suspension}, we obtain a diagram
  \[
    \gC[X\join\Del^0]\stackrel{t^{X}}\longrightarrow
    \cattwo[X]\stackrel{\hat{f}}\longrightarrow \eC
  \]
  of shape $X$, whose base $\gC{X} \to \eC$ is the constant functor at $A$ and
  whose nadir is $B$. The constant functor clearly factors through the
  subcategory $g_*\eC\subset\eC$ and thus this construction defines a lax
  cocone.
\end{proof}

We are now in a position to provide a fully formal description of the promised
comprehension construction:

\begin{defn}[the comprehension construction]\label{defn:basic-comprehension}
  Suppose that $p\colon E\tfib B$ is a cocartesian fibration in our
  $\infty$-cosmos $\eK$. We may apply the construction of
  Lemma~\ref{lem:constant-lax-cones} to the identity map
  $\id_{\qB}\colon\qB\to\Fun_{\eK}(1,B)$ so as to define a lax cocone
  \begin{equation}\label{eq:comprehension-lax-cocone}
    \ell^B \defeq \gC[\qB\join\Del^0] \xrightarrow{t^{\qB}}
    \cattwo[\qB] \xrightarrow{k^B} \eK
  \end{equation}
  of shape $\qB$, whose base is the constant functor at $1$ and whose nadir is
  $B$. By definition the map $k^B$ here is the adjoint transpose of the identity
  $\id_{\qB}$, that is the counit of the
  adjunction~\eqref{eq:cattwo-fun-adjunction} at the object $\langle {1,B}
  \rangle\colon \catone+\catone\to\eK$.

  Applying Proposition~\ref{prop:extending-cocart-cones} to the inclusion
  $\emptyset\inc\qB$, we may extend this information to construct a lax cocone
  $\ell^E\colon \gC[\qB\join\Del^0]\to\eK$ with nadir $E$ and a simplicial
  natural transformation $p\colon \ell^E\to \ell^B$ whose component at $\top$ is
  the specified cocartesian fibration $p\colon E\tfib B$, these comprising a
  cocartesian cocone over $\ell^B$. The \emph{comprehension functor\/}
  $c_p\colon\qB\to\qK$ associated with $p\colon E\tfib B$ is defined to be the
  adjoint transpose of the base $E_\bullet\colon\gC\qB\to g^*\eK$ of the
  cocartesian cocone $\ell^E$ under the adjunction $\gC\dashv\hN$.
\end{defn}


\begin{obs}[comprehension functors are essentially unique]\label{obs:unique-comprehension}
Definition \ref{defn:basic-comprehension} refers to \emph{the\/} comprehension functor
  $c_p\colon\qB\to\qK$ associated with a cocartesian fibration $p\colon E\tfib
  B$ in $\eK$ without discussing the sense in which this construction is unique.
  
  To rectify this omission, consider a situation in which $(\ell^E,p)$ and
  $(\bar\ell^E,\bar{p})$ both define cocartesian cocones over the lax cocone $\ell^B$
\eqref{eq:comprehension-lax-cocone}. These are vertices in the space of $p$-cocartesian cocones over
  $\ell^B$, which is a contractible Kan complex by Theorem \ref{thm:cocart-lifts-lax}\ref{itm:triv-fib}, 
 and
  so they are connected by an (homotopically unique) isomorphism in there. It
  follows, as in Corollary \ref{cor:unique-cocart-lifts}, that we may project this isomorphism
  of defining cocartesian cocones along the domain and diagram functors
  \[ \qop{cocone}_{\eK}^\cattwo(\qB) \xrightarrow{\dom} \qop{cocone}_{\eK}(\qB) \xrightarrow{\mathrm{diag}} \qK^{\qB}\] 
  to provide an isomorphism between the
  associated comprehension functors $c_p,\bar{c}_p\colon \qB\to\qK$ in
  $\qK^{\qB}$. This result demonstrates the precise sense in which the
  comprehension functor associated to $p$ is essentially unique.
\end{obs}

\begin{prop}[comprehension and change of universe]\label{prop:comprehension-cou}
  Suppose that $G\colon\eK\to\eL$ is a functor $\infty$-cosmoi, that $p\colon
  E\tfib B$ is a cocartesian fibration in $\eK$, and that $q\colon F\tfib C$
is the cocartesian fibration in $\eL$ obtained by applying $G$ to $p$.
  Then there exists an invertible 1-arrow
  \begin{equation*}
    \xymatrix@=2em{
      {\qB}\ar[r]^{c_p}\ar[d]_{g}\ar@{}[dr]|{\textstyle\cong} & {\qK}\ar[d]^{G} \\
      {\qC}\ar[r]_{c_{q}} & {\qL}
    }
  \end{equation*}
  in $\qCat$, where the horizontal maps are the comprehension
  functors associated with $p$ and $q$, the left-hand vertical is the map
  \begin{equation*}
    \xymatrix@R=0em@C=6em{
   g \colon    {\qB = \Fun_{\eK}(1,B)}\ar[r]^-{G_{1,B}} &
      {\Fun_{\eL}(G1,GB)\cong\Fun_{\eL}(1,C) = \qC}
    }
  \end{equation*}
  and the right-hand vertical is constructed by applying the homotopy coherent
  nerve functor to $G\colon\eK\to\eL$. 
\end{prop}

\begin{proof}
  The diagram $c_q \circ g$ is the transpose of the base of the restricted
  cocartesian cocone
  \begin{equation*}
    \xymatrix@R=0em@C=8em{
      {\gC[\qB\join\Del^0]}\ar[r]^-{\gC[g\join\Del^0]} &
      {\gC[\qC\join\Del^0]}\ar@/_2ex/[]!R(0.5);[r]_{\ell^C}\ar@/^2ex/[]!R(0.5);[r]^{\ell^F}
      \ar@{}[r]|-{\textstyle \Downarrow q} & {\eL}
    }
  \end{equation*}
  Similarly, the diagram $G \circ c_p$ is the transpose of the basis of the
  whiskered simplicial natural transformation
  \begin{equation*}
    \xymatrix@R=0em@C=8em{
      {\gC[\qB\join\Del^0]}\ar@/_2ex/[]!R(0.5);[r]_{\ell^B}\ar@/^2ex/[]!R(0.5);[r]^{\ell^E}
      \ar@{}[r]|-{\textstyle \Downarrow p} & {\eK} \ar[r]^-G & \eL
    }
  \end{equation*}
  which is also a cocartesian cocone. This latter fact follows directly from
  Lemmas~\ref{lem:functor-preservation}
  and~\ref{lem:functor-preservation-reprise} which imply that postcomposition
  with a functor of $\infty$-cosmoi preserves cocartesian cocones.
  
  The codomain cocones of these two cartesian cocones are the composites on
  the outside of the following diagram:
  \[
    \xymatrix@R=2em@C=4em{
      \gC[\qB \star \Del^0] \ar[r]_-{t^{\qB}} \ar[d]_{\gC[g \join\Del^0]}
      \ar@/^2.5ex/[]!R(0.75);[rr]^-{\ell^B}&
      \cattwo[\qB] \ar[d]_{\cattwo[g]}\ar[r]_-{k^B} &
      \eK \ar[d]^G \\
      \gC[\qC\join\Del^0] \ar[r]^-{t^{\qC}}\ar@/_2.5ex/_-{\ell^C}[]!R(0.75);[rr] &
      \cattwo[\qC] \ar[r]^-{k^C} & \eL}
  \]
  Here the left-hand square commutes by naturality of the transformation $t$
  defined in Observation \ref{obs:lax-to-suspension} while the right-hand square
  commutes by the naturality of the counit of the adjunction
  in~\eqref{eq:cattwo-fun-adjunction}.

  In this way we see that the cartesian cocones depicted above both have the
  same codomain cocone $G\circ\ell^B = \ell^C\circ\gC[g\join\Del^0]$ and nadir
  $q\colon \qF\tfib \qC$. Consequently, each of these is a vertex in the space
  $\qop{cocone}_{\eL}^{\cattwo}(\qB)_{\langle G\circ\ell^B,q\rangle}$ of
  $q$-cocartesian lifts of $G\circ\ell^B$. By Theorem
  \ref{thm:cocart-lifts-lax}\ref{itm:triv-fib}, it follows that these are
  isomorphic and Corollary \ref{cor:unique-cocart-lifts} explains that this
  isomorphism projects along the domain and diagram functors to define an
  isomorphism from the transpose of the base of the former to the transpose of
  the base of the latter, that is from $c_q \circ g$ to $G \circ c_p$ as
  required.
\end{proof}

\begin{prop}[comprehension and change of base]\label{prop:comprehension-cob}
  Suppose that we are given a pullback
  \begin{equation*}
    \xymatrix@=2em{
      {F}\pbexcursion\ar@{->>}[d]_{q}\ar[r]^g &
      {E}\ar@{->>}[d]^{p} \\
      {A}\ar[r]_{f} & {B}
    }
  \end{equation*}
  in the $\infty$-cosmos $\eK$ in which $p$ and thus $q$ are cocartesian
  fibrations. Then there exists an isomorphic $1$-arrow
  \begin{equation*}
    \xymatrix@C=1em@R=2em{
      {\qA}\ar[rr]^f\ar[dr]_-{c_q} & \ar@{}[d]|(0.4){\textstyle\cong} &
      {\qB}\ar[dl]^-{c_p} \\
      & {\qK} &
    }
  \end{equation*}
  in $\qCat$.
\end{prop}

\begin{proof}
  We argue that $c_q$ and $c_p \circ f$ transpose to define simplicial functors
  $\gC\qA \to \eK$ that each serve as the base for the domain cocones of a pair
  of cocartesian cocones with a common nadir $q \colon F \tfib A$ and common
  codomain cocone. Corollary \ref{cor:unique-cocart-lifts} then implies that the
  functors $c_q$ and $c_p \circ f$ are isomorphic in $\qK^{\qA}$.

  The functor $c_q$ is defined as the transpose of the base domain component of
  a cocartesian cocone $(\ell^F,\ell^A,q)$ with nadir $q \colon F \tfib A$,
  and whose codomain cocone is the composite
  \[
    \ell^A \colon \gC[\qA \star \Del^0]\xrightarrow{t^{\qA}}
    \cattwo[\qA] \xrightarrow{k^A} \eK
  \]
  with base the constant functor at $1$ and nadir $A$. Similarly, the functor
  $c_p$ is defined as the transpose of the base domain component of a
  cocartesian cocone $(\ell^E,\ell^B,p)$ with nadir $p \colon E \tfib B$, and
  whose codomain cocone is the composite
  \[
    \ell^B \colon \gC[\qB \star \Del^0] \xrightarrow{t^{\qB}}
    \cattwo[\qB] \xrightarrow{k^B} \eK
  \]
  with base the constant functor at $1$ and nadir $B$.

  The functor of $\infty$-categories $f \colon A \to B$ induces a functor of
  underlying quasi-categories $f \colon \qA \to \qB$ making the diagram of
  simplicial categories commute
  \[ \xymatrix@R=1em{
      \gC[\qA \star \Del^0] \ar[r]^-{t^{\qA}} \ar[dd]_{\gC[f \join\Del^0]} &
      \cattwo[\qA] \ar[dd]_{\cattwo[f]}\ar[drr]^{k^A} \\
      & \ar@{}[rr]|-{\Downarrow f \circ -} & & \eK \\
      \gC[\qB\join\Del^0] \ar[r]_-{t^{\qB}} & \cattwo[\qB] \ar[urr]_{k^B} }
  \]
  up to a simplicial natural transformation in the right-hand triangle defined
  by ``whiskering with $f$''. From this diagram we see, by Observations
  \ref{obs:whiskering-cocone} and \ref{obs:cocart-cone-restriction}, that the
  cocartesian cocone $(\ell^E,\ell^B, p)$ restricts along $\gC[f\join\Del^0]
  \colon \gC[\qA\join\Del^0] \to \gC[\qB\join\Del^0]$ to define a cocartesian
  cocone of shape $\qA$ whose codomain is the lax cocone whose base is the
  constant functor at $1$ and whose nadir is $B$, and this lax cocone is the
  whiskered composite of the lax cocone whose base is the constant functor at
  $1$ and whose nadir is $A$ with $f \colon A \to B$.

  Now Lemma \ref{lem:cocart-cocone-pb} tells us that the cocartesian cocone
  \begin{equation*}
    \xymatrix@R=0em@C=8em{
      {\gC[\qA\join\Del^0]}\ar[r]^-{\gC[f\join\Del^0]} &
      {\gC[\qB\join\Del^0]}\ar@/_2ex/[]!R(0.5);[r]_{\ell^B}\ar@/^2ex/[]!R(0.5);[r]^{\ell^E}
      \ar@{}[r]|-{\Downarrow\, p} & {\eK}
    }
  \end{equation*}
  pulls back along $f$ to define a cocartesian cocone
  $(\bar{\ell}^F,\ell^A,\bar{q})$ with codomain $\ell^A$ and nadir $q$. The
  cocartesian cocones $(\ell^F,\ell^A,q)$ and $(\bar{\ell}^F,\ell^A,\bar{q})$
  each define vertices in the space of $q$-cocartesian lifts of the lax cocone
  $\ell^A$. By Corollary \ref{cor:unique-cocart-lifts}, there is then an
  isomorphism from the transpose of the base of $\ell^F$ to the transpose of the
  base of $\bar{\ell}^F$, that is from $c_q$ to $c_p \circ f$.
\end{proof}

By interpreting the construction of Definition \ref{defn:basic-comprehension} in
a sliced $\infty$-cosmos $\eK_{/A}$, we immediately obtain a generalised
comprehension functor whose domain is a generic function complex
$\Fun_{\eK}(A,B)$. We introduce the following notation for its codomain:

\begin{ntn}
  Extending the notational conventions established in
  Notation~\ref{ntn:qcat-from-cosmos}, we adopt the following nomenclature for
  homotopy coherent nerves of various slices of the $\infty$-cosmos $\eK$. There
  are quasi-categorically enriched subcategories
  \[
    \coCart(\eK)_{/A}\inc \eK_{/A}\hookleftarrow \Cart(\eK)_{/A}
  \]
  spanned by the cocartesian and cartesian fibrations over $A$, respectively,
  and the cartesian functors between them.\footnote{In a sequel, we will prove
    that $\coCart(\eK)_{/A}$ and $\Cart(\eK)_{/A}$ are in fact themselves
    $\infty$-cosmoi, but we shall not need this fact here.} Let:
  \begin{itemize}
  \item $\qK_{/A}$ denote the quasi-category $\hN(g_*(\eK_{/A}))$,
  \item $\coCart(\qK)_{/A}$ denote the quasi-category $\hN(g_*(\coCart(\eK)_{/A}))$,
    and
  \item $\Cart(\qK)_{/A}$ denote the quasi-category $\hN(g_*(\Cart(\eK)_{/A}))$.
  \end{itemize}
\end{ntn}

Given a cocartesian fibration $p\colon E\tfib B$ in $\eK$ and an arbitrary
$\infty$-category $A$, we may generalise the comprehension construction to
provide a comprehension functor of the following form:

\begin{thm}\label{thm:general-comprehension}For any cocartesian fibration $p
  \colon E \tfib B$ in an $\infty$-cosmos $\eK$ and any $\infty$-category $A$,
  there is a functor
  \[
    \xymatrix@R=0em@C=6em{
      {\Fun_{\eK}(A,B)}\ar[r]^{c_{p,A}} & {\coCart(\qK)_{/A}}
    }
  \]
  defined on 0-arrows by mapping a functor $a \colon A \to B$ to the pullback:
  \[
    \xymatrix{
      E_a \ar@{->>}[d]_{p_a} \ar[r]^-{\ell^E_a} \pbexcursion &
      E \ar@{->>}[d]^p \\
      A \ar[r]_a & B}
  \]
  Its action on 1-arrows $f\colon a\to b$ is defined by lifting $f$ to a $p$-cocartesian 1-arrow as displayed in the diagram
  \begin{equation*}
    \begin{xy}
      0;<1.4cm,0cm>:<0cm,0.75cm>::
      *{\xybox{
          \POS(1,0)*+{A}="one"
          \POS(0,1)*+{A}="two"
          \POS(3,0.5)*+{B}="three"
          \ar@{=} "one";"two"
          \ar@/_5pt/ "one";"three"_{b}^(0.1){}="otm"
          \ar@/^10pt/ "two";"three"^{a}_(0.5){}="ttm"|(0.325){\hole}
          \ar@{=>} "ttm"-<0pt,7pt> ; "otm"+<0pt,10pt> ^(0.3){f}
          \POS(1,2.5)*+{E_{b}}="one'"
          \POS(1,2.5)*{\pbcorner}
          \POS(0,3.5)*+{E_{a}}="two'"
          \POS(0,3.6)*{\pbcorner}
          \POS(3,3)*+{E}="three'"
          \ar@/_5pt/ "one'";"three'"_{\ell^E_{{b}}}^(0.1){}="otm'"
          \ar@/^10pt/ "two'";"three'"^{\ell^E_{{a}}}_(0.55){}="ttm'"
          \ar@{->>} "one'";"one"_(.3){p_b}
          \ar@{->>} "two'";"two"_{p_a}
          \ar@{->>} "three'";"three"^{p}
          \ar@{..>} "two'";"one'"_*!/^2pt/{\scriptstyle E_f}
          \ar@{=>} "ttm'"-<0pt,4pt> ; "otm'"+<0pt,4pt> ^(0.3){\ell^E_{{f}}}
        }}
    \end{xy}
  \end{equation*}
  and then factoring its codomain to obtain the requisite cartesian functor $E_f \colon E_a \to E_b$ between the fibres over $a$ and $b$. 
\end{thm}

\begin{proof}
  The product functor $-\times A\colon\eK\to\eK_{/A}$,
  which carries an object $X$ to the projection $\pi\colon X\times A\tfib A$,
  is functor of $\infty$-cosmoi. It carries the cocartesian fibration $p$ in
  $\eK$ to a cocartesian fibration $p\times A\colon E\times A\tfib B\times A$ in
  $\eK_{/A}$, to which we may apply the comprehension construction of
  Definition~\ref{defn:basic-comprehension} to give a functor
  \begin{equation}\label{eq:general-comprehension-functor}
    {\Fun_{\eK}(A,B)}\stackrel{\cong}\longrightarrow
    {\Fun_{\eK_{/A}}\left(
        \vcenter{\xymatrix@1@R=2.5ex{A\ar@{->>}[d]^-{\id_A}\\ A}},
        \vcenter{\xymatrix@1@R=2.5ex{B \times A \ar@{->>}[d]^-{\pi} \\ A}}
      \right)}
    \xrightarrow{\mkern15mu c_{p\times A}\mkern15mu} {\qK_{/A}}
  \end{equation}
  All that remains in order to construct the functor advertised in the statement is to show that this
  latter functor lands in the sub-quasi-category $\coCart(\qK)_{/A}\subseteq
  \qK_{/A}$. 
  
  The functor $c_{p,A}$ maps a $0$-arrow $a\colon A\to B$ to the object in
  $\eK_{/A}$ constructed by forming the left-hand pullback
  \begin{equation*}
    \xymatrix@R=2em@C=3em{
      {E_a}\pbexcursion\ar@{->>}[d]_{p_a}\ar[r]^-{(\ell^E_{{a}}, p_a)} & {E\times A}\ar@{->>}[d]^-{p\times A} \ar[r]^-\pi & E \ar@{->>}[d]^p \\
      {A}\ar[r]_-{(a,\id_A)} & {B\times A} \ar[r]_-\pi & B
    }
  \end{equation*}
  which lies in $\eK_{/A}$. The composite pullback rectangle in $\eK$ reveals that the action on $0$-arrows has the form advertised and produces a cocartesian fibration in $\eK$ with codomain $A$.

  All that remains is to show that
  it also carries each $1$-arrow in $\Fun_{\eK}(A,B)$ to a cartesian functor in
  $\coCart(\eK)_{/A}$. In the construction of \eqref{eq:general-comprehension-functor}, a 1-arrow $f \colon a \To b$ in $\Fun_{\eK}(A,B)$ is identified with the 1-arrow $(f, \id_A) \in \Fun_{\eK_{/A}}(\id_A,\pi)$, which is then lifted, in accordance with the construction of the comprehension functor in Definition \ref{defn:basic-comprehension}, to a $p \times A$-cocartesian 1-arrow as displayed:
  \begin{equation*}
    \begin{xy}
      0;<2cm,0cm>:<0cm,0.75cm>::
      *{\xybox{
          \POS(1,0)*+{A}="one"
          \POS(0,1)*+{A}="two"
          \POS(3,0.5)*+{B \times A}="three"
          \ar@{=} "one";"two"
          \ar@/_5pt/ "one";"three"_{(b,\id_A)}^(0.1){}="otm"
          \ar@/^10pt/ "two";"three"^{(a,\id_A)}_(0.5){}="ttm"|(0.325){\hole}
          \ar@{=>} "ttm"-<0pt,7pt> ; "otm"+<0pt,10pt> ^(0.3){(f, \id_A)}
          \POS(1,2.5)*+{E_{b}}="one'"
          \POS(1,2.5)*{\pbcorner}
          \POS(0,3.5)*+{E_{a}}="two'"
          \POS(0,3.6)*{\pbcorner}
          \POS(3,3)*+{E \times A}="three'"
          \ar@/_5pt/ "one'";"three'"_{(\ell^E_{{b}},p_b)}^(0.1){}="otm'"
          \ar@/^10pt/ "two'";"three'"^{(\ell^E_{{a}},p_a)}_(0.6){}="ttm'"
          \ar@{->>} "one'";"one"_(.3){p_b}
          \ar@{->>} "two'";"two"_{p_a}
          \ar@{->>} "three'";"three"^{p \times A}
          \ar@{..>} "two'";"one'"_*!/^2pt/{\scriptstyle E_f}
          \ar@{=>} "ttm'"-<0pt,4pt> ; "otm'"+<0pt,4pt> ^(0.1){(\ell^E_{{f}},\id_{p_a})}
        }}
    \end{xy}
  \end{equation*}
  Proposition \ref{prop:cart-fib-pullback} tells us that $(\ell^E_{a},p_a)$ is a
  cartesian functor. Since $(\ell^E_{{f}},\id_{p_A})$ is a $p$-cocartesian
  1-arrow, Lemma \ref{lem:cocartesian-interchange} tells us that
  $(\ell^E_{{b}},p_b)\circ E_f$ is also a cartesian functor. Now Proposition
  \ref{prop:cart-fib-pullback} again tells us that $(\ell^E_{{b}},p_b)$ is a
  cartesian functor that creates $p_b$-cocartesian 1-arrows, so it follows that
  $E_f$ is a cartesian functor, as claimed.
\end{proof}

A similar comprehension construction for cartesian fibrations can be obtained as
an instance of the comprehension construction of Theorem
\ref{thm:general-comprehension} implemented in the dual $\infty$-cosmos $\eK\co$
of Example \ref{ex:dual-cosmoi}. This duality interchanges cartesian and
cocartesian fibrations. One way to see this is to recall that the homotopy
2-category of $\eK\co$ is the 2-categorical dual $\eK_2\co$ of the homotopy
2-category $\eK_2$ of $\eK$ and make use of the 2-categorical definition from
\S\refIV{sec:cartesian}. Another way to see this is to observe that given an
object $B$ in $\eK$ the object $B^{\cattwo}$ is the cotensor with $\cattwo$ both
in $\eK$ and in $\eK\co$, but the duality swaps the roles of its domain and
codomain projections $p_0,p_1\colon B^{\cattwo}\tfib B$.
  
\begin{ntn}
Extending the notational conventions established in
Notation~\ref{ntn:qcat-from-cosmos}, for an $\infty$-cosmos $\eK$, let:
\begin{itemize}
\item $\qK\co$ denote the homotopy coherent nerve $\hN(g_*\eK\co)$,
\item $\qK\co_{/A}$ denote the homotopy coherent nerve of the Kan complex
enriched core of the isomorphic sliced $\infty$-cosmoi $\eK\co_{/A} \cong
(\eK_{/A})\co$, and
\item $\Cart(\qK)_{/A}\co$ denote the homotopy coherent nerve of the Kan complex
enriched core of the isomorphic $\infty$-cosmoi $\coCart(\eK\co)_{/A}
\cong\Cart(\eK)_{/A}\co.$
\end{itemize}
\end{ntn}

\begin{cor}[comprehension for cartesian fibrations]\label{cor:non-co}
  For any cartesian fibration $p \colon E \tfib B$ in an $\infty$-cosmos $\eK$
  and any $\infty$-category $A$, there is a functor
  \[
    \xymatrix@R=0em@C=6em{
      {\Fun_{\eK}(A,B)\op}\ar[r]^{c_{p,A}} & {\Cart(\qK)_{/A}\co}
    }
  \]
  defined on 0-arrows by mapping a functor $a \colon A \to B$ to the pullback:
  \[
    \xymatrix{
      E_a \ar@{->>}[d]_{p_a} \ar[r]^-{\ell^E_a} \pbexcursion &
      E \ar@{->>}[d]^p \\
      A \ar[r]_a & B}
  \]
  Its action on 1-arrows $f\colon b\to a$ is defined by lifting $f$ to a
  $p$-cartesian 1-arrow as displayed in the diagram
  \begin{equation*}
    \begin{xy}
      0;<1.4cm,0cm>:<0cm,0.75cm>::
      *{\xybox{
          \POS(1,0)*+{A}="one"
          \POS(0,1)*+{A}="two"
          \POS(3,0.5)*+{B}="three"
          \ar@{=} "one";"two"
          \ar@/_5pt/ "one";"three"_{b}^(0.1){}="otm"
          \ar@/^10pt/ "two";"three"^{a}_(0.5){}="ttm"|(0.325){\hole}
          \ar@{<=} "ttm"-<0pt,7pt> ; "otm"+<0pt,10pt> ^(0.3){f}
          \POS(1,2.5)*+{E_{b}}="one'"
          \POS(1,2.5)*{\pbcorner}
          \POS(0,3.5)*+{E_{a}}="two'"
          \POS(0,3.6)*{\pbcorner}
          \POS(3,3)*+{E}="three'"
          \ar@/_5pt/ "one'";"three'"_{\ell^E_{b}}^(0.1){}="otm'"
          \ar@/^10pt/ "two'";"three'"^{\ell^E_{a}}_(0.6){}="ttm'"
          \ar@{->>} "one'";"one"_(.3){p_b}
          \ar@{->>} "two'";"two"_{p_a}
          \ar@{->>} "three'";"three"^{p}
          \ar@{..>} "two'";"one'"_*!/^2pt/{\scriptstyle E_f}
          \ar@{<=} "ttm'"-<0pt,4pt> ; "otm'"+<0pt,4pt> ^(0.2){\ell^E_{f}}
        }}
    \end{xy}
  \end{equation*}
  and then factoring its domain to obtain the requisite cartesian functor $E_f \colon E_a \to E_b$ between the fibres over $a$ and $b$. 
\end{cor}
\begin{proof}
  A cartesian fibration $p\colon E \tfib B$ in $\eK$ defines a cocartesian fibration in $\eK\co$. Interpreting the comprehension construction of Theorem \ref{thm:general-comprehension} in $\eK\co$ defines a simplicial functor functor
  \[ \Fun_{\eK\co}(A,B) \to \coCart(\qK\co)_{/A} \subset
    \qK\co_{/A}
  \]
  
  Because the duality isomorphism $\eK\co_{/A} \cong (\eK_{/A})\co$  interchanges cocartesian and cartesian fibrations, $\coCart(\eK\co)_{/A}$ is isomorphic to
  $\Cart(\eK)_{/A}\co$, yielding a similar isomorphism upon passing to maximal Kan complex enriched subcategories and homotopy coherent nerves. Consequently the comprehension construction can be rewritten as
  as
  \begin{equation*}
    \Fun_{\eK}(A,B)\op \to \Cart(\qK)_{/A}\co
    \subset \qK_{/A}\co \qedhere
  \end{equation*}
\end{proof}
  
\begin{rmk}[comprehension for groupoidal (co)cartesian
  fibrations]\label{rmk:groupoidal-comprehension}
  In \S\refIV{ssec:groupoidal}, we introduce special classes of cocartesian or
  cartesian fibrations whose fibres are groupoidal $\infty$-categories. An
  $\infty$-category $E$ is \emph{groupoidal} if $\Fun_{\eK}(X,E)$ is a Kan
  complex for all $X$, and a (co)cartesian fibration $p \colon E \tfib B$ is
  \emph{groupoidal} just when it is a groupoidal object in the slice
  $\infty$-cosmos $\eK_{/B}$. In the $\infty$-cosmos of quasi-categories, the
  groupoidal objects are the Kan complexes, and the groupoidal cocartesian
  fibrations and groupoidal cartesian fibrations coincide, respectively, with
  the \emph{left fibrations} and \emph{right fibrations} of Joyal; see Example
  \refIV{ex:quasi-groupoidal-cart}.
  
  Corollary \refIV{cor:groupoidal-pullback} proves that groupoidal fibrations
  are stable under pullback, so it follows immediately that if $p\colon E \tfib B$
  is a groupoidal (co)cartesian fibration, then its comprehension functor
  $c_{p,A}$ lands in the subcategory of groupoidal (co)cartesian fibrations. Since
  all functors between groupoidal (co)cartesian fibrations are cartesian, the
  groupoidal (co)cartesian fibrations over $A$ define a full subcategory of the
  slice $\eK_{/A}$, which by groupoidal-ness is automatically Kan complex
  enriched. We decline to introduce notation for the large quasi-categories of
  groupoidal cocartesian or groupoidal cartesian fibrations here but comment
  whenever the comprehension functors of Theorem \ref{thm:general-comprehension}
  or Corollary \ref{cor:non-co} land in these subcategories.
\end{rmk}

\begin{ntn}
  We shall use the notation $\qqCat$ to denote the (huge) quasi-category
  constructed by taking the homotopy coherent nerve of the local groupoidal core
  $g_*(\qCat)$. We also use $\qQ$ to denote the full sub-quasi-category of
  $\qqCat$ spanned by the small quasi-categories; in particular $\qQ$ is itself
  an object of $\qCat$. In turn, the full sub-quasi-category of $\qQ$ spanned by
  the small Kan complexes is denoted $\qS$; by tradition this is known as the
  \emph{quasi-category of spaces}.
\end{ntn}

\begin{rmk}[unstraightening of a quasi-category valued
  functor]\label{rmk:unstraightening}
  Given a cocartesian fibration $u\colon \qE\tfib\qQ$ of quasi-categories and a
  quasi-category $\qA$, we may apply Theorem~\ref{thm:general-comprehension} in
  the $\infty$-cosmos $\qCat$ to give a corresponding comprehension functor:
  \begin{equation}\label{eq:unstraightening}
    \xymatrix@R=0em@C=6em{
      {\qQ^\qA = \Fun_{\qqCat}(\qA,\qQ)}\ar[r]^-{c_{u,\qA}} &
      {\coCart(\qqCat)_{/\qA}}
    }
  \end{equation}
  Indeed, there exists a particular cocartesian fibration $u\colon\qQ_*\tfib\qQ$
  with the property that for each quasi-category $\qA$ the comprehension functor
  $c_{u,\qA}$ defines an equivalence between the functor quasi-category
  $\qQ^\qA$ and the full sub-quasi-category of $\coCart(\qqCat)_{/\qA}$ spanned
  by the cocartesian fibrations with small fibres. This provides a purely
  quasi-categorical analogue of Lurie's unstraightening construction,
  ``unstraightening'' a quasi-categorically valued functor $\qA \to \qQ$ into a
  cartesian fibration over $\qA$.

  In a forthcoming paper, we shall prove the classification result alluded to in
  the last paragraph, as an application of the $\infty$-categorical Beck
  monadicity theorem \cite{RiehlVerity:2012hc}. In the meantime, we satisfy
  ourselves here by describing the construction of one model of this classifying
  cocartesian fibration $u\colon\qQ_*\tfib \qQ$. To that end, observe that the
  quasi-category $\qqCat$ of quasi-categories sits inside the homotopy coherent
  nerve $N(\qCat)$ of the full quasi-categorically enriched category of
  quasi-categories. We may form the slice $\slicel{*}{\nrvhc(\qCat)}$ of this
  larger simplicial set under the terminal quasi-category $* \in \qCat$ and
  observe that its $0$-simplices are pointed quasi-categories $(\qB,b)$ and its
  $1$-simplices $(f,\bar{f}) \colon (\qB ,b) \to (\qC,c)$ are pairs consisting
  of a simplicial map $f \colon \qB \to \qC$ together with a 1-simplex $\bar{f}
  \colon c \to f(b)$ in $\qC$. Now the pullback
  \[
    \xymatrix{
      {\qQ_*} \ar@{->>}[d]_{u} \ar[r] \pbexcursion &
      {\slicel{*}{\nrvhc(\qCat)}} \ar[d] \\
      {\qQ} \ar@{^(->}[r] & {\nrvhc(\qCat)}}
  \]
  of the canonical projection associated with the slice
  $\slicel{*}{\nrvhc(\qCat)}$ can be shown to define a cocartesian fibration
  with small fibres over $\qQ$. What is more, by an analogue of Proposition
  \ref{prop:fun-to-r-hom} below, if $\qB$ is a small quasi-category then the
  fibre of this $u\colon\qQ_*\tfib\qQ$ over the corresponding point $\qB\colon
  \Del^0\to\qQ$ is equivalent to $\qB$ itself.
    
  Another case of the comprehension functor defines the ``unstraightening'' of a
  functor valued in spaces:
  \begin{equation}\label{eq:groupoidal-unstraightening}
    \xymatrix@R=0em@C=6em{
      {\qS^\qA}\ar[r]^-{c_{u,\qA}} & {\coCart(\qqCat)_{/\qA}}.
    }
  \end{equation}
  The comprehension functor here is that associated with the forgetful
  cocartesian fibration $u \colon \qS_* \tfib \qS$ from pointed spaces to
  spaces. Here the quasi-category $\qS_*$ can simply be taken to be the comma
  quasi-category $*\comma \qS$ and $p$ to be the associated projection to $\qS$; so
  Example~\refIV{ex:groupoidal-representable} tells us that this is a groupoidal
  cocartesian fibration and thus we find that the associated unstraightening
  functor \eqref{eq:groupoidal-unstraightening} actually lands in the
  quasi-category of groupoidal cocartesian fibrations over $\qA$. Indeed, the
  fibre of $u \colon \qS_* \to \qS$ over a point corresponding to a small Kan
  complex $\qU$ is isomorphic to the hom-space $\Hom_{\qS}(*,\qU)$ discussed in
  Definition~\ref{defn:qcat-comma-hom} below, so Corollary~\ref{cor:fun-to-hom}
  tells us that this is canonically equivalent to $\Fun_{\Kan}(*,\qU)\cong\qU$
  itself.
\end{rmk}

\begin{rmk}[straightening of a fibration]\label{rmk:straightening}
  For any cocartesian fibration $p \colon E \tfib B$, the comprehension
  construction defines a functor
  \[
    c_{p,1}\colon \qB=\Fun_\eK(1,B) \to \qK
  \]
  from the underlying quasi-category of $B$ to the quasi-category of
  $\infty$-categories and $\infty$-functors in $\eK$. It follows that any
  cocartesian fibration $p \colon E \tfib B$ can be ``straightened'' into an
  $\infty$-category-valued functor, namely $c_{p,1}\colon \qB \to \qK$.
  If $p$ is groupoidal, then this functor lands in the homotopy coherent nerve
  of the full subcategory spanned by the groupoidal objects in $\eK$.

  When $p \colon \qE \to \qB$ is a cocartesian fibration of quasi-categories
  with small fibres, the domain of the comprehension functor $c_{p,1} \colon
  \Fun_{\qCat}(1,\qB) \to \qqCat$ is isomorphic to the quasi-category $\qB$ and
  it maps into the quasi-category $\qQ$ of small quasi-categories. When $p$ is
  also groupoidal, the ``straightened'' comprehension functor $c_{p,1} \colon
  \qB \to \qCat$ lands in the quasi-category of spaces $\qS$.
\end{rmk}

\subsection{Yoneda embeddings}\label{sec:yoneda}

A special case of the comprehension construction of Theorem
\ref{thm:general-comprehension}, applied to a canonical groupoidal cocartesian
fibration in a sliced $\infty$-cosmos, defines the covariant Yoneda embedding.
The contravariant Yoneda embedding is dual.

\begin{defn}[the covariant Yoneda embedding]\label{defn:yoneda-embedding}
  For any object $A$ in the $\infty$-cosmos $\eK$, the cotensor $(p_1,p_0) \colon
  A^\cattwo \tfib A \times A$ defines a groupoidal cocartesian fibration
  \begin{equation*}
    \xymatrix{ A^\cattwo \ar@{->>}[rr]^-{(p_1,p_0)} \ar@{->>}[dr]_{p_0}
      & & A \times A \ar@{->>}[dl]^{\pi_0} \\ & A}
  \end{equation*}
  in the slice $\infty$-cosmos $\eK_{/A}$; see Lemma
  \refV{lem:disc-cart-on-right}. By applying the comprehension construction of
  Definition~\ref{defn:basic-comprehension} to this cocartesian fibration we
  obtain a comprehension functor:
  \begin{equation}\label{eq:yoneda-special-case}
    \Fun_{\eK_{/A}}\left(
      \vcenter{\xymatrix@1@R=2.5ex{A\ar@{->>}[d]^-{\id_A}\\ A}},
      \vcenter{\xymatrix@1@R=2.5ex{A \times A \ar@{->>}[d]^-{\pi_0} \\ A}}
    \right)
    \xrightarrow{\mkern15mu c_{(p_0,p_1)} \mkern15mu}  \qK_{/A},
  \end{equation}
  Now the domain of this functor receives a map
  \begin{equation*}
    \qA = \Fun_{\eK}(1,A) \longrightarrow \Fun_{\eK/A}\left(
      \vcenter{\xymatrix@1@R=2.5ex{A\ar@{->>}[d]^-{\id_A}\\ A}},
      \vcenter{\xymatrix@1@R=2.5ex{A \times A \ar@{->>}[d]^-{\pi_0} \\ A}}
    \right)
  \end{equation*}
  which is defined on objects by sending $a \colon 1 \to A$ to:
  \begin{equation*}
    A \cong 1 \times A \xrightarrow{a \times \id_A} A \times A.
  \end{equation*}
  We may may compose with~\eqref{eq:yoneda-special-case} to give a functor
  $\yoneda\colon\qA \to \qK_{/A}$; this acts on a vertex $a\colon 1\to A$ to
  return the pullback of $A^\cattwo\tfib A\times A$ along $a\times \id_A\colon
  1\times A\to A\times A$, this being the groupoidal cartesian fibration
  $A\comma a\tfib A$. Consequently we may restrict the codomain of the functor
  $\yoneda$ to give a functor
  \begin{equation*}
    \yoneda\colon \Fun_{\eK}(1,A) \longrightarrow \Cart(\qK)_{/A} \subset \qK_{/A}
  \end{equation*}
This defines the \emph{covariant Yoneda embedding}.
\end{defn}

\begin{defn}[the contravariant Yoneda embedding]\label{defn:contra-yoneda-embedding}
  Applying the covariant Yoneda construction of Definition
  \ref{defn:yoneda-embedding} in the dual $\infty$-cosmos $\eK\co$ we obtain a
  dual functor, which this time maps each $a\colon1\to A$ to the groupoidal
  cocartesian fibration $a\comma A\tfib A$. This gives rise to an embedding
  \begin{equation*}
    \yoneda\colon \qA\op = \Fun_{\eK\co}(1,A) \longrightarrow \Cart(\qK\co)_{/A} \subset
    \qK\co_{/A}
  \end{equation*}
  The isomorphism $\eK\co_{/A}\cong (\eK_{/A})\co$ interchanges cocartesian and
  cartesian fibrations, so it follows that $\Cart(\eK\co)_{/A}$ is isomorphic to
  $(\coCart(\eK)_{/A})\co$. Consequently the embedding above can be rewritten as
  \begin{equation*}
    \yoneda\colon \qA\op \longrightarrow (\coCart(\qK)_{/A})\co
    \subset (\qK_{/A})\co
  \end{equation*}
  This defines the \emph{contravariant Yoneda embedding}.
\end{defn}


%% file: homs.tex

\section{Computing comprehension}\label{sec:computing}

For any pair of objects $a$ and $b$ in a quasi-category $\qB$, there is a hom-space $\hom_\qB(a,b)$ between them constructed in Definition \ref{defn:qcat-comma-hom} below, and this construction is functorial in maps of quasi-categories. In particular, the comprehension functor $c_p \colon \qB \to \qK$ associated to a cocartesian fibration $p \colon E \tfib B$ in an $\infty$-cosmos $\eK$ induces a map of Kan complexes
\begin{equation}\label{eq:comprehension-action} \Hom(c_p) \colon \Hom_\qB(a,b) \longrightarrow \Hom_\qK(E_a,E_b).\end{equation} Our aim in this section is to provide an explicit description of this functor up to isomorphism. This will enable us in particular to prove that the Yoneda embeddings of Definitions \ref{defn:yoneda-embedding} and \ref{defn:contra-yoneda-embedding} are fully faithful. As a corollary, we conclude that every quasi-category $\qB$ is equivalent to the homotopy coherent nerve of a Kan-complex enriched category, namely the subcategory of $\qCat_{/\qB}$ spanned by the representable cartesian fibrations $\qB \comma b\tfib \qB$.

Our first task, which occupies \S\ref{ssec:qcat-homs}, is to show that the codomain of \eqref{eq:comprehension-action} is equivalent to $\Fun_{g_*\eK}(E_a,E_b)$, the maximal Kan complex contained in the function complex in $\eK$ between the fibres of $p \colon E \tfib B$ over $a,b \colon 1 \to B$. Then in \S\ref{ssec:computing-homs} we construct an action map
\[ m_{a,b} \colon a \comma b \times E_a \longrightarrow E_b.\] Here, the comma $\infty$-category $a \comma b$ defines a groupoidal object in $\eK$ that is naturally regarded as the \emph{internal hom} of $B$ between the objects $a$ and $b$. This map gives rise to a functor
\[ \tilde{m}_{a,b} \colon \Hom_{\qB}(a,b) \longrightarrow \Fun_{g_*\eK}(E_a,E_b),\] which we refer to as the \emph{external action} of the hom-space $\Hom_{\qB}(a,b)$ on the fibres of $p \colon E \tfib B$. The remainder of the section is devoted to a proof of the main theorem:

\begin{thm}[computing the action of comprehension on hom-spaces]\label{thm:comprehension-on-homs}
The action of the comprehension functor $c_p \colon \qB \to \qK$ on the hom-space from $a$ to $b$ is equivalent to the composite of the the external action of $\Hom_{\qB}(a,b)$ on fibres of
  $p\colon E\tfib B$ with the canonical comparison equivalence $\xymatrix{ \Fun_{g_*\eK}(E_a,E_b)\ar@{^(->}[r]^{\simeq} & \Hom_{\qK}(E_a,E_b)}$. That is, there  exists an essentially commutative triangle
\[
    \xymatrix@=2.5em{
      {\Hom_{\qB}(a,b)}\ar[r]^-{\tilde{m}_{a,b}}\ar[dr]_{\Hom(c_p)} &
      {\Fun_{g_*\eK}(E_a,E_b)}\ar@{^(->}[d]^{\simeq}\ar@{}[dl]|(0.35){\cong} \\ 
      & {\Hom_{\qK}(E_a,E_b)}}
\]
\end{thm}

This accomplished, the above-mentioned results about the Yoneda embedding are deduced as easy corollaries.

\subsection{Hom-spaces in quasi-categories}\label{ssec:qcat-homs}

\begin{defn}[hom-spaces in quasi-categories]\label{defn:qcat-comma-hom}
  For any pair of objects $a$ and $b$ in a quasi-category $\qA$, we define
  the \emph{hom-space} between them to be the comma quasi-category displayed by
  the following diagram:
  \begin{equation*}
    \xymatrix@=2em{
      {\Hom_{\qA}(a,b)}\ar@{}[r]|-{\textstyle\defeq} &
      {a\comma b}\ar@{->>}[r]\ar@{->>}[d] & {1}\ar[d]^{a} \\
      {} & {1}\ar[r]_{b} & {\qA}
      \ar@{} [u];[l] \ar@{=>}  ?(0.35);?(0.65)^{\phi}
    }
  \end{equation*}
  These hom-spaces are all Kan complexes (see \refI{obs:pointwise-adjoint-correspondence}) and they may otherwise be thought of as
  being the bi-fibres of the identity module $(p_1,p_0) \colon \qA^\cattwo \tfib \qA \times \qA$
 (see \refV{prop:hom-is-a-module}). Given a functor
  $f\colon\qA\to\qB$, the commutative diagram
  \begin{equation*}
    \xymatrix@=2em{
      {1}\ar[r]^{a}\ar@{=}[d] & {\qA}\ar[d]^f & {1}\ar[l]_{b}\ar@{=}[d] \\
      {1}\ar[r]_{fa} & {\qB} & {1}\ar[l]^{fb}
    }
  \end{equation*}
  gives rise to an induced map $\comma(\id_1,f,\id_1)\colon a\comma b\to
  fa\comma fb$, as described in Definition \ref{defn:comma}, which we shall denote
  $f_{a,b}\colon\Hom_\qA(a,b)\to\Hom_\qB(fa,fb)$.
\end{defn}

  Comprehension functors are defined in a way which tends to preclude direct
  computation of their behaviours. Their actions on hom-spaces, however, are
  much more amenable to direct computation; consequently, properties expressible
  in terms of their actions on hom-spaces will usually be relatively easy to
  verify. As an example of an important property of this kind, we offer up the
  following familiar result:

\begin{lem}\label{lem:equiv-of-qcats}
  A functor $f\colon\qA\to\qB$ of quasi-categories is an equivalence if and only
  if it is 
  \begin{itemize}
  \item \textbf{fully faithful}, in the sense that for all objects $a,b\in\qA$ the
    map $f_{a,b}\colon\Hom_\qA(a,b)\to\Hom_\qB(fa,fb)$ is an equivalence of Kan
    complexes, and
  \item \textbf{essentially surjective}, in the sense that for all objects
    $b\in\qB$ there exists an object $a\in\qA$ and an isomorphism $fa\cong b$ in
    $\qB$.
  \end{itemize}
\end{lem}

\begin{proof}
  The ``only if'' direction is clear from an application of Lemma~\ref{lem:comma}
  and we leave it to the reader. For the ``if'' direction consider the
  commutative diagram
  \begin{equation*}
    \xymatrix@=2em{
      {\qA}\ar[r]^{\id_\qA}\ar@{=}[d] & {\qA}\ar[d]^f & {\qA}\ar[l]_{\id_\qA}\ar@{=}[d] \\
      {\qA}\ar[r]_{f} & {\qB} & {\qA}\ar[l]^{f}
    }
  \end{equation*}
  which induces a module map $\comma(\id_\qA,f,\id_\qA)\colon \qA^{\cattwo}\to
  f\comma f$ whose action on the bi-fibre over the objects $a$ and $b$ is
  $f_{a,b}\colon\Hom_{\qA}(a,b)\to\Hom_{\qB}(fa,fb)$. The assumption that $f$ is
  fully faithful simply postulates that those actions on bi-fibres are all
  equivalences and --- by Corollary \refV{cor:fibrewise-rep}, or better, by a
  result to appear in
  \cite{RiehlVerity:2017ts} 
  which extends the argument used to prove Lemma \refV{lem:qcat-pointwise-rep}
  to the case of a general map between any pair of modules, not necessarily
  representable --- this happens if and only if the module map
  $\comma(\id_\qA,f,\id_\qA)$ itself is an equivalence.

  By factoring $f$ as an isofibration following an equivalence, we may assume
  without loss of generality that it is an isofibration. Our task is now to show
  that $f$ is a trivial fibration. It follows from Lemma~\ref{lem:comma} that
  $\comma(\id_\qA,f,\id_\qA)\colon \qA^{\cattwo}\tfib f\comma f$ is an
  isofibration and is thus a trivial fibration if and only if $f$ is fully
  faithful. Now a routine transposition argument\footnote{The map
    $\comma(\id_\qA,f,\id_\qA)\colon A^{\cattwo}\tfib f\comma f$ is the
    \emph{Leibniz weighted limit} of $f \colon \qA \to \qB$ weighted by
    $\boundary\Del^1\inc\Del^1$; see \cite{RiehlVerity:2013kx}.}
  \[
    \xymatrix{
      \boundary\Del^n \ar@{^(->}[d] \ar[r] &
      \qA^\cattwo \ar@{->>}[d]^{\comma(\id_\qA,f,\id_\qA)}   & &
      \ar@{}[d]|{\displaystyle\leftrightsquigarrow}  &
      \partial\Del^n \times \Del^1 \cup_{\partial\Del^n\times\partial\Del^1}
      \Del^n \times \partial\Del^1 \ar@{^(->}[d] \ar[r] &
      \qA \ar@{->>}[d]^f \\
      \Del^n \ar[r] \ar@{-->}[ur] & f \comma f & &  &
      \Del^n \times \Del^1 \ar[r] \ar@{-->}[ur] & \qB}
  \]
  demonstrates that $\comma(\id_\qA,f,\id_\qA)$ possesses the right lifting
  property with respect to the boundary inclusions
  $\boundary\Del^n\subset\Del^n$ for $n \geq 0$ if and only if $f\colon \qA\tfib
  \qB$ has the right lifting property with respect to the Leibniz product
  inclusions $(\boundary\Del^n\subset\Del^n)\leib\times
  (\boundary\Del^1\subset\Del^1)$ for $n \geq 0$.

  Observe also that when $f\colon \qA\tfib \qB$ is an isofibration then the
  essential surjectivity property implies that $f$ is actually surjective on
  objects; for each object $b\in \qB$ simply lift the isomorphism $fa\cong b$
  using the isofibration property of $f$. In other words, we have shown that $f$
  also possesses the right lifting property with respect to the inclusion
  $\emptyset\subset\Del^0$. As this map together with the Leibniz product
  inclusions mentioned in the previous paragraph generate the class of
  monomorphisms under transfinite composition and pushout, it follows now from
  the fact that $\comma(\id_\qA,f,\id_\qA)$ is a trivial fibration that $f$ is a
  trivial fibration as well, which is what we aimed to show.
\end{proof}
  
We now introduce a pair of alternative models for the hom-space between a pair
of objects in a quasi-category.

\begin{defn}[right and left hom-spaces]\label{defn:l-r-homs}
  Given objects $a$ and $b$ in a quasi-category $\qA$ we define its \emph{right
    hom-space} $\Hom_{\qA}^r(a,b)$ to be the simplicial set with
  \begin{itemize}
  \item $n$-simplices the $(n+1)$-simplices $x\in\qA$ with
    $x\cdot\face^{\fbv{n+1}}=b$ and $x\cdot\face^{n+1}$ the $n$-simplex degenerated
    from $a$, and
  \item action of a simplicial operator $\alpha\colon[m]\to[n]$ on an
    $n$-simplex $x\in \Hom_{\qA}^r(a,b)$ given by $x\cdot(\alpha\join[0])$ in $\qA$.
  \end{itemize}
  It is well-known and easily checked that this is a Kan complex (see e.g., \cite[1.2.2.3]{Lurie:2009fk}).
  
  The \emph{left hom-space}
  $\Hom_{\qA}^l(a,b)$ is defined to be the Kan complex
  $(\Hom_{\qA\op}^r(b,a))\op$.
\end{defn}

Importantly, the hom-Kan complexes of Definitions \ref{defn:qcat-comma-hom} and \ref{defn:l-r-homs} are equivalent:

\begin{prop}[relating hom-spaces]\label{prop:hom-space-comparison}
  Suppose that $a$ and $b$ are objects in a quasi-category $\qA$. Then there
  exists a canonical trivial cofibration\footnote{A map between Kan complexes is a trivial cofibration in the Quillen model structure if and only if it is a trivial cofibration in the Joyal model structure, which is the case just when it is a monomorphism and also an equivalence.}:
\[
    \xymatrix@R=0em@C=5em{
      \Hom^r_{\qA}(a,b)\ar@{^(->}[r]^{\sim} & \Hom_{\qA}(a,b)
    }
\]
\end{prop}

\begin{proof}
  This is a standard result, to be found in \refI{lem:slice-equiv-comma} or \cite[\S 15.4]{Riehl:2014kx} for
  example. We recall just enough of the construction given there to support our
  arguments later.

  Suppose that $U$ is a simplicial set and consider the following pushouts:
  \begin{equation}\label{eq:10}
    \xymatrix@=2em{
      {U+U}\ar@{^(->}[d]\ar[r] &
      {\Del^0+\Del^0}\ar@{^(->}[d] \\
      {U\times\Del^1}\ar[r] & {\Sigma U}\poexcursion
    }\mkern50mu
    \xymatrix@=2em{
      {U}\ar@{^(->}[d]\ar[r] & {\Del^0}\ar@{^(->}[d] \\
      {U\join\Del^0}\ar[r] & {\Sigma^rU}\poexcursion
    }
  \end{equation}
  Both of these have exactly two $0$-simplices, which we shall call $-$ and $+$
  respectively, and in each all non-degenerate $1$-simplices have source $-$ and
  target $+$; consequently we may regard them as being objects $\langle {-,+}
  \rangle \colon\Del^0+\Del^0\inc \Sigma U$ and $\langle {-,+} \rangle \colon \Del^0 +
  \Del^0 \inc \Sigma^rU$ in the slice category $\prescript{\Del^0+\Del^0/}{}{\sSet}$.
  These constructions are functorial in $U$ in an obvious way, giving us a pair of
  functors $\Sigma,\Sigma^r\colon\sSet\to\prescript{\Del^0+\Del^0/}{}{\sSet}$ which preserve
  all  colimits.

  Our interest in the functors $\Sigma$ and $\Sigma^r$ lies in their relationship to the
  hom-space and right hom-space constructions. To be specific, there exists a
  pair of adjunctions
  \begin{equation*}
    \xymatrix@R=0em@C=6em{
      \prescript{\Del^0+\Del^0/}{}{\sSet} \ar@/_2ex/[]!R(0.5);[r]_-{\Hom}
      \ar@{}[r]|-{\textstyle\bot} & {\sSet}\ar@/_2ex/[l]!R(0.5)_-{\Sigma}
    }\mkern30mu
    \xymatrix@R=0em@C=6em{
      \prescript{\Del^0+\Del^0/}{}{\sSet} \ar@/_2ex/[]!R(0.5);[r]_-{\Hom^r}
      \ar@{}[r]|-{\textstyle\bot} & {\sSet}\ar@/_2ex/[l]!R(0.5)_-{\Sigma^r}
    }
  \end{equation*}
  whose right adjoints carry an object $\langle {a,b} \rangle\colon
  \Del^0+\Del^0\to\qA$ to the respective hom-spaces $\Hom_{\qA}(a,b)$ and
  $\Hom^r_{\qA}(a,b)$  introduced in Definitions~\ref{defn:qcat-comma-hom}
  and~\ref{defn:l-r-homs}. It follows that we may define the natural comparison
  $\Hom_{\qA}^r(a,b)\to\Hom_{\qA}(a,b)$ sought in the statement by taking the
mate of a natural transformation
  \begin{equation*}
    \xymatrix@R=0em@C=8em{
      {\sSet}\ar@/^2ex/[r]!L(0.5)^{\Sigma}_{}="one"
      \ar@/_2ex/[r]!L(0.5)_{\Sigma^r}^{}="two" &
      {\prescript{\Del^0+\Del^0/}{}{\sSet}}
      \ar@{} "one";"two" \ar@{=>}  ?(0.25);?(0.75)^{u}
    }
  \end{equation*}
  under those adjunctions. Observe, however, that the functors $\Sigma$ and
  $\Sigma^r$ are left Kan extensions of their restrictions along the Yoneda
  embedding $\Del^\bullet \colon \Del \to \sSet$. So it suffices to specify 
  the natural transformation $u$ on those restrictions and to extend it to all
  simplicial sets using the universal property of $\Sigma$ as a left Kan
  extension. With that in mind observe that there exists an order preserving map
  $u^n\colon[n]\times[1]\to[n+1]$ which maps $(i,0)\mapsto i$ and $(i,1)\mapsto
  n+1$, and this is clearly natural in $[n]\in\Del$. Taking nerves we get maps
  $u^n\colon\Del^n\times\Del^1 \to \Del^n\join\Del^0$ which pass to quotients to
  give a family of maps $u^n\colon \Sigma\Del^n\to \Sigma^r\Del^n$ natural in
  $[n]\in\Del$ as displayed below:
  \[
    \xymatrix{
      \Del^n \times \Del^1 \ar@{->>}[d]_{u^n} \ar@{->>}[r] &
      \Sigma\Del^n \ar@{-->}[d]^{u^n} \\
      \Del^n \join\Del^0 \ar@{->>}[r] & \Sigma^r\Del^n}
  \]
  Note that each of the solid-arrow maps in this defining diagram are
  epimorphisms, and thus so is the induced map $u^n \colon\Sigma\Del^n\to
  \Sigma^r\Del^n$.
  
  By definition, $n$-simplices in $\Hom_{\qA}(a,b)$ (resp.\ $\Hom_{\qA}^r(a,b)$)
  correspond to maps from $\Sigma\Del^n$ (resp.\ $\Sigma^r\Del^n$) to the object
  $\langle {a,b} \rangle\colon\Del^0+\Del^0\to\qA$ in
  $\prescript{\Del^0+\Del^0/}{}{\sSet}$. Furthermore, the corresponding
  component $\hat{u}^{\langle {a,b}
    \rangle}\colon\Hom^r_{\qA}(a,b)\to\Hom_{\qA}(a,b)$ of the mate
  $\hat{u}\colon\Hom^r\Rightarrow\Hom$ of $u\colon\Sigma\Rightarrow\Sigma^r$
  acts on $n$-simplices $x\colon\Sigma^r\Del^n\to\qA$ by precomposing with
  $u^n\colon\Sigma\Del^n\to\Sigma^r\Del^n$; since these maps $u^n$ are
  epimorphisms, the comparison $\Hom^r_{\qA}(a,b)\inc\Hom_{\qA}(a,b)$ is
  injective. What is more, \cite[15.4.7]{Riehl:2014kx} proves that each
  $u^n\colon\Sigma\Del^n\to\Sigma^r\Del^n$ for $n \geq 0$ is a weak equivalence
  in the Joyal model structure.
  
  We have argued that the natural transformation $u\colon\Sigma \Rightarrow
  \Sigma^r \colon \Del \to \prescript{\Del^0+\Del^0/}{}{\sSet}$ is a pointwise
  Joyal weak equivalence. Moreover Lemma \ref{lem:unaugmentable} implies that
  both cosimplicial objects are Reedy cofibrant. Consequently, the following
  lemma establishes that $\hat{u}^{\langle {a,b} \rangle} \colon
  \Hom^r_{\qA}(a,b)\inc\Hom_{\qA}(a,b)$ is a trivial cofibration of Kan
  complexes as required.
\end{proof}

\begin{lem}\label{lem:reedy-trick}
  Suppose that $\lcat{M}$ is a model category, that $\alpha^\bullet\colon
  X^\bullet\to Y^\bullet$ is a map of cosimplicial objects
  $X^\bullet,Y^\bullet\colon\Del\to\lcat{M}$, and that $A$ is a fibrant object
  in $\lcat{M}$. On applying the representable functor $\Hom_{\lcat{M}}(-,A)$ we
  obtain a map of simplicial sets \[\Hom_{\lcat{M}}(\alpha^\bullet,A)\colon
    \Hom_{\lcat{M}}(Y^\bullet,A) \to \Hom_{\lcat{M}}(X^\bullet,A)\] which is:
  \begin{enumerate}[label=(\roman*)]
  \item\label{item:lem:reedy-trick:1} A trivial fibration of simplicial sets
    whenever $\alpha^\bullet$ is a Reedy trivial cofibration, and
  \item\label{item:lem:reedy-trick:2} A weak equivalence in the Joyal (or
    Quillen) model structure whenever $\alpha^\bullet$ is a pointwise weak
    equivalence and $X^\bullet$ and $Y^\bullet$ are Reedy cofibrant.
  \end{enumerate}
\end{lem}

This result follows by a standard Reedy model category theoretic argument, the
reader may wish to consult~\cite{RiehlVerity:2013kx} for the required background
theory. Also see \cite[\S 15.4]{Riehl:2014kx} for further details.

\begin{proof}
  To prove~\ref{item:lem:reedy-trick:1} we assume that $\alpha^\bullet$ is a
  Reedy trivial cofibration, i.e., that its relative latching maps are all
  trivial cofibrations. The functor $\Hom_{\lcat{M}}(-,A)$ carries colimits to
  limits, so it carries the relative latching maps of $\alpha^\bullet$ to the
  relative matching maps of $\Hom_{\lcat{M}}(\alpha^\bullet,A)$, since the
  former are defined in terms of certain colimits and the latter in terms of
  dual limit constructions. What is more, the assumption that $A$ is fibrant is
  equivalent to postulating that $\Hom_{\lcat{M}}(-,A)$ carries trivial
  cofibrations to surjections of sets; so on combining these observations we
  find that the relative matching maps of $\Hom_{\lcat{M}}(\alpha^\bullet,A)$
  are surjections. That, in turn, is equivalent to saying that
  $\Hom_{\lcat{M}}(\alpha^\bullet,A)$ has the right lifting property with
  respect to each $\boundary\Del^n\inc\Del^n$ and thus that it is a trivial
  fibration as required. Result~\ref{item:lem:reedy-trick:2} now follows
  from~\ref{item:lem:reedy-trick:1} by Ken Brown's lemma \cite{Brown:1973zl}.
\end{proof}

\begin{prop}[relating hom-spaces and function complexes]\label{prop:fun-to-r-hom}
  Suppose that $\eC$ is a Kan-complex enriched category and that the
  quasi-category $\qC$ is its homotopy coherent nerve. For each pair of
  objects $A,B\in\eC$ there exists a canonical trivial cofibration:
\[
    \xymatrix@R=0em@C=5em{
      \Fun_{\eC}(A,B)\ar@{^(->}[r]^{\sim} & \Hom^r_{\qC}(A,B)
    }
  \]
\end{prop}

\begin{proof}
  To construct the comparison of the statement, we compare the adjunction
  $\cattwo[-]\dashv\Fun$ of~\eqref{eq:cattwo-fun-adjunction} to that obtained by
  composing the adjunction $\Sigma^r\dashv \Hom^r$ discussed in the proof of
  Proposition \ref{prop:hom-space-comparison} with the adjunction between
  homotopy coherent nerve and realisation $\gC\dashv\hN$:
  \begin{equation}\label{eq:composite-quillen-adj}
    \xymatrix@R=0em@C=6em{
      \prescript{\catone+\catone/}{}{\sCat} \ar@/_2ex/[]!R(0.5);[r]_-{\Hom^r\circ\hN}
      \ar@{}[r]|-{\textstyle\bot} & {\sSet}\ar@/_2ex/[l]!R(0.5)_-{\gC\circ\Sigma^r}
    }
  \end{equation}
  We shall construct the natural inclusion of the statement as the mate of a
  natural transformation $w\colon \gC\circ\Sigma^r\Rightarrow \cattwo[-]$ under
  these adjunctions. To that latter end, given a simplicial set $U$, consider the
  following diagram
  \begin{equation*}
    \xymatrix@=2em{
      {\gC[U]}\ar@{^(->}[d]\ar[r] &
      {\catone}\ar[d]\ar@/^2.5ex/[ddr]^{-} & \\
      {\gC[U\join\Del^0]}\ar[r]\ar@/_1.5ex/[drr]_{t^U} &
      {\gC[\Sigma^rU]}\poexcursion
      \ar@{-->}[dr]^-{w^U} & \\
      && {\cattwo[U]}
    }
  \end{equation*}
  in which the upper-left square is the pushout of simplicial computads obtained
  by applying the left adjoint coherent realisation functor to the defining
  pushout of $\Sigma^r U$ displayed in~\eqref{eq:10}. The lower-diagonal map
  $t^U$ is an instance of the natural comparison constructed in
  Observation~\ref{obs:lax-to-suspension}. To check that the outer quadrilateral
  commutes it is enough to do so for representables $U=\Del^n$, a trivial task
  that follows directly from the explicit description given in
  Observation~\ref{obs:lax-to-suspension}, and then use the uniqueness property
  of natural transformations induced by Kan extensions to extend to all
  simplicial sets. Now apply the universal property of the upper-left pushout to
  induce the dashed comparison $w^U$, which inherits naturality in $U$ from that
  of $t^U$.
  
  Notice that $n$-simplices of $\Fun_{\eC}(A,B)$ (resp.\ $\Hom^r_{\qC}(A,B)$)
  correspond to simplicial functors from $\cattwo[\Del^n]$ (resp.\
  $\gC[\Sigma^r\Del^n]$) to the object $\langle A,B
  \rangle\colon\catone+\catone\to \eC$ in
  $\prescript{\catone+\catone/}{}{\sCat}$. Furthermore, the corresponding
  component $\Fun_{\eC}(A,B)\to\Hom^r_{\qC}(A,B)$ of the mate of
  $w\colon\gC\circ\Sigma^r \Rightarrow\cattwo[-]$ acts on $n$-simplices
  $x\colon\cattwo[\Del^n]\to\eC$ by precomposition with
  $w^n\colon\gC[\Sigma^r\Del^n]\to \cattwo[\Del^n]$. Since the maps $w^n$ are
  epimorphisms, the comparison $\Fun_{\eC}(A,B)\inc\Hom^r_{\qC}(A,B)$ is
  injective. Finally to see that this is an equivalence, observe that the action
  of the simplicial functors $w^n\colon\gC[\Sigma^r\Del^n]\to\cattwo[\Del^n]$ on
  function complexes from $-$ to $+$ provides a map $w^\bullet$ of cosimplicial
  spaces from $\Fun_{\gC[\Sigma^r\Del^\bullet]}(-,+)$ to $\Del^\bullet$. Both of
  these cosimplicial objects are Reedy cofibrant, by
  Lemma~\ref{lem:unaugmentable} (or see \cite[14.3.9,15.4.6]{Riehl:2014kx}), and
  pointwise contractible, from which it follows that $w^\bullet$ is a pointwise
  weak equivalence. So we may apply Lemma~\ref{lem:reedy-trick} to complete our
  proof.
\end{proof}

Combining Propositions \ref{prop:hom-space-comparison} and
\ref{prop:fun-to-r-hom}, we conclude:

\begin{cor}\label{cor:fun-to-hom} 
  Suppose that $\eC$ is a Kan-complex enriched category and that the
  quasi-category $\qC$ is its homotopy coherent nerve. For each pair of
  objects $A,B\in\eC$ there exists a canonical trivial cofibration in the Joyal model structure:
  \[
    \xymatrix@R=0em@C=5em{
      \Fun_{\eC}(A,B)\ar@{^(->}[r]^{\sim} & \Hom_{\qC}(A,B)
    }
  \]
\end{cor}

\subsection{Computing the action of comprehension functors on hom-spaces}\label{ssec:computing-homs}

\begin{defn}[external action of internal homs on fibres]\label{defn:external-action}
  Returning to our running context, let $p\colon E\tfib B$ be a cocartesian
  fibration in our cosmos $\eK$ and fix two objects $a$ and $b$ in the
  underlying quasi-category $\qB=\Fun_{\eK}(1,B)$. Consider the following
  cubical diagram:
  \begin{equation}\label{eq:internal-external}
    \xymatrix@!=1.1em{
      {a\comma b\times E_{a}} \ar[rr]^-{\pi_0} \ar@{.>}[dr]_{m_{a,b}\mkern-8mu}
      \ar@/_2ex/@{.>}[drrr]!L+U(0.5) \ar@{->>}[dd]_{\pi_1} & &
      {E_{a}} \ar@{}[dddr]|(.15){\displaystyle\lrcorner}  \ar[dr]^{E_{\hat{a}}} \ar@{->>}'[d][dd] \\
      & {E_{b}}  \ar@{}[ddrr]|(.15){\displaystyle\lrcorner} 
      \ar@{->>}[dd] \ar[rr]_(0.4){E_{\hat{b}}} & &
      {E} \ar@{->>}[dd]^{p} \\
      {a\comma b} \ar[dr] \ar'[r][rr] & &
      {1} \ar[dr]^(0.4){a} \\
      & {1} \ar[rr]_{b} & &
      {B}
      \ar@{}"1,3";"2,2"\ar@{=>} ?(0.3);?(0.6)_{\chi}
      \ar@{}"3,3";"4,2"\ar@{=>} ?(0.35);?(0.65)_{\phi}
    }
  \end{equation}
  The square on the bottom of this cube is the universal cone displaying the
  comma object $a\comma b$ and those at the front and right hand sides are the
  pullback squares used to define the action of the comprehension functor
  $c_p\colon \qB\to\qK$ on the objects $a$ and $b$. The arrow $\chi$ is chosen
  to be a $p$-cocartesian lift for $\phi$, whose codomain $0$-arrow factors
  through $E_{b}$ under the universal property of the front pullback square to
  give the $0$-arrow $m_{a,b}$.

  By analogy with Definition \ref{defn:qcat-comma-hom}, it is natural to think
  of $a\comma b$ as being the \emph{internal hom\/} of $B$ between the objects
  $a$ and $b$ and to regard $m_{a,b}\colon a\comma b\times E_{a}\to E_{b}$ as
  its \emph{action on fibres\/} of the cocartesian fibration $p\colon E\tfib B$.
  This action of the internal hom on fibres gives rise to a functor:
  \begin{equation*}
    \tilde{m}_{a,b}\defeq\mkern20mu
    \xymatrix@R=0em@C=4em{
      {\Fun_{\eK}(1,a\comma b)}\ar[r]^-{-\times E_a} &
      {\Fun_{\eK}(E_a,a\comma b \times E_a)}\ar[rr]^-{\Fun_{\eK}(E_a,m_{a,b})} &&
      {\Fun_{\eK}(E_a,E_b)}
    }
  \end{equation*}
  Note that the universal property of $a\comma b$ provides a
  canonical isomorphism \[\Fun_{\eK}(1,a\comma b)\cong\Fun_{\eK}(1,a)\comma
  \Fun_{\eK}(1,b) \cong \Hom_{\qB}(a,b),\] the latter isomorphism by Definition \ref{defn:qcat-comma-hom}. Since the domain of $\tilde{m}_{a,b}$ is a Kan complex, it factors through the inclusion $\Fun_{g_*\eK}(E_a,E_b) \inc \Fun_{\eK}(E_a,E_b)$ of the groupoidal core of the codomain. We refer to the functor 
  \[\tilde{m}_{a,b}\colon\Hom_{\qB}(a,b) \to\Fun_{g_*\eK}(E_a,E_b)\] derived from \eqref{eq:internal-external} as the  \emph{external action} of the hom-space $\Hom_{\qB}(a,b)$ on fibres of
  $p\colon E\tfib B$. 
  \end{defn}
  
  To understand how the  external action relates to the action of the comprehension functor
  $c_p\colon\qB\to\qK$ we require various auxiliary structures:

\begin{defn}\label{defn:5}
  Given a simplicial set $U$, let $\catthree[U]$ denote the
  simplicial computad with three objects $-$, $+$, and $\top$, non-trivial
  function complexes given by
  \begin{equation*}
    \Fun_{\catthree[U]}(-,+)\defeq U
    \mkern40mu \Fun_{\catthree[U]}(+,\top)\defeq \Del^0
  \end{equation*}
  and the pushout
  \begin{equation*}
    \xymatrix@C=4em@R=2em{
      {U\times\Del^{\fbv{0}}}\ar@{^(->}[r]\ar[d]_{!} &
      {U\times\Del^1}\ar[d]^{q^U} \\
      {\Del^0}\ar@{^(->}[r] & {\Fun_{\catthree[U]}(-,\top)}\poexcursion
    }
  \end{equation*}
  and composition between those given by the inclusion:
  \begin{equation*}
    \Fun_{\catthree[U]}(-,+)\times\Fun_{\catthree[U]}(+,\bot)\cong
    U\times\Del^{\fbv{1}}\xrightincarrow{\mkern60mu}
    U\times\Del^1\xrightarrow{\mkern20mu q^U\mkern20mu}
    \Fun_{\catthree[U]}(-,\bot)
  \end{equation*}
  More concretely, $\catthree[U]$ has following arrows:
  \begin{itemize}
  \item  a $0$-arrow $i_{+}\in\Fun_{\catthree[U]}(+,\top)$ along with its degenerated images,
    which we shall also call $i_{+}$,
  \item an $r$-arrow $u\in\Fun_{\catthree[U]}(-,+)$ corresponding to each
    $r$-simplex $u\in U$, and
  \item an $r$-arrow $[u,\tau]\in\Fun_{\catthree[U]}(-,\top)$ corresponding to
    each equivalence class of $r$-simplices $(u,\tau)\in U\times\Del^1$ under
    the equivalence relation given by:
    \begin{equation*}
      (u,\tau)\sim(u',\tau')\mkern10mu\text{if and only if}\mkern10mu
      (u,\tau) = (u',\tau') \text{ or } \tau=\tau'=0.
    \end{equation*}
    All but one of these equivalence classes contains a single $r$-simplex,
    we shall adopt the notation $i_{-}$ for the $r$-arrow given by that unique
    non-singleton class $[u,0]$.
  \end{itemize}
  The simplicial actions on these arrows are induced by those on $U$ and
  $U\times\Del^1$ and the composition operation between these function complexes is
  given by $u\circ i_{+} = [u,1]$.
  
  We may extend this construction to a functor $\catthree[-]\colon\sSet\to
  \sCptd$ in the manifest way. This gives rise to an induced functor
  $\catthree[-]\colon\sSet\cong \prescript{\emptyset/}{}{\sSet}\to
  \prescript{\catthree[\emptyset]/}{}{\sCptd}$ of slices which, as one may
  easily verify, possesses a right adjoint and thus preserves all 
  colimits. We shall identify the simplicial category $\cattwo[U]$ of
Definition \ref{defn:cattwo-fun-adjunction} with the full simplicial subcategory of $\catthree[U]$
  spanned by the objects $-$ and $+$.
\end{defn}

\begin{obs}\label{obs:catthree-data}
  A simplicial functor $\ell^C\colon\catthree[U]\to\eC$ may be specified by giving
  the following data:
  \begin{itemize}
  \item three objects $C_{-}, C_{+},$ and $C$ in $\eC$,
  \item a pair of $0$-arrows $\ell^C_{-}\colon C_{-}\to C$ and $\ell^C_{+}\colon
    C_+\to C$ in $\eC$,
  \item a pair of simplicial maps $C_\bullet\colon U\to\Fun_{\eC}(C_{-},C_{+})$ and
    $\ell_\bullet^C\colon U\times\Del^1\to\Fun_{\eC}(C_{-},C)$ making the following
    squares commute:
    \begin{equation}\label{eq:20}
      \xymatrix@R=2em@C=5em{
        {U\times\Del^{\fbv{0}}}\ar@{^(->}[r]\ar[d]_{!} &
        {U\times\Del^1}\ar[d]_{\ell^C_\bullet} &
        {U\times\Del^{\fbv{1}}}\ar@{_(->}[l]\ar[d]^{C_\bullet} \\
        {\Del^0}\ar[r]_-{\ell^C_{-}} &
        {\Fun_{\eC}(C_{-},C)} &
        {\Fun_{\eC}(C_{-},C_{+})}\ar[l]^-{\Fun_{\eC}(C_{-},\ell^C_{+})}
      }
    \end{equation}
  \end{itemize}

  Given two simplicial functors $\ell^C,\ell^D\colon\catthree[U]\to\eC$
  specified by the collections of data discussed in the last paragraph, a
  simplicial natural transformation $t\colon \ell^C\Rightarrow \ell^D$ may be
  presented as a triple of $0$-arrows $t_{-}\colon C_{-}\to D_{-}$, $t_{+}\colon
  C_{+}\to D_{+}$, and $t\colon C\to D$ satisfying the following naturality
  conditions:
  \begin{itemize}
  \item the square
    \begin{equation}\label{eq:21}
      \xymatrix@=2em{
        {C_{+}}\ar[r]^{\ell^C_{+}}\ar[d]_{t_{+}} &
        {C}\ar[d]^t \\
        {D_{+}}\ar[r]_{\ell^D_{+}} & {D}
      }
    \end{equation}
    commutes in $\eC$, and
  \item the following squares commute in $\sSet$:
    \begin{equation}\label{eq:23}
      \xymatrix@C=4.5em@R=2em{
        {U}\ar[r]^-{C_\bullet}\ar[d]_{D_\bullet} &
        {\Fun_{\eC}(C_{-},C_{+})}\ar[d]^{\Fun_{\eC}(C_{-},t_+)} \\
        {\Fun_{\eC}(D_{-},D_{+})}\ar[r]_-{\Fun_{eD}(t_{-},D_{+})} &
        {\Fun_{\eC}(C_{-},D_{+})}
      }
      \xymatrix@C=4.5em@R=2em{
        {U\times\Del^1}\ar[r]^-{\ell^C_\bullet}\ar[d]_{\ell^D_\bullet} &
        {\Fun_{\eC}(C_{-},C)}\ar[d]^{\Fun_{\eC}(C_{-},t)} \\
        {\Fun_{\eC}(D_{-},D)}\ar[r]_-{\Fun_{eD}(t_{-},D)} &
        {\Fun_{\eC}(C_{-},D)}
      }
    \end{equation}
  \end{itemize}
For later reference, note that last of these naturality conditions
  implies that the square
  \begin{equation}\label{eq:22}
    \xymatrix@=2em{
      {C_{-}}\ar[r]^{\ell^C_{-}}\ar[d]_{t_{-}} &
      {C}\ar[d]^t \\
      {D_{-}}\ar[r]_{\ell^D_{-}} & {D}
    }
  \end{equation}
  also commutes in $\eC$.
\end{obs}

\begin{defn}\label{defn:cocart-nat-trans}
  Suppose that we are given a pair of simplicial functors
  $\ell^B,\ell^E\colon\catthree[U]\to\eK$ landing in our
  $\infty$-cosmos $\eK$. Then we say that simplicial natural transformation
  $p\colon\ell^E\Rightarrow\ell^B$ is \emph{cocartesian} if
  \begin{enumerate}[label=(\roman*)]
  \item its component $p\colon E\tfib B$ at $\top$ is a cocartesian fibration,
  \item the naturality squares
    \begin{equation*}
      \xymatrix@=2em{
        {E_{+}}\pbexcursion\ar[r]^{\ell^E_{+}}\ar@{->>}[d]_{p_{+}} &
        {E}\ar@{->>}[d]^p \\
        {B_{+}}\ar[r]_{\ell^B_{+}} & {B}
      }\mkern60mu
    \xymatrix@=2em{
      {E_{-}}\pbexcursion\ar[r]^{\ell^E_{-}}\ar@{->>}[d]_{p_{-}} &
      {E}\ar@{->>}[d]^p \\
      {B_{-}}\ar[r]_{\ell^B_{-}} & {B}
    }
    \end{equation*}
    are pullbacks, and
  \item for each $0$-simplex $u\in U$ the $1$-arrow
    $\hat{u}=(u\cdot\degen^0,\id_{[1]})$ in $\Fun_{\catthree[U]}(-,\top)$ is
    mapped by $\ell^E$ to a $1$-arrow $\ell^E_{\hat{u}}$ in
    $\Fun_{\eK}(E_{-},E)$ which is $p$-cocartesian.
  \end{enumerate}
Cocartesian natural transformations are stable
  under precomposition by the simplicial functor
  $\catthree[f]\colon\catthree[V]\to\catthree[U]$ associated with each
  simplicial map $f\colon V\to U$.  
\end{defn}

\begin{lem}\label{lem:cocart-extending-external}
  Given a cocartesian fibration $p\colon E\tfib B$ in $\eK$ and a pair of
  objects $a,b\colon 1\to B$ in $\qB$, there exists a cocartesian natural
  transformation 
  \begin{equation*}
    \xymatrix@R=2.5em@C=5em{
      {\cattwo[\Hom_{\qB}(a,b)]}
      \ar@{^(->}[d]\ar[r]^-{\hat{m}_{a,b}} & {\eK} \\
      {{\catthree[\Hom_{\qB}(a,b)]}}
      \ar@/^1.5ex/[]!R(0.5);[ur]^(0.45){\ell^E}_(0.45){}="one"
      \ar@/_1.5ex/[]!R(0.5);[ur]_-{\ell^B}^-{}="two" &
      \ar@{}"one";"two"\ar@{=>} ?(0.3);?(0.7) ^{p}
    }
  \end{equation*}
extending the transpose $\hat{m}_{a,b}$ of the external action
  $\tilde{m}_{a,b}\colon\Hom_{\qB}(a,b)\to \Fun_{g_*\eK}(E_a,E_b)$ of
under the adjunction $\cattwo[-]\dashv\Fun$ of
~\ref{eq:cattwo-fun-adjunction}.
\end{lem}

\begin{proof}
  It will be convenient in the remainder of this proof to refer to the functor
  space $\Fun_{\eK}(1,a\comma b)$ rather than to the hom-space $\Hom_{\qB}(a,b)$
  to which it is isomorphic. We utilise the explicit description of
  Observation~\ref{obs:catthree-data} to derive the structures called for in the statement
  from those discussed in Definition \ref{defn:external-action}.

  Define $\ell^B$ to act on objects by mapping $-,+\mapsto 1$ and $\top\mapsto
  B$. This choice uniquely determines its action on the function complex from $-$
  to $+$; after all its target is the function complex $\Fun_{\eK}(1,1)$ which
  is isomorphic to $\Del^0$ because $1$ is terminal. We also define the
  $0$-arrows $\ell^B_{-}\colon 1\to B$ and $\ell^B_{+}\colon 1\to B$ to be $a$
  and $b$ respectively. The $1$-arrow $\phi$ displayed in~\eqref{eq:internal-external} gives
  rise to a simplicial map
  \begin{equation*}
    \ell^B_\bullet\defeq\Fun_{\eK}(1,a\comma b)\times\Del^1
    \xrightarrow{\mkern10mu\Fun_{\eK}(1,a\comma b)\times\phi\mkern10mu}
    \Fun_{\eK}(1,a\comma b)\times\Fun_{\eK}(a\comma b,B)
    \xrightarrow{\mkern10mu\circ\mkern10mu}\Fun_{\eK}(1,B)
  \end{equation*}
that makes the squares in the following diagram commute:
  \begin{equation*}
    \xymatrix@C=2.5em@R=1.5em{
      {\Fun_{\eK}(1,a\comma b)\times\Del^{\fbv{0}}}\ar[d]\ar@{^(->}[r] &
      {\Fun_{\eK}(1,a\comma b)\times\Del^1}\ar[d]^{\ell^B_\bullet} &
      {\Fun_{\eK}(1,a\comma b)\times\Del^{\fbv{1}}}\ar@{_(->}[l]\ar[d] \\
      {\Del^0}\ar[r]_-{a} & \Fun_{\eK}(1,B) & {\Del^0}\ar[l]^-{b}
    }
  \end{equation*}
  This is an appropriate instance of the diagram displayed in~\eqref{eq:20}, so Observation \ref{obs:catthree-data} explains that this data defines a simplicial functor
  $\ell^B\colon\catthree[\Fun_{\eK}(1,a\comma b)]\to\eK$.

  Correspondingly we define $\ell^E$ to act on objects by mapping $-\mapsto
  E_a$, $+\mapsto E_b$, and $\top\mapsto E$ and we define $\ell^E_{-}$ and
  $\ell^E_{+}$ to be the $0$-arrows $E_{\hat{a}}\colon E_a\to E$ and
  $E_{\hat{b}}\colon E_b\to E$, respectively, as named in~\eqref{eq:internal-external}. We then
  take $E_\bullet$ to be the external action $\tilde{m}_{a,b}\colon
  \Fun_{\eK}(1,a\comma b)\to\Fun_{\eK}(E_a,E_b)$ and use the cocartesian
  $1$-arrow $\chi$ of~\eqref{eq:internal-external} to define:
  \begin{equation*}
    \ell^E_\bullet\defeq\Fun_{\eK}(1,a\comma b)\times\Del^1
    \xrightarrow{\mkern8mu(-\times E_a)\times\chi\mkern8mu}
    \Fun_{\eK}(E_a,a\comma b\times E_a)\times\Fun_{\eK}(a\comma b\times E_a,E)
    \xrightarrow{\mkern8mu\circ\mkern8mu}\Fun_{\eK}(E_a,E)
  \end{equation*}
  Given these maps it is routine to check that the following diagram commutes
  \begin{equation*}
    \xymatrix@R=1.5em@C=2.5em{
      {\Fun_{\eK}(1,a\comma b)\times\Del^{\fbv{0}}}\ar[d]\ar@{^(->}[r] &
      {\Fun_{\eK}(1,a\comma b)\times\Del^1}\ar[d]^{\ell^E_\bullet} &
      {\Fun_{\eK}(1,a\comma b)\times\Del^{\fbv{1}}}\ar@{_(->}[l]\ar[d]^{\tilde{m}_{a,b}} \\
      {\Del^0}\ar[r]_-{E_{\hat{a}}} & \Fun_{\eK}(E_a,E) &
      {\Fun_{\eK}(E_a,E_b)}\ar[l]^-{\Fun_{\eK}(E_a,E_{\hat{b}})}
    }
  \end{equation*}
  this being the instance of ~\eqref{eq:20} required to demonstrate that this
  data again assembles into a simplicial functor
  $\ell^E\colon\catthree[\Fun_{\eK}(1,a\comma b)]\to\eK$.
  
  The simplicial natural transformation $p\colon\ell^E\Rightarrow\ell^B$ has
  components $!\colon E_a\tfib 1$, $!\colon E_b\tfib 1$ and $p\colon E\tfib B$
  at the objects $-$, $+$, and $\top$ respectively. For these the naturality
  squares of~\eqref{eq:21} and~\eqref{eq:22} are the defining pullbacks for the
  fibres $E_a$ and $E_b$ as depicted in~\eqref{eq:internal-external}. Furthermore, to verify
  simplicial naturality of this family we must demonstrate the commutativity of
  the following instances of squares shown in~\eqref{eq:23}:
  \begin{equation*}
    \xymatrix@C=4em@R=2em{
      {\Fun_{\eK}(1,a\comma b)}\ar[r]^-{\tilde{m}_{a,b}}\ar[d] &
      {\Fun_{\eK}(E_a,E_b)}\ar[d]^{\Fun_{\eK}(E_a,!)} \\
      {\Fun_{\eK}(1,1)}\ar[r]_{\Fun_{\eK}(!,1)} &
      {\Fun_{\eK}(E_a,1)}
      }
    \xymatrix@R=2em@C=4em{
      {\Fun_{\eK}(1,a\comma b)\times\Del^1}\ar[r]^-{\ell^E_\bullet}\ar[d]_{\ell^B_\bullet} &
      {\Fun_{\eK}(E_a, E)}\ar@{->>}[d]^{\Fun_{\eK}(E_a,p)} \\
      {\Fun_{\eK}(1,B)}\ar[r]_{\Fun_{\eK}(!,B)} & {\Fun_{\eK}(E_a,B)}
    }
  \end{equation*}
  Commutativity of the left-hand square is trivially verified since
  $\Fun_{\eK}(1,1)\cong\Fun_{\eK}(E_a,1)\cong\Del^0$, and that of the right-hand
  square is a matter of routine verification from the definitions given above.
  
  Note also that for each $0$-arrow $f\in\Fun_{\eK}(1,a\comma b)$ the simplicial
  map $\ell^E_\bullet$ carries the $1$-simplex $(f\cdot\degen^0,\id_{[1]})$ in
  $\Fun_{\eK}(1,a\comma b)\times\Del^1$ to the whiskered $1$-arrow
  $\chi\circ(E_a\times f)$. Recall, however, that $\chi$ was chosen to be a
  $p$-cocartesian lift and $p$-cocartesian $1$-arrows are stable under
  precomposition, so it follows that $\chi\circ(E_a\times f)$ is also
  $p$-cocartesian. This completes the verification that the simplicial natural
  transformation $p\colon\ell^E\Rightarrow\ell^B$ is cocartesian as postulated.
\end{proof}


\begin{rmk}\label{rmk:r-hom-instead}
  Restricting along the equivalence $\Hom^r_{\qB}(a,b)\trvcof\Hom_{\qB}(a,b)$,
  we can equally well regard the external action as a functor
  $\tilde{m}_{a,b}\colon\Hom^r_{\qB}(a,b)\to\Fun_{g_*\eK}(E_a,E_b)$. The cocartesian  natural transformation of Lemma \ref{lem:cocart-extending-external}  restricts along  the computad map $\catthree[\Hom^r_{\qB}(a,b)]
  \inc\catthree[\Hom_{\qB}(a,b)]$ to give a cocartesian natural transformation:
  \begin{equation*}
    \xymatrix@R=2.5em@C=5em{
      {\cattwo[\Hom^r_{\qB}(a,b)]}
      \ar@{^(->}[d]\ar[r]^-{\hat{m}_{a,b}} & {\eK} \\
      {{\catthree[\Hom^r_{\qB}(a,b)]}}
      \ar@/^1.5ex/[]!R(0.5);[ur]^(0.45){\ell^E}_(0.45){}="one"
      \ar@/_1.5ex/[]!R(0.5);[ur]_-{\ell^B}^-{}="two" &
      \ar@{}"one";"two"\ar@{=>} ?(0.3);?(0.7) ^{p}
    }
  \end{equation*}
\end{rmk}

The name ``cocartesian natural transformation'' suggests that there should be a connection to the cocartesian cocones of Definition \ref{defn:cocart-cocone}.  In Proposition \ref{prop:cocart-cone-from-nat}, we will show that a cocartesian natural transformation of shape $\catthree[U]$ restricts along a simplicial computad functor defined in the following lemma to a cocartesian cocone of shape $\gC[\Sigma^rU\join\Del^0]$. 

\begin{lem}\label{lem:comparison-extension}
The simplicial functor $w^U\colon\gC[\Sigma^r U]\to\cattwo[U]$ inducing the equivalence of
  Proposition~\ref{prop:fun-to-r-hom} extends along the manifest inclusions to define a simplicial computad functor:
  \begin{equation*}
    \xymatrix@R=2em@C=3em{
      {\gC[\Sigma^rU]}
      \ar@{^(->}[d]\ar[r]^-{w^{U}} &
      {\cattwo[U]}\ar@{^(->}[d] \\
      {\gC[\Sigma^rU\join\Del^0]}\ar@{-->}[r]_-{v^{U}} &
      {{\catthree[U]}}
    }
  \end{equation*}
  and moreover the map  $v^U\colon\gC[\Sigma^r
  U\join\Del^0]\to\catthree[U]$ in $\sCptd$ is natural in $U\in\sSet$.
\end{lem}

\begin{proof}
We first argue that it suffices to define the natural map $v$ on the standard simplices $U=\Del^n$.  To that end, observe that $\gC[\Sigma^r \emptyset \join\Del^0] \cong \gC[\Horn^{2,2}] \cong \catthree[\emptyset]$ so the domain and codomain of $v$ may be regarded as functors
\[ \gC[\Sigma^r(-)\join\Del^0], \catthree[-] \colon \sSet \cong \prescript{\emptyset/}{}{\sSet}\to
  \prescript{\catthree[\emptyset]/}{}{\sCat}.\] These functors both preserve colimits  and are thus left Kan extensions of their
  restrictions along the Yoneda embedding $\Del^\bullet\colon\Del\inc\sSet$.
  Consequently we may again construct any natural transformation between them by
  specifying it on standard simplices $\Del^n$ and Kan extending. Since $\gC[\Sigma^r(-)]$ is similarly defined 
as a left Kan extension along the Yoneda embedding, 
  to verify the commutativity of the square in the statement for an arbitrary
  $U$ it is enough to check it for standard simplices.

  So consider the case $U=\Del^n$. In \eqref{eq:10}, the simplicial set $\Sigma^r
  \Del^n$ is given as a quotient of $\Del^{n+1}$. Since the functor   $\gC[-\join\Del^0]$ preserves the defining pushout \eqref{eq:10}, the simplicial category $\gC[\Sigma^r \Del^n\join\Del^0]$ may be similarly expressed as a quotient
  of $\gC\Del^{n+2}$. Consequently, we
  proceed to define the required simplicial computad morphism
  $v^n\colon\gC[\Sigma^r\Del^n\join\Del^0]\to \catthree[\Del^n]$ by constructing
  a suitable morphism $\bar{v}^n\colon\gC\Del^{n+2}\to \catthree[\Del^n]$ and passing
  to the quotient, as displayed in the diagram:
    \begin{equation}\label{eq:11}
    \xymatrix@=2em{
      {\gC[\Del^n\join\Del^0]}\ar@{^(->}[d]\ar[r] &
      {\gC[\Del^0\join\Del^0]}\ar@{^(->}[d]\ar@/^2ex/[ddr] \\
      {\gC[(\Del^n\join\Del^0)\join\Del^0]}\ar[r]\ar@/_2ex/[rrd]_{\bar{v}^n} &
      {\gC[\Sigma^r\Del^n\join\Del^0]}\poexcursion\ar@{-->}[dr]^{v^n}\\
      && {\catthree[\Del^n]}
    }
  \end{equation}  This latter functor acts on the objects of $\gC\Del^{n+2}$ by
  mapping $0,1,\ldots,n\mapsto -$; $n+1\mapsto +$; and $n+2\mapsto\top$. Its action
  on a $r$-arrow $T^\bullet$ of $\Fun_{\gC\Del^{n+2}}(i,j)$ for $0\leq i<j\leq
  n+2$ is defined by cases as follows:
  \begin{itemize}
  \item $i,j\leq n$, then $\bar{v}^n(i)=\bar{v}^n(j)=-$ and $T^\bullet$ must be mapped to
    the unique $n$-arrow in $\Fun_{\catthree[\Del^n]}(-,-)=\Del^0$,
  \item $i\leq n$ and $j=n+1$, then $\bar{v}^n(i)=-$ and $\bar{v}^n(j)=+$ and we map
    $T^\bullet$ to the $r$-simplex $\alpha\colon[r]\to[n]$ of
    $\Fun_{\catthree[\Del^n]}(-,+)=\Del^n$ given by the formula
    $\alpha(i)\defeq\max(T^i\setminus\{n+1\})$,
  \item $i\leq n$ and $j=n+2$, then $\bar{v}^n(i)=-$ and $\bar{v}^n(j)=\top$ and we map
    $T^\bullet$ to the equivalence class $[\alpha,\rho]$ in
    $\Fun_{\catthree[\Del^n]}(-,\top)$ where $\alpha\colon[r]\to[n]$ and
    $\rho\colon[r]\to[1]$ are the simplicial operators given by the formulae:
    \begin{equation*}
      \alpha(i)\defeq\max(T^i\setminus\{n+1, n+2\}) \mkern40mu
      \rho(i)\defeq
      \begin{cases}
        0 & \text{if $n+1\notin T^i$, and} \\
        1 & \text{if $n+1\in T^i$.}
      \end{cases}
    \end{equation*}
  \item $i=n+1$ and $j=n+2$, then $\bar{v}^n(i)=+$ and $\bar{v}^n(j)=\top$ and we map
    $T^\bullet$ to the unique $r$-arrow in
    $\Fun_{\catthree[\Del^n]}(+,\top)=\Del^0$.
  \end{itemize}
  It is a simple matter to verify that these actions do indeed define a morphism
  of simplicial computads, and their naturality in $[n]\in\Del$ is clear.

  The simplicial category $\gC[\Del^0\join\Del^0]\cong\gC\Del^1$ in the
  top-right corner of \eqref{eq:11} is isomorphic to the generic $0$-arrow $\cattwo[\Del^0]$ and
  the map from there to the bottom-right nadir is that which corresponds to the
  unique $0$-arrow in $\Fun_{\catthree[\Del^n]}(-,\bot)\cong\Del^0$. To show
  that the outer square commutes, it is enough to observe that both of its
  composites $\gC\Del^{n+1}\to\catthree[\Del^n]$ act on
  objects by mapping $0,\ldots,n\mapsto -$ and $n+1\mapsto \top$ and map all
  $r$-arrows in  the function complexes $\Fun_{\gC\Del^{n+1}}(i,n+1)$, for $i\leq n$, to unique $r$-arrow of
  $\Fun_{\catthree[\Del^n]}(-,\top)$ represented by any simplex in
  $\Del^n\times\Del^1$ whose second component is the constant simplicial
  operator $1\colon[r]\to[1]$. This justifies the existence of the induced
  simplicial computad morphism $v^n\colon
  \gC[\Sigma^r\Del^n\join\Del^0]\to\catthree[\Del^n]$, by an application of the
  pushout property of the upper square, whose naturality in $[n]\in\Del$ follows from that of $\bar{v}$. 
   Finally the commutativity of the instance of the square in
  the statement at $U=\Del^n$ is immediate from the specification given above
  and the description of $w^n$ given in Proposition~\ref{prop:fun-to-r-hom} and
  Observation~\ref{obs:lax-to-suspension}.
\end{proof}

\begin{prop}\label{prop:cocart-cone-from-nat}
  Restriction along the simplicial computad map of Lemma~\ref{lem:comparison-extension} sends each
  cocartesian natural transformation $p\colon\ell^E\Rightarrow\ell^B\colon
  \catthree[U]\to\eK$ to a cocartesian cocone 
  \[
  \xymatrix@C=30pt{ \gC[\Sigma^rU\join\Del^0] \ar[r]^-{v^U}&  \catthree[U] \ar@/^2.5ex/[r]^{\ell^E} \ar@{}[r]|{\Downarrow p} \ar@/_2.5ex/[r]_{\ell^B} & \eK}\]
  of shape $\Sigma^r U$ in $\eK$.
\end{prop}

\begin{proof}
On objects the simplicial computad functor $v^U \colon \gC[\Sigma^rU \join\Del^0] \to \catthree[U]$ sends the objects $-$ and $+$ of $\gC[\Sigma^rU]$ to $-$ and $+$ and sends the cocone vertex to $\top$. Thus, conditions (i) and (ii) of Definition \ref{defn:cocart-nat-trans} immediately establish the corresponding conditions (i) and (ii) of Definition \ref{defn:cocart-cocone}. To establish the final condition (iii) of Definition \ref{defn:cocart-cocone}, observe that the only non-degenerate 1-simplices of $\Sigma^rU$ have source $-$ and target $+$. Each of these is the unique quotient of a 1-simplex in $U\join\Del^0$ from a 0-simplex $u \in U$ to the cocone vertex $*$. In Recollection \ref{rec:cocone-notation}, we denoted such 1-simplices by $\hat{u} \defeq (u, \id_{[0]})$. Each such 1-simplex indexes a 1-arrow in $\gC[\Sigma^rU\join\Del^0]$ from $-$ to $\top$ whose codomain factors through the object $+$. The functor $\bar{v}^U \colon \gC[\Sigma^rU \join\Del^0] \to \catthree[U]$ carries the 1-arrow indexed by $\hat{u}$ to the 1-arrow $\hat{u} = (u \cdot \degen^0, \id_{[1]})$ in $\Fun_{\catthree[U]}(-,\top)$, whence we see that condition (iii) of Definition \ref{defn:cocart-nat-trans} implies condition (iii) of Definition \ref{defn:cocart-cocone}.
\end{proof}


In the following proposition, we shall have use for one further natural family
of simplicial comparison functors:

\begin{defn}\label{defn:last-map}
  Given a simplicial set $U$ we define $s^U\colon\catthree[U]\to \cattwo[\Sigma
  U]$ to be the simplicial functor which acts on objects by $-,+\mapsto
  -$ and $\top\mapsto +$, and whose action
  $\Fun_{\catthree[U]}(-,\top)\to\Fun_{\cattwo[\Sigma U]}(-,+)$ is the induced
  map arising from the observation that these function complexes are certain
  compatible quotients of the simplicial set $U\times\Del^1$. 
\end{defn}

\begin{lem}\label{lem:commuting-comparisons}
  Given a simplicial set $U$ there is a commutative square
  \begin{equation}\label{eq:24}
    \xymatrix@=2em{
      {\gC[\Sigma^r U\join \Del^0]}\ar[r]^-{v^U}\ar[d]_-{t^{\Sigma^r U}} &
      {\catthree[U]}\ar[d]^-{s^U} \\
      {\cattwo[\Sigma^r U]} &
      {\cattwo[\Sigma U]}\ar[l]^-{\cattwo[u^U]}
    }
  \end{equation}
  of the maps given in Definition~\ref{defn:last-map}, Lemma~\ref{lem:comparison-extension}, Proposition \ref{prop:hom-space-comparison}, and Observation~\ref{obs:lax-to-suspension}.
\end{lem}

\begin{proof}
As argued in the proof of Lemma \ref{lem:comparison-extension}, we need only
  check the desired commutativity result for representables $U=\Del^n$ and then
  use the uniqueness property of left Kan extensions to infer then that it holds for
  arbitrary simplicial sets.

  Specialising now to the case $U=\Del^n$, observe that both legs of the square
  in~\eqref{eq:24} act on objects to map $-,+\mapsto -$ and $\top\mapsto +$. A moment's reflection reveals that it is enough then to verify that their
  actions on the function complex $\Fun_{\gC[\Sigma^r \Del^n\join \Del^0]}(-,\top)$
  coincide: after all they must act identically on $\Fun_{\gC[\Sigma^r \Del^n\join
    \Del^0]}(-,+)$, simply because they both map it to $\Fun_{\cattwo[\Sigma^r
    \Del^n]}(-,-)\cong\Del^0$, from which it follows that their actions on
  $\Fun_{\gC[\Sigma^r \Del^n\join \Del^0]}(+,\top)$ are determined by those on
  $\Fun_{\gC[\Sigma^r \Del^n\join \Del^0]}(-,\top)$. To compute those latter
  actions consider the following diagram
  \begin{equation*}
    \xymatrix@=1.5em{
      {\Fun_{\gC\Del^{n+2}}(0,n+2)}\ar[rd]\ar[rrr]^{v^n}\ar[ddd]_{t^{n+1}} & & &
      {\Del^n\times\Del^1}\ar[ld]\ar@{=}[ddd] \\
      & {\Fun_{\gC[\Sigma^r \Del^n\join \Del^0]}(-,\top)}
      \ar[r]^-{v^n}\ar[d]_-{t^{\Sigma^r \Del^n}} &
      {\Fun_{\catthree[\Del^n]}(-,\top)}\ar[d]^-{s^n} & \\
      & {\Fun_{\cattwo[\Sigma^r \Del^n]}(-,+)} &
      {\Fun_{\cattwo[\Sigma \Del^n]}(-,+)}\ar[l]^-{\cattwo[u^n]} & \\
      {\Del^{n+1}}\ar[ur] & & & {\Del^n\times\Del^1}\ar[ul]\ar[lll]^{u^n}
    }
  \end{equation*}
  which expresses the various function complexes involved as quotients of the spaces
  upon which the actions of various functors between those spaces were defined.
  Since the upper-left diagonal is a quotient map, we may demonstrate that the
  inner square commutes by showing that the outer square does. To do that it is
  enough to check that the actions of each leg of the outer square coincide on
  $0$-simplices, and that is immediate from the explicit descriptions given in
  the proofs of Proposition~\ref{prop:hom-space-comparison} and Lemma \ref{lem:comparison-extension} and in
  Observation~\ref{obs:lax-to-suspension}.
\end{proof}

Returning to Definition \ref{defn:external-action}, we have now marshalled all of the components
necessary to show that the action of a comprehension functor on a hom-space
coincides with the action of that hom-space on fibres, in the sense made precise
in the following theorem:

{
\renewcommand{\thethm}{\ref{thm:comprehension-on-homs}}
\begin{thm}[computing the action of comprehension on hom-spaces]
  There exists an essentially commutative triangle
  \begin{equation}\label{eq:comprehension-on-homs}
    \xymatrix@=2.5em{
      {\Hom_{\qB}(a,b)}\ar[r]^-{\tilde{m}_{a,b}}\ar[dr]_{\Hom(c_p)} &
      {\Fun_{g_*\eK}(E_a,E_b)}\ar@{^(->}[d]^{\simeq}\ar@{}[dl]|(0.35){\cong} \\ 
      & {\Hom_{\qK}(E_a,E_b)}
    }
  \end{equation}
  in which the right-hand vertical is the trivial cofibration of Corollary \ref{cor:fun-to-hom}. This demonstrates that the action
  of the comprehension functor $c_p\colon\qB\to\qK$ on the hom-space from $a$ to
  $b$ is equivalent to the external action of $\Hom_{\qB}(a,b)$ on fibres of
  $p\colon E\tfib B$.
\end{thm}
\addtocounter{thm}{-1}
}

\begin{proof}
  We start by replacing hom-spaces by right hom-spaces in~\eqref{eq:comprehension-on-homs}, to
  which end observe that we have pair of commutative squares
  \begin{equation*}
    \xymatrix@R=2em@C=2.5em{
      {\Hom^r_{\qB}(a,b)}\ar[r]^-{\tilde{m}_{a,b}}\ar@{^(->}[d]_{\simeq} &
      {\Fun_{g_*\eK}(E_a,E_b)}\ar@{^(->}[r]^{\simeq}\ar@{=}[d] &
      {\Hom^r_{\qK}(E_a,E_b)}\ar@{^(->}[d]^{\simeq} \\
      {\Hom_{\qB}(a,b)}\ar[r]_-{\tilde{m}_{a,b}} &
      {\Fun_{g_*\eK}(E_a,E_b)}\ar@{^(->}[r]_{\simeq} &
      {\Hom_{\qK}(E_a,E_b)}
    }\mkern10mu
    \xymatrix@R=2em@C=3em{
      {\Hom^r_{\qB}(a,b)}\ar[r]^-{\Hom^r(c_p)}\ar@{^(->}[d]_{\simeq} &
      {\Hom^r_{\qK}(E_a,E_b)}\ar@{^(->}[d]^{\simeq} \\
      {\Hom_{\qB}(a,b)}\ar[r]_-{\Hom(c_p)} &
      {\Hom_{\qK}(E_a,E_b)}
    }
  \end{equation*}
  whose lower-horizontal maps are the two legs of the triangle in~\eqref{eq:comprehension-on-homs}.
  The vertical inclusions in these squares are trivial cofibrations of Kan
  complexes, so it follows that the triangle in~\eqref{eq:comprehension-on-homs} commutes up to
  isomorphism if and only if the following triangle does:
  \begin{equation}\label{eq:19}
    \xymatrix@=2.5em{
      {\Hom^r_{\qB}(a,b)}\ar[r]^-{\tilde{m}_{a,b}}\ar[dr]_{\Hom^r(c_p)} &
      {\Fun_{g_*\eK}(E_a,E_b)}\ar@{^(->}[d]^{\simeq}\ar@{}[dl]|(0.35){\cong} \\ 
      & {\Hom^r_{\qK}(E_a,E_b)}
    }
  \end{equation}
  Consequently, we shall consider what happens when we transpose the two legs of
  this latter triangle under the composite adjunction $\gC \circ \Sigma^r \dashv \Hom^r \circ N$ \eqref{eq:composite-quillen-adj} considered in the proof of Proposition \ref{prop:fun-to-r-hom}.
  
  To compute the transpose of the diagonal map in~\eqref{eq:19} we consider the
  following diagram
  \begin{equation}\label{eq:16}
    \xymatrix@R=2.5em@C=5.5em{
      {\gC[\Sigma^r\Hom^r_{\qB}(a,b)]}
      \ar@{^(->}[d]\ar[r]^-{\gC[h^r_{a,b}]} &
      {\gC\qB} \ar@{^(->}[d]\ar[r]^-{E_\bullet} & {\eK} \\
      {\gC[\Sigma^r\Hom^r_{\qB}(a,b)\join\Del^0]}
      \ar[r]_-{\gC[h^r_{a,b}\join\Del^0]} & {\gC[\qB\join\Del^0]}
      \ar@/^2.5ex/[]!R(0.5);[ur]^(0.45){\ell^E}_(0.45){}="one"
      \ar@/_2.5ex/[]!R(0.5);[ur]_-{\ell^B}^-{}="two" &
      \ar@{}"one";"two"\ar@{=>} ?(0.3);?(0.7) ^{p}
    }
  \end{equation}
  in which the simplicial natural transformation on the right is the cocartesian
  cocone used to define the comprehension functor $c_p\colon \qB\to\qK$ in Definition~\ref{defn:basic-comprehension}. The functor denoted $h^r_{a,b}\colon
  \Sigma^r\Hom^r_{\qB}(a,b) \to\qB$ is the counit of the adjunction
  $\Sigma^r\dashv\Hom^r$ of Proposition~\ref{prop:hom-space-comparison} at the object $\langle {a,b}
  \rangle\colon\Del^0+\Del^0\to\qB$ and the commutativity of the left hand
  square in~\eqref{eq:16} is simply a matter of the naturality of the family of
  inclusions $\gC[-]\inc\gC[-\join\Del^0]$. The simplicial functor
  $E_\bullet\colon \gC\qB\to\eK$ is the adjoint transpose of the comprehension
  functor $c_p\colon\qB\to\qK$ under the adjunction $\gC\dashv\hN$; so it is
  immediate that the transpose of the upper horizontal of~\eqref{eq:16} under
  the adjunction $\gC\circ\Sigma^r\dashv \Hom^r\circ\hN$ is simply the hom-space
  action $\Hom^r(c_p)\colon\Hom^r_{\qB}(a,b) \to \Hom^r_{\qK}(E_a,E_b)$, this being
  the diagonal in~\eqref{eq:comprehension-on-homs}. Furthermore, Proposition \ref{prop:cocart-cone-from-nat} tells us that the whiskered transformation along the  bottom of the diagram is a cocartesian cocone of shape
  $\Sigma^r\Hom^r_{\qB}(a,b)$ with nadir $p\colon E\tfib B$.

  On the other hand, to compute the transpose of the upper composite
  of~\eqref{eq:19} we consider the following diagram
  \begin{equation}\label{eq:15}
    \xymatrix@R=2.5em@C=5em{
      {\gC[\Sigma^r\Hom^r_{\qB}(a,b)]}
      \ar@{^(->}[d]\ar[r]^-{w^{\Hom^r_{\qB}(a,b)}} &
      {\cattwo[\Hom^r_{\qB}(a,b)]}
      \ar@{^(->}[d]\ar[r]^-{\hat{m}_{a,b}} & {\eK} \\
      {\gC[\Sigma^r\Hom^r_{\qB}(a,b)\join\Del^0]}\ar[r]_-{v^{\Hom^r_{\qB}(a,b)}} &
      {{\catthree[\Hom^r_{\qB}(a,b)]}}
      \ar@/^2.5ex/[]!R(0.5);[ur]^(0.45){\bar\ell^E}_(0.45){}="one"
      \ar@/_2.5ex/[]!R(0.5);[ur]_-{\bar\ell^B}^-{}="two" &
      \ar@{}"one";"two"\ar@{=>} ?(0.3);?(0.7) ^{~\bar{p}}
    }
  \end{equation}
  in which the triangle on the right displays the cocartesian natural
  transformation discussed in Remark~\ref{rmk:r-hom-instead} and the commutative square on
  the left is an instance of that discussed in Lemma~\ref{lem:comparison-extension}. Now the
  duality established in Proposition~\ref{prop:fun-to-r-hom} demonstrates that the natural
  inclusion $\Fun_{\eK}(A,B)\inc\Hom_{\qK}(A,B)$ is mate to $w^U\colon
  \gC[\Sigma^r U]\to\cattwo[U]$ under the adjunctions $\gC\circ\Sigma^r\dashv
  \Hom^r\circ\hN$ and $\cattwo[-]\dashv\Fun$. It follows, by a standard mating
  argument, that the upper horizontal in~\eqref{eq:15} is the adjoint transpose
  of the composite $\Hom^r_{\qB}(a,b)\xrightarrow{\tilde{m}_{a,b}} \Fun_{g_*\eK}(
  E_a, E_b) \xrightincarrow{\mkern10mu\simeq\mkern10mu} \Hom^r_{\qK}(E_a,E_b)$
  under the adjunction $\gC\circ\Sigma^r\dashv \Hom^r\circ\hN$. Furthermore,
  applying Proposition~\ref{prop:cocart-cone-from-nat} we see that the whiskered transformation along the
  bottom of the~\eqref{eq:15} is a cocartesian cocone of shape
  $\Sigma^r\Hom^r_{\qB}(a,b)$ with nadir $p\colon E\tfib B$.

  Thus far we have shown that duals of the functors in~\eqref{eq:comprehension-on-homs} may be
  presented in terms of certain cocartesian cocones with nadir $p\colon E\tfib
  B$. To relate these it will suffice, by Observation~\ref{obs:unique-comprehension}, to show that
  the cocones over which they live are the same, those being the composites
  depicted on the outside of the following diagram:
  \begin{equation*}
    \xymatrix@=2em{
      {\gC[\Sigma^r\Hom^r_{\qB}(a,b)\join\Del^0]}
      \ar[rrrr]^{v^{\Hom^r_{\qB}(a,b)}}
      \ar[dr]_{t^{\Sigma^r\Hom^r_{\qB}(a,b)}}
      \ar[ddd]_{\gC[h^r_{a,b}\join\Del^0]} && \ar@{}[d]|(0.4){\text{(A)}}&&
      {\catthree[\Hom^r_{\qB}(a,b)]}
      \ar@{^(->}[d]\ar[dl]^(0.4){s^{\Hom^r_{\qB}(a,b)}} \\
      & {\cattwo[\Sigma^r\Hom^r_{\qB}(a,b)]}
      \ar[d]_{\cattwo[h^r_{a,b}]}\ar@{}[dl]|(0.6){\text{(B)}} &
      \ar@{}[d]|(0.4){\text{(C)}} &
      {\cattwo[\Sigma\Hom^r_{\qB}(a,b)]}
      \ar[ll]_{\cattwo[u^{\Hom^r_{\qB}(a,b)}]}\ar@{^(->}[d]
      \ar@{}[r]|-{\text{(D)}} &
      {\catthree[\Hom_{\qB}(a,b)]}\ar@{}[ddl]+<3em,0em>|(0.6){\text{(F)}}
      \ar[dd]^{\bar\ell^B}\ar[dl]^(0.4){s^{\Hom_{\qB}(a,b)}} \\
      & {\cattwo[\qB]}\ar[drrr]_(0.4){k^B}\ar@{}[d]+<1em,0em>|{\text{(E)}} &&
      {\cattwo[\Sigma\Hom_{\qB}(a,b)]}
      \ar[ll]_-{\cattwo[h_{a,b}]} & \\
      {\gC[\qB\join\Del^0]}
      \ar[rrrr]_{\ell^B}\ar[ur]_(0.6){t^{\qB}} &&&& {\eK}
    }
  \end{equation*}
  Here the square marked (A) is an instance of
  Lemma~\ref{lem:commuting-comparisons} and those marked (B) and (D) are
  naturality squares. The maps $h^r_{a,b}\colon\Sigma^r\Hom^r_{\qB}(a,b)\to\qB$
  and $h_{a,b}\colon\Sigma\Hom_{\qB}(a,b)\to\qB$ are instances of the counits of
  the adjunctions $\Sigma^r\dashv\Hom^r$ and $\Sigma\dashv\Hom$ respectively,
  and the square (C) commutes because $\Hom^r_{\qB}(a,b)\trvcof\Hom_{\qB}(a,b)$
  and $u^U\colon\Sigma U\to\Sigma^r U$ are mates under those adjunctions. The
  map $k^B\colon\cattwo[\qB]\to\eK$ is the instance of the counit of
  $\cattwo[-]\dashv\Fun$ and so the commutativity of the triangle (E) simply
  expresses the definition of the cocone $\ell^B$ given there. Finally, to check
  the commutativity of (F) it is clear that it is sufficient to check the
  agreement of the actions of its legs on the function complex
  $\Fun_{\catthree[\Hom_{\qB}(a,b)]}(-,\top)$, this being a matter of routine
  verification directly from the definitions of the various maps given in
  Definition~\ref{defn:last-map} and~\ref{defn:cattwo-fun-adjunction},
  Lemma~\ref{lem:cocart-extending-external}, and
  Proposition~\ref{prop:hom-space-comparison}.

  Consequently, we have shown that the cocartesian cocones in~\eqref{eq:16}
  and~\eqref{eq:15} lie over the same underlying lax cocone. Indeed, their
  restrictions along the inclusion $\langle -,+ \rangle\colon\Del^0+\Del^0\inc
  \Sigma^r\Hom^r_{\qB}(a,b)$ are identical, since they are constructed using the
  same pullbacks in $\eK$. So we may apply Corollary \ref{cor:unique-cocart-lifts} to show that
  the transposes $e,\bar{e}\colon\Sigma^r\Hom^r_{\qB}(a,b)\to\qK$ of their bases
  are isomorphic as $1$-arrows between the objects $\langle -,+
  \rangle\colon\Del^0+\Del^0\inc \Sigma^r\Hom^r_{\qB}(a,b)$ and $\langle
  {E_a,E_b}\rangle \colon \Del^0+\Del^0 \to\qK$ in the simplicial slice
  $\prescript{\Del^0+\Del^0/}{}{\SSet}$. By construction the transposes of $e$
  and $\bar{e}$ under the adjunction $\Sigma^r\dashv\Hom^r$ of
  Proposition~\ref{prop:hom-space-comparison} are the legs of~\eqref{eq:19}, so all that remains is
  to show that we may also transpose the isomorphism between them to give the
  isomorphism depicted in that triangle. This, however, follows from the easily
  verified fact that the adjunction $\Sigma^r\dashv\Hom^r$ is simplicially enriched. Since the codomain of \eqref{eq:19} is a Kan complex, the transposed 1-simplex inhabiting this triangle is automatically an isomorphism, which is what we wanted to show.
\end{proof}

As an application of Theorem \ref{thm:comprehension-on-homs} we prove that the Yoneda embedding of Definition~\ref{defn:yoneda-embedding} is fully faithful:

\begin{thm}[Yoneda embeddings are fully faithful]\label{thm:yoneda-ff} $\quad$
\begin{enumerate}[label=(\roman*)]
  \item\label{itm:yon-ff} The Yoneda embedding 
    \begin{equation*}
    \yoneda\colon \Fun_{\eK}(1,B) \longrightarrow \Cart(\qK)_{/B} \subset \qK_{/B}
  \end{equation*}
  is a fully  faithful functor of quasi-categories. 
  \item\label{itm:cauchy-embedding} Every quasi-category is
  equivalent to the homotopy coherent nerve of some Kan complex enriched
  category.
  \end{enumerate} 
\end{thm}

\begin{proof}
  Given two objects $a,b\colon 1\to B$ of the underlying
  quasi-category $\qB$, to prove \ref{itm:yon-ff} our task is to show that the action of the Yoneda
  functor
  \begin{equation*}
    \qB\cong\Fun_{\eK}(1,B)\xrightarrow{\mkern10mu -\times B\mkern10mu}
    \Fun_{\eK_{/B}}\left(
      \vcenter{\xymatrix@1@R=2.5ex{B\ar@{->>}[d]^-{\id_B}\\ B}},
      \vcenter{\xymatrix@1@R=2.5ex{B \times B \ar@{->>}[d]^-{\pi_0} \\ B}}
    \right) \xrightarrow{\mkern10mu c_{(p_1,p_0)}\mkern10mu} \qK_{/B}
  \end{equation*}
  on the hom-space from $a$ to $b$ is an equivalence of Kan-complexes. Of
  course, the identity $\id_B\colon B\to B$ is the terminal object in $\eK_{/B}$
  and the comma object associated with the maps $B\times a$ and $B\times b$ in
  there is given by the projection $\pi_0\colon B\times a\comma b \tfib B$. So
  we may apply Theorem~\ref{thm:comprehension-on-homs} to show that the action of the Yoneda
  functor on the hom-space from $a$ to $b$ is equivalent to the composite:
  \begin{equation}\label{eq:yoneda-equivalence-inverse}
    \Fun_{\eK}(1,a\comma b)\xrightarrow{\mkern10mu -\times B\comma a\mkern10mu}
    \Fun_{\eK_{/B}}\left( 
      \vcenter{\xymatrix@1@R=2.5ex{B\comma a \ar@{->>}[d]\\ B}},
      \vcenter{\xymatrix@1@R=2.5ex{a\comma b\times B\comma a \ar@{->>}[d] \\ B}}
    \right)\xrightarrow{\mkern10mu\tilde{m}_{a,b}\mkern10mu}
    \Fun_{\eK_{/B}}\left( 
      \vcenter{\xymatrix@1@R=2.5ex{B\comma a \ar@{->>}[d]\\ B}},
      \vcenter{\xymatrix@1@R=2.5ex{B\comma b \ar@{->>}[d] \\ B}}
    \right)
  \end{equation}
  Since the representable cartesian fibrations $B \comma b \tfib B$ are groupoidal objects in $\eK_{/B}$, the codomain of \eqref{eq:yoneda-equivalence-inverse} is already a Kan complex, and we can omit the notation for the groupoidal core.
  
  The Yoneda lemma, as proven in Theorem \refIV{thm:yoneda}, tells us that restriction along the natural map $\id_a \colon 1 \to B \comma a$ from $a \colon 1 \to B$ to $B \comma a \tfib B$ induces an equivalence of function complexes
  \begin{equation}\label{eq:yoneda-equivalence}     \Fun_{\eK_{/B}}\left( 
      \vcenter{\xymatrix@1@R=2.5ex{B\comma a \ar@{->>}[d]\\ B}},
      \vcenter{\xymatrix@1@R=2.5ex{B\comma b \ar@{->>}[d] \\ B}}
    \right) \xrightarrow{~\simeq~}    \Fun_{\eK_{/B}}\left( 
      \vcenter{\xymatrix@1@R=2.5ex{1 \ar@{->}[d]_a\\ B}},
      \vcenter{\xymatrix@1@R=2.5ex{B\comma b \ar@{->>}[d] \\ B}}
    \right)
    \end{equation}
    From the pullback that defines the function complexes in $\eK_{/B}$
    \[ 
    \xymatrix{ \Fun_{\eK_{/B}}(a \colon 1 \to B, B \comma b \tfib B) \pbexcursion \ar[r] \ar@{->>}[d] & \Fun_{\eK}(1,B \comma b) \ar@{->>}[d] \\ 1 \cong \Fun_{\eK}(1,B) \ar[r]_-{\Fun_{\eK}(1,a)} & \Fun_{\eK}(1,B)}
    \]
    we see that $\Fun_{\eK_{/B}}(a \colon 1 \to B, B \comma b \tfib B) \cong \Fun_{\eK}(1,a\comma b)$, which is the definition of $\Hom_{\qB}(a,b)$.  The proof of Theorem  \refIV{thm:yoneda} demonstrates that the map \eqref{eq:yoneda-equivalence} is an equivalence by constructing an explicit equivalence inverse: namely, the functor \eqref{eq:yoneda-equivalence-inverse}. This proves that the Yoneda embedding is fully faithful.
  
  Now to deduce \ref{itm:cauchy-embedding} from \ref{itm:yon-ff}, we apply Lemma~\ref{lem:equiv-of-qcats} to conclude that the Yoneda embedding
  restricts in it codomain to provide an equivalence between a quasi-category $\qB$ and the
  homotopy coherent nerve of the full simplicial subcategory of $\qCat_{/\qB}$
  spanned by the representable cartesian fibrations $\qB\comma b\tfib \qB$ indexed by its objects $b \colon 1 \to \qB$.
\end{proof}


%% file: all.bbl
\begin{thebibliography}{RV17b}

\bibitem[B73]{Brown:1973zl}
K.~S. Brown.
\newblock Abstract homotopy theory and generalized sheaf cohomology.
\newblock {\em Transactions of the American Mathematical Society},
  186:419--458, 1973.

\bibitem[B07]{Bergner:2007fk}
J.~E. Bergner.
\newblock A model category structure on the category of simplicial categories.
\newblock {\em Transactions of the American Mathematical Society},
  359:2043--2058, 2007.

\bibitem[DS11a]{Dugger:2011ro}
D.~Dugger and D.~I. Spivak.
\newblock Rigidification of quasi-categories.
\newblock {\em Algebr. Geom. Topol.}, 11(1):225--261, 2011.

\bibitem[DS11b]{DuggerSpivak:2011ms}
D.~Dugger and D.~I. Spivak.
\newblock Mapping spaces in quasi-categories.
\newblock {\em Algebr. Geom. Topol.}, 11(1):263--325, 2011.

\bibitem[J08]{Joyal:2008tq}
A.~Joyal.
\newblock {\em The theory of quasi-categories and its applications}.
\newblock Quadern 45 vol II. Centre de Recerca Matem\`{a}tica Barcelona,
  \url{http://mat.uab.cat/~kock/crm/hocat/advanced-course/Quadern45-2.pdf},
  2008.

\bibitem[L09]{Lurie:2009fk}
J.~Lurie.
\newblock {\em {Higher Topos Theory}}, volume 170 of {\em Annals of
  Mathematical Studies}.
\newblock Princeton University Press, Princeton, New Jersey, 2009.

\bibitem[L17]{Lurie:2012uq}
J.~Lurie.
\newblock {Higher Algebra}.
\newblock \url{http://www.math.harvard.edu/~lurie/papers/HA.pdf}, September
  2017.

\bibitem[R73]{reedy1973htm}
C.~Reedy.
\newblock Homotopy theory of model categories.
\newblock {\em preprint}, 1973.

\bibitem[R01]{Rezk:2001sf}
C.~Rezk.
\newblock A model for the homotopy theory of homotopy theory.
\newblock {\em Transactions of the American Mathematical Society},
  353(3):973--1007, 2001.

\bibitem[R10]{Rezk:2010fk}
C.~Rezk.
\newblock A cartesian presentation of weak $n$-categories.
\newblock {\em Geometry and Topology}, 2010.

\bibitem[R11]{Riehl:2011ot}
E.~Riehl.
\newblock On the structure of simplicial categories associated to
  quasi-categories.
\newblock {\em Math. Proc. Cambridge Philos. Soc.}, 150(3):489--504, 2011.

\bibitem[R14]{Riehl:2014kx}
E.~Riehl.
\newblock {\em Categorical homotopy theory}, volume~24 of {\em New Mathematical
  Monographs}.
\newblock Cambridge University Press, 2014.

\bibitem[RV-0]{RiehlVerity:2013kx}
E.~Riehl and D.~Verity.
\newblock The theory and practice of {Reedy} categories.
\newblock {\em Theory and Applications of Categories}, 29(9):256--301, 2014.

\bibitem[RV-I]{RiehlVerity:2012tt}
E.~Riehl and D.~Verity.
\newblock The 2-category theory of quasi-categories.
\newblock {\em Adv. Math.}, 280:549--642, 2015.

\bibitem[RV-III]{RiehlVerity:2013cp}
E.~Riehl and D.~Verity.
\newblock Completeness results for quasi-categories of algebras, homotopy
  limits, and related general constructions.
\newblock {\em Homol.~Homotopy Appl.}, 17(1):1--33, 2015.

\bibitem[RV-II]{RiehlVerity:2012hc}
E.~Riehl and D.~Verity.
\newblock Homotopy coherent adjunctions and the formal theory of monads.
\newblock {\em Adv. Math.}, 286:802--888, 2016.

\bibitem[RV-IV]{RiehlVerity:2015fy}
E.~Riehl and D.~Verity.
\newblock Fibrations and {Y}oneda's lemma in an $\infty$-cosmos.
\newblock {\em J.~Pure Appl.~Algebra}, 221(3):499--564, 2017.

\bibitem[RV-V]{RiehlVerity:2015ke}
E.~Riehl and D.~Verity.
\newblock Kan extensions and the calculus of modules for $\infty$-categories.
\newblock {\em Algebr.~Geom.~Topol.}, 17-1:189--271, 2017.

\bibitem[RV-VII]{RiehlVerity:2017ts}
E.~Riehl and D.~Verity.
\newblock The calculus of two-sided fibrations and modules.
\newblock {I}n preparation, 2018.

\bibitem[S74]{Street:1974:FibYoneda}
R.~Street.
\newblock Fibrations and {Y}oneda's lemma in a {$2$}-category.
\newblock In {\em Category {S}eminar ({P}roc. {S}em., {S}ydney, 1972/1973)},
  pages 104--133. Lecture Notes in Math., Vol. 420. Springer, Berlin, 1974.

\bibitem[W07]{Weber:2007ys}
M.~Weber.
\newblock Yoneda {S}tructures from 2-toposes.
\newblock {\em Applied Categorical Structures}, 15(3):259--323, 2007.

\end{thebibliography}
